\documentclass[11pt, reqno]{amsart}

\usepackage{etoolbox}

\newbool{ForSubmission}
\booltrue{ForSubmission}

\newbool{InlineBib}
\booltrue{InlineBib}

\usepackage{amscd,amsmath}
\usepackage{mathrsfs}
\usepackage{amsfonts}
\usepackage{amssymb, bm, xspace}
\usepackage{enumerate}
\usepackage{setspace}

\usepackage{datetime}

\usepackage{enumitem}

\usepackage{xcolor}

\usepackage[pagebackref=true, colorlinks=true, citecolor=blue]{hyperref}

\usepackage[margin=1.2in]{geometry}

\usepackage{mathtools}

\usepackage{cleveref}

\newcommand{\Comment}[1]{{\color{blue}#1}}
\newcommand{\OptionalDetails}[1]{
    \ifbool{ForSubmission}
        {
        }
        {\begin{quote}\Comment{\footnotesize
        \medskip

        \noindent#1}
        \end{quote}
        }
    }

\newbool{HaveBBM}
\booltrue{HaveBBM}

\ifbool{HaveBBM}{
	\usepackage{bbm}
    }
    {
    }

\newbool{arXivFormat}
\boolfalse{arXivFormat}
    
\newcommand{\IfarXivElse}[2]{
    \ifbool{arXivFormat}
        {#1}{#2}
    }

\renewcommand{\mathbf}[1]{\bm{#1} \textbf{ *** Use bm instead of mathbf ***}}

\newcommand{\eqn}{\begin{eqnarray}}
\newcommand{\een}{\end{eqnarray}}

\newtheorem{theorem}{Theorem}[section]
\newtheorem{prop}[theorem]{Proposition}
\newtheorem{lemma}[theorem]{Lemma}
\newtheorem{cor}[theorem]{Corollary}

\newtheorem{definition}[theorem]{Definition}
\newtheorem{remark}[theorem]{Remark}

\numberwithin{equation}{section}

\newcommand{\abs}[1]{\left\vert#1\right\vert}

\newcommand{\BB}[1]{\ensuremath{\mathbb{#1}}}

\newcommand{\R}{\ensuremath{\BB{R}}}
\newcommand{\N}{\ensuremath{\BB{N}}}

\newcommand{\iny}{\ensuremath{\infty}}
\newcommand{\grad}{\ensuremath{\nabla}}

\newcommand{\CharFunc}{
    \ifbool{HaveBBM}{
        \ensuremath{\mathbbm{1}}
        }
        {
        \ensuremath{\bm{1}}
        }
    }

\DeclareMathOperator{\dv}{div} %
\DeclareMathOperator{\curl}{curl} %
\DeclareMathOperator{\tr}{tr} %
\DeclareMathOperator{\supp}{supp} %
\DeclareMathOperator{\PV}{p.v.} %
\newcommand{\starp}{{\, \widetilde{*} \,}}
\newcommand{\starpdot}{{\, \widetilde{*} \cdot \,}}
\newcommand{\stardot}{\mathop{* \cdot}}
\newcommand{\prt}{\ensuremath{\partial}}
\newcommand{\brac}[1]{\ensuremath{\left[ #1 \right]}}
\newcommand{\bigbrac}[1]{\ensuremath{\Big[ #1 \Big]}}
\newcommand{\pr}[1]{\ensuremath{\left( #1 \right)}}
\newcommand{\bigpr}[1]{\ensuremath{\Big( #1 \Big)}}
\DeclarePairedDelimiter{\set}{\{}{\}}

\newcommand{\bigset}[1]{\ensuremath{\left\{ #1 \right\}}}

\newcommand{\norm}[1]{\ensuremath{\left\Vert #1 \right\Vert}}

\newcommand{\smallnorm}[1]{\ensuremath{\Vert #1 \Vert}}

\newcommand{\n}{\bm{n}}

\newcommand{\wh}{\widehat}

\newcommand{\FTF}
    {\Cal{F}}

\newcommand{\FTR}
    {\Cal{F}^{-1}}

\renewcommand{\epsilon}{\varepsilon}
\newcommand{\eps}{\ensuremath{\varepsilon}}

\newcommand{\Cal}[1]{\ensuremath{\mathcal{#1}}}
\newcommand{\al}{\ensuremath{\alpha}}
\newcommand{\la}{\ensuremath{\lambda}}

\newcommand{\diff}[2]{\frac{ d#1}{d#2}}

\newcommand{\ol}{\overline}
\newcommand{\smallabs}[1]{\ensuremath{\vert #1 \vert}}

\renewcommand{\matrix}[2]{\begin{pmatrix} #1 \\ #2 \end{pmatrix}}

\newcommand{\matrixone}[1]{\begin{pmatrix} #1 \end{pmatrix}}

\newcommand{\intg}{\int_0^r \rho g(\rho) \, d \rho}

\newcommand{\Holder}
    {H\"{o}lder\xspace}

\newcommand{\Ignore}[1]{}

\definecolor{Correction}{named}{red}

\newcommand{\spacer}{\vspace{2mm}}
\newcommand{\halfspacer}{\vspace{1mm}}

\newcommand{\e}{\bm{e}}

\newcommand{\Ndelta}[2]{\Cal{N}_{#2}(#1)}

\crefname{cor}{Corollary}{Corollaries} 
									   
\crefname{lemma}{Lemma}{Lemmas}	       

\crefname{section}{Section}{Sections}
\Crefname{section}{Section}{Sections}

\crefname{appendix}{Appendix}{Appendices}
\Crefname{appendix}{Appendix}{Appendices}

\crefname{theorem}{Theorem}{Theorems}
\Crefname{theorem}{Theorem}{Theorems}

\crefname{prop}{Proposition}{Propositions}
\Crefname{prop}{Proposition}{Propositions}

\crefname{conj}{Conjecture}{Conjectures}
\Crefname{conj}{Conjecture}{Conjectures}

\crefname{definition}{Definition}{Definitions}
\Crefname{definition}{Definition}{Definitions}

\crefname{remark}{Remark}{Remarks}
\Crefname{remark}{Remark}{Remarks}

\crefname{assumption}{Assumption}{Assumptions}
\Crefname{assumption}{Assumption}{Assumptions}

\crefformat{equation}{(#2#1#3)}
\crefrangeformat{equation}{(#3#1#4) through (#5#2#6)}
\crefmultiformat{equation}
    {(#2#1#3)}%
    { and~(#2#1#3)}
    {, (#2#1#3)}
    { and~(#2#1#3)}

\begin{document}

\newdateformat{mydate}{\THEDAY~\monthname~\THEYEAR}

\title
    [The vortex patches of Serfati]
    {The vortex patches of Serfati}

\author{Hantaek Bae}
\address{Department of Mathematical Sciences, Ulsan National Institute of Science and Technology (UNIST), Korea}
\email{hantaek@unist.ac.kr}

\author{James P Kelliher}
\address{Department of Mathematics, University of California, Riverside, USA}
\email{kelliher@math.ucr.edu}

\date{(compiled on {\dayofweekname{\day}{\month}{\year} \mydate\today)}}

\begin{abstract}
In 1993, two proofs of the persistence of regularity of the boundary of a classical vortex patch for the 2D Euler equations were published, one by Chemin in \cite{Chemin1993Persistance} (announced in 1991 in \cite{Chemin1991VortexPatch}) the other by Bertozzi and Constantin in \cite{ConstantinBertozzi1993}. Chemin, in fact, proved a more general result, extending it further in his 1995 text \cite{C1998} showing, roughly, that vorticity initially having discontinuities only in directions normal to a family of vector fields that together foliate the plane continue to be so characterized by the time-evolved vector fields. A different, four-page ``elementary'' proof of Chemin's 1993 result was published in 1994 by Ph. Serfati \cite{SerfatiVortexPatch1994}, who also gave a fuller characterization of the velocity gradient's regularity.
We give a detailed version of Serfati's proof along with an extension of it to a family of vector fields that reproduces the 1995 result of Chemin.
\end{abstract}

\maketitle

\ifbool{ForSubmission}{
    }
    {
    \begin{center}
    {
        \footnotesize
        \Comment{
        \framebox[0.69\textwidth]{
            \parbox[center][3.5em][c]{0.66\textwidth}
            {
            Includes some proofs and additional details not intended
            for submission.
            These details appear in blue in small font.           
        }
        }
        }
    }
    \end{center}

    \renewcommand\contentsname{}  
    \setcounter{tocdepth}{1}      
    {\small
        \tableofcontents
    }
    }

\noindent 
In the late 1980s into the early 1990s there was a great deal of interest in determining whether a vortex patch having a smooth boundary at time zero continues to have a smooth boundary for all time as it evolves under the 2D Euler equations.
Majda had suggested in \cite{M1986} the possibility of singularities forming in finite time.
Existing numerical evidence showed that the boundary typically deforms dramatically over time, and hinted at the development of such singularities. The announcement in 1991 \cite{Chemin1991VortexPatch} and the two 1993 papers \cite{Chemin1993Persistance, ConstantinBertozzi1993} came, then, as a surprise to many, showing as they did that the boundary remains regular for all time.



In 1994, another proof of the persistence of regularity of a vortex patch was published by Ph. Serfati in the four-page paper, \cite{SerfatiVortexPatch1994}. Like Chemin's \cite{Chemin1991VortexPatch}, it was published in a journal devoted primarily to short announcements, but unlike \cite{Chemin1991VortexPatch}, it was never followed by a full-length publication.
In this highly condensed form much is omitted that would aid the reader in understanding,
and much is left to the reader to decipher.
It's opaqueness has kept \cite{SerfatiVortexPatch1994} from having an influence on subsequent developments in two-dimensional fluid mechanics. One of our purposes here is to present our interpretation of Serfati's argument in a detailed enough form to make it accessible, for it is not only an elegant approach to the vortex patch problem, but some of its ideas, buried for two decades, have potential applications to problems of current interest.

Chemin proved a more general result in \cite{C1998} of which the persistence of regularity of the boundary of a vortex patch was a special case. He employed a family of vector fields and showed, speaking roughly, that if the initial vorticity is $C^\al$ in the direction of this family for some $\al > 0$ then this property holds true for all time. A second purpose of this work is to show that if one extends Serfati's hypotheses on the initial data by employing a family of vector fields then one obtains the same result as \cite{C1998}. We also reinterpret this result as showing that if the initial velocity is $C^{1 + \al}$ in the direction of the family of vector fields then this property holds true for all time.

Finally, Serfati also showed that the gradient of the velocity is $C^\al$ after being corrected by a $C^\al$ multiple of the vorticity. We give a different proof of this result (which was one sentence in \cite{SerfatiVortexPatch1994}) and show that it yields an improved estimate on the local propagation of \Holder regularity of the velocity.

\section{Introduction and statements of results}\label{S:Introduction}

\noindent

\noindent The Euler equations (without forcing) in velocity form can be written,
\begin{align}\label{e:Eu}
    \left\{
    \begin{array}{rl}
        \prt_t u + (u \cdot \nabla) u + \nabla p &= 0, \\
        \dv u &= 0,
    \end{array}
    \right.
\end{align}
where $u$ is the velocity field and $p$ is the pressure. The operator $u \cdot \grad = u^i \prt_i$, where we follow the usual convention that repeated indices are summed over. These equations model the flow of an incompressible inviscid fluid.

By introducing the 2D vorticity,
\begin{align*}
    \omega=\prt_1 u^2-\prt_2 u^1 ,
\end{align*}
we obtain the vorticity formulation,
\begin{align}\label{e:Eomega}
    \left\{
    \begin{array}{rl}
        \prt_t \omega + u \cdot \nabla \omega &= 0, \\
        u &= K * \omega.
    \end{array}
    \right.
\end{align}
Here,
\begin{align}\label{e:BS}
    K(x)
        = \frac{1}{2 \pi} \frac{x^\perp}{\abs{x}^2}, \quad
        x^\perp := (-x_2, x_1),
\end{align}
is the Biot-Savart kernel, which can also be written 
\begin{align*}
    K
        = \grad^\perp \mathcal{F},
        \quad
        \Cal{F}(x)
            = \frac{1}{2\pi}\log \abs{x},
            \quad
            \nabla^\perp := (-\prt_2, \prt_1),
\end{align*}
$\Cal{F}$ being the fundamental solution to the Laplacian.

Let $\eta(t, x)$ be the flow map associated to the velocity field $u$, so that
\begin{align}\label{e:etaDef}
    \prt_t \eta(t, x)
        = u \pr{t, \eta(t, x)},
            \quad
            \eta(0, x) = x.
\end{align}
Then \cref{e:Eomega} tells us that the vorticity is transported by the flow map, so that
\begin{align}\label{e:omegaFlow}
    \omega(t, x) = \omega_0(\eta^{-1}(t, x)) 
\end{align}
is the vorticity of the solution to the Euler equations at time $t$, where $\omega_0$ is the initial vorticity.

All this presupposes that sufficiently regular solutions exist and are unique. In fact, it all can be made sense of for initial vorticity in $L^1 \cap L^\iny$, in which case the vorticity remains in $L^1 \cap L^\iny$, as first shown by Yudovich in \cite{Y1963}. One must, however, use a weak formulation of \cref{e:Eu} or \cref{e:Eomega}, though $\cref{e:Eomega}_2$ and \cref{e:omegaFlow} continue to hold.

If the vorticity is initially the characteristic function of a bounded domain, it will remain so for all time as the Euler solution evolves, since $\eta(t, \cdot)$ is a diffeomorphism. A (classical) vortex patch is such a bounded domain. So if
\begin{align}\label{e:InitialVorticity}
    \omega_0
        = \CharFunc_\Omega,
\end{align}
where $\Omega$ is a bounded domain, then by \cref{e:omegaFlow},
\begin{align*}
    \omega(t) = \CharFunc_{\Omega_t},
        \quad
        \Omega_t := \eta(t, \Omega).
\end{align*}
The bounded domain, $\Omega_t$, is the vortex patch at time $t$.

The regularity of the boundary of $\Omega$ will be specified using a parameter, $\al$.

\smallskip
\begin{center}
    \fbox{Throughout this paper we fix $\al \in (0, 1)$.}
\end{center}
\smallskip

We can now state the result of \cite{Chemin1993Persistance, ConstantinBertozzi1993} more precisely. 

\begin{theorem}\cite{Chemin1991VortexPatch, Chemin1993Persistance, ConstantinBertozzi1993}\label{T:VortexPatch}
    Let $\Omega$ be a bounded domain whose boundary is the image of a
    simple closed curve $\gamma_0 \in C^{1 + \alpha}(\mathbb{S}^1)$ and
    let $\omega_0$ be as in \cref{e:InitialVorticity}. There exists a unique solution
    $u$ to the 2D Euler equations, with 
    \begin{align*}
        \grad u(t) \in L^\iny(\R^2), \quad
        \gamma(t, \cdot)
            := \eta(t, \gamma_0(\cdot)) \in C^{1 + \alpha}(\mathbb{S}^1)
        \text{ for all } t \in \R.
    \end{align*}
\end{theorem}

In \cite{Chemin1993Persistance}, Chemin proves a more general result of which \cref{T:VortexPatch} is a corollary. We show in \cref{S:EquivInitConditions} that Serfati's result in \cite{Serfati1995A} is equivalent to that in \cite{Chemin1993Persistance}. To describe Serfati's result, we must first make some definitions. 
Let $\Sigma$ be a closed subset of $\R^2$ and $Y_0$ be a $C^\alpha$-vector field in $\R^2$. Let $\Ndelta{\Sigma}{\delta} := \set{x\in \R^2: d(x, \Sigma) < \delta}$. Serfati assumes the following initial conditions:
\begin{align}\label{e:SerfatiInitialData}
    \left\{
        \begin{array}{l}
            \halfspacer \displaystyle
            \omega_0
                = \omega_0^1 + \omega^2_0
                    \in (L^1 \cap L^\iny)(\R^2),
                    \quad \omega_0^1 = 0 \ \text{on } \Sigma^C,
                    \,
                    \omega^2_0 \in C^\alpha(\R^2),\\
           \halfspacer \displaystyle
           Y_0\in C^\alpha(\R^2), \quad \abs{Y_0} \ge c > 0 \ \text{on} \
                 \Ndelta{\Sigma}{\delta_0}, \, \delta_0 > 0, \\
            \displaystyle K * \dv(\omega_0 Y_0) \in C^\alpha(\R^2).
    \end{array}
    \right.
\end{align}
We streamline these conditions to
\begin{align}\label{e:Assumption}
    \left\{
        \begin{array}{l}
            \halfspacer \displaystyle
            \omega_0\in C^\alpha(\R^2 \setminus \Sigma)
                \cap (L^1 \cap L^\iny)(\R^2), \\
            \halfspacer \displaystyle
           Y_0\in C^\alpha(\R^2), \quad \abs{Y_0} \ge c > 0 \ \text{on} \
                 \Ndelta{\Sigma}{\delta_0}, \, \delta_0 > 0, \\
            \halfspacer \dv (\omega_0 Y_0) \in C^{\alpha-1}(\R^2), \\
            \dv Y_0 \in C^\al(\R^2).
    \end{array}
    \right.
\end{align}
(For the negative index \Holder space, $C^{\al - 1}(\R^2)$, see \cref{D:HolderSpaces}.)

We show in \cref{A:ProofsOfLemmas} that \cref{e:SerfatiInitialData}$_3$ is equivalent to \cref{e:Assumption}$_3$ (since $\omega_{0}Y_0 \in L^\iny$). Also, we added the condition in \cref{e:Assumption}$_4$, which is missing in \cite{SerfatiVortexPatch1994}, though present in \cite{Chemin1993Persistance, C1998}, as it is necessary in the proof of the convergence of the approximate solutions (see, however, \cref{R:WeakerAssumption}). Such convergence is not addressed by Serfati in \cite{SerfatiVortexPatch1994}. This is the only place this condition is required. (See also \cref{R:SerfatidivY}.)

We define the pushforward of $Y_0$ by
\begin{align}\label{e:Y}
    Y(t, \eta(t, x)) := (Y_0(x) \cdot \nabla) \eta(t, x).
\end{align}
This is just the Jacobian of the diffeomorphism, $\eta(t, \cdot)$, multiplied by $Y_0$. Equivalently,
\begin{align*}
    Y(t, x)
        = \eta(t)_* Y_0(t, x)
        := (Y_0(\eta^{-1}(t, x))
            \cdot \grad) \eta(t, \eta^{-1}(t, x)).
\end{align*}

We will make frequent use of constants of the form,
\begin{align}\label{e:Calpha}
    c_\al
        = C(\omega_0, Y_0) \al^{-1}, \quad
    C_\al
        = C(\omega_0, Y_0) \al^{-1} (1 - \al)^{-1},
\end{align}
where $C(\omega_0, Y_0)$ is a constant that depends upon only $\omega_0$ and $Y_0$. The values of these constants can vary from expression to expression and even between two occurrences within the same expression.


\begin{theorem}[Serfati \cite{SerfatiVortexPatch1994}]\label{T:Serfati}
Assume that $\omega_0$ is an initial vorticity for which there exists some vector field $Y_0$ so that \cref{e:Assumption} is satisfied and let $\omega$ in $L^\iny(\R; (L^1 \cap L^\iny)(\R^2))$ be the unique solution to the Euler equations. We have,
%
\begin{align}\label{e:MainBounds}
    \begin{split}
        \norm{\grad u(t)}_{L^\iny}
            \le c_\al e^{c_\al t}, \quad
        \norm{Y(t)}_{C^\alpha}
            \le C_\al e^{e^{c_\al t}}. \\
  \end{split}
\end{align}
Moreover, there exists a matrix $A(t)\in C^\alpha(\R^2) $ such that 
\begin{align}\label{e:AEq}
    \nabla u(t) - \omega(t) A(t)
        \in C^\alpha(\R^2) \quad \text{for all time.}
\end{align}
(An explicit form for the matrix $A$ is given in \cref{e:AExplicit}.)
\end{theorem}

\begin{remark}\label{R:SerfatidivY}
    The first part of \cref{T:Serfati} giving \cref{e:MainBounds} is the same result as
    that of Chemin in \cite{Chemin1993Persistance}. 
    In \cite{Chemin1993Persistance}, though, Chemin assumes $\dv Y_0 = 0$
    (dropping this restriction in \cite{C1998}). Serfati does
    not state this restriction on $\dv Y_0$; the present authors, however,
    were unable to determine
    from Serfati's proof whether or not he meant to do so. Given that he
    did not add the required condition $\dv Y_0 \in C^\al$, it seems likely
    that he did intend to do so. We will show in our version of the proof that, in any case,
    $\dv Y_0 = 0$ is not required to complete the proof.
\end{remark}

A classical vortex patch satisfies \cref{e:Assumption}, as we show in \cref{S:ProofOfVortexPatch}. We describe other examples satisfying \cref{e:Assumption} in \cref{S:Examples}.

A number of additional useful facts follow from the proof of \cref{T:Serfati} or are simple consequences of it. We summarize these facts in \cref{T:SerfatiAdditional,T:LocalPropagation}.

\begin{theorem}\label{T:SerfatiAdditional}
    Let $\omega_0$, $Y_0$ be as in \cref{T:Serfati}.
    Then
    \begin{align}
        \norm{\dv Y(t)}_{C^\al}
            &\le \norm{\dv Y_0}_{C^\al} e^{e^{c_\al t}},
                \label{e:YdivBound}\\
    	\norm{\dv (\omega Y)(t)}_{C^{\al - 1}}
    		&\le C_\al e^{e^{c_\al t}},
    			\label{e:omegaYdivBound} \\
        \norm{(Y \cdot \grad) u(t)}_{C^\al}
            &\le C_\al e^{e^{c_\al t}},
                \label{e:YgraduBound} \\
        \norm{\grad \eta(t)}_{L^\iny}, \, \smallnorm{\grad \eta^{-1}(t)}_{L^\iny}
            &\le e^{e^{c_\al t}},
                \label{e:gradetaMainBound} \\
        \norm{A(t)}_{C^\al}, \,
        \norm{\nabla u(t) - \omega(t) A(t)}_{C^\al}
            &\le C_\al e^{e^{c_\al t}}.
                \label{e:ABound}
    \end{align}
    Suppose that $Y_0$ is divergence-free, and let $\phi_0$ be a stream
    function\footnote{For any continuous divergence-free vector field
    a stream function always exists and is unique up to an additive constant.}
    for $Y_0$; that is, $Y_0 = \grad^\perp \phi_0$. Let $\phi$ be $\phi_0$ transported
    by the flow, so that $\phi(t, x) := \phi_0(\eta^{-1}(t, x))$.
    Further, suppose that
    $\gamma_0$ is a $C^{1 + \al}$ level curve of $\phi_0$ and
    let $\gamma(t, \cdot) = \eta(t, \gamma_0(\cdot))$. Then $\gamma(t)$ is a $C^{1 + \al}$
    level curve of $\phi(t)$ with
    \begin{align}\label{e:LevelCurveSmoothness}
        \norm{\gamma(t)}_{C^{1 + \al}}
            \le C_\al e^{e^{c_\al t}}.
    \end{align}
\end{theorem}

The bounds in \crefrange{e:YdivBound}{e:ABound} are proven as part of the proof of \cref{T:Serfati}. We prove \cref{e:LevelCurveSmoothness} in \cref{S:ProofOfVortexPatch}.
 The bound in \cref{e:YgraduBound} means that $u$ remains $C^{1 + \al}$-smooth in the direction of $Y$. (For a classical vortex patch, it means that $\grad u$ is discontinuous only across the boundary.)

A simple consequence of \cref{T:Serfati} is the local propagation of \Holder regularity stated in \cref{T:LocalPropagation}. Critical to its proof is Serfati's construction of the matrix $A$;  \cref{e:graduCalLocalBound} of \cref{T:LocalPropagation} does not follow from \cite{Chemin1993Persistance} or \cite{C1998}, which has no analog of $A$.

\begin{theorem}\label{T:LocalPropagation}
    Let $\omega_0$, $Y_0$
    be as in \cref{T:Serfati}.
    If $\omega_0 \in C^{\beta}(U)$ for some open subset $U$ of $\R^2$
    and $\beta \in [0, 1)$ then
    $\omega(t) \in C^\beta(U)$ for all $t$, with
    \begin{align}\label{e:omegaCalLocalBound}
        \norm{\omega(t)}_{C^\beta(U_t)}
            \le \norm{\omega_0}_{C^\beta(U)}
                e^{e^{c_\al t}},
    \end{align}
    where $U_t = \eta(t, U)$.
    Further,
    \begin{align}\label{e:graduCalLocalBound}
            \norm{\grad u(t)}_{C^\al(U_t)}
            \le C_\al e^{e^{c_\al t}}.
    \end{align}
\end{theorem}
\begin{proof}
For any $x, y \in U_t$,
\begin{align*}
    \frac{\abs{\omega(t, x) - \omega(t, y)}}{\abs{x - y}^\beta}
        = \frac{\abs{\omega_0(\eta^{-1}(t, x)) - \omega_0(\eta^{-1}(t, y))}}
            {\abs{\eta^{-1}(t, x) - \eta^{-1}(t, y)}^\beta}
            \pr{\frac{\abs{\eta^{-1}(t, x) - \eta^{-1}(t, y)}}
                {\abs{x - y}}}^\beta.
\end{align*}
Together with \cref{e:gradetaMainBound} this gives \cref{e:omegaCalLocalBound} (a bound that holds for any Lipschitz velocity field). The bound in \cref{e:graduCalLocalBound} then follows from \cref{e:ABound}.
\end{proof}

\cref{T:LocalPropagation} improves, for initial data satisfying \cref{e:Assumption}, existing estimates of local propagation of \Holder regularity for bounded initial vorticity. For instance, Proposition 8.3 of \cite{MB2002} would only give $\grad u(t) \in C^\al_{loc}(U_t)$. 


In \cite{C1998}, Chemin extends the result he established in \cite{Chemin1993Persistance} by employing a family of vector fields in whose the direction the initial vorticity has regularity. We do the same for Serfati's initial conditions in \cref{S:SerfatiGen}, yielding the same result as Chemin. Moreover, we show in \cref{S:EquivInitConditions} that \cref{e:Assumption}$_3$ is equivalent to
\begin{align*} 
    Y_0 \cdot \grad u_0 \in C^\al(\R^2),
\end{align*}
meaning the initial velocity field has $C^{1 + \al}$ regularity in the direction of $Y_0$. By \cref{e:YgraduBound}, this regularity persists for all time.

In outline, Serfati's proof involves showing that the vorticity is transported over time in such a manner that its discontinuities are characterized by $Y(t)$. The regularity of $Y(t)$ is shown to be retained over all time, its estimate being inextricably entwined with an estimate on $\grad u(t)$ in $L^\iny$. At this high level, Serfati's approach is the same as that of Chemin in \cite{Chemin1993Persistance} and Constantin and Bertozzi in \cite{ConstantinBertozzi1993}.

Thus, though Serfati's approach is novel in many ways, it clearly owes much to both \cite{Chemin1993Persistance} and \cite{ConstantinBertozzi1993}. Like Chemin, Serfati proves a more general result involving the persistence of tangential regularity along a vector field, for which a vortex patch is a special case.
Like Bertozzi and Constantin, Serfati uses estimates on singular integrals, some of them in much the same manner (\cref{C:graduCor}, for instance, a special case of which is used in \cite{ConstantinBertozzi1993}). Unlike \cite{ConstantinBertozzi1993}, however, Serfati uses no ``geometric lemma'' and there is also no clear analog in \cite{Chemin1993Persistance, ConstantinBertozzi1993} of Serfati's linear algebra lemma, \cref{L:SerfatiLemma1}.

More concisely, one could say that the setup of the problem and the use of transport estimates is much like that of \cite{Chemin1993Persistance} while the estimates involving the gradient of the velocity are more like that of \cite{ConstantinBertozzi1993}, but the most difficult to estimate term is bounded in an entirely novel way. Plus, Serfati characterizes the gradient of the velocity more fully. (Also see the comments at the end of \cref{S:EquivInitConditions}.)

There has been a number of papers since \cite{SerfatiVortexPatch1994} related to the regularity of the boundary of vortex patches. Among those treating 2D vortex patches are \cite{CW1995, C1996, Danchin1996, Danchin1997, SueurViscousVortexPatches2014}, which study the inviscid limit (see \cite{SueurViscousVortexPatches2014} for historical comments as well); \cite{Chemin1995SingularVortexPatches, Danchin1997SingularVortexPatches, Danchin1997Cusp, Danchin2000Cusp}, which study vortex patch boundaries having singularities; \cite{Depauw1998, Depauw1999} for vortex patches in a bounded domain.
None of these, however, have used techniques from \cite{SerfatiVortexPatch1994}: they are all intellectual descendants of either \cite{Chemin1993Persistance, C1998} or \cite{ConstantinBertozzi1993} (or both).

\medskip

This paper is organized as follows. In \cref{S:Notation}, we fix some notation and make a few  definitions. In \cref{S:Lemmas}, we state six key lemmas we will need. In \cref{S:Transport}, we study the transport equations of $Y$ and a related vector field $R$, as well as the propagation of regularity of $\dv(\omega Y)$. In \cref{S:ProofOfSerfatisResult}, we prove the first part of \cref{T:Serfati}, the bounds in \cref{e:MainBounds}, and also prove \crefrange{e:YdivBound}{e:gradetaMainBound}. In \cref{S:MatrixA}, we prove the second part of \cref{T:Serfati}, the existence of the matrix $A$ satisfying \cref{e:AEq}, along with the proof of \cref{e:ABound}. In \cref{S:ProofOfVortexPatch}, we consider the case of a classical vortex patch, showing how \cref{T:VortexPatch} follows from \cref{T:Serfati}, along the way proving \cref{e:LevelCurveSmoothness}. In \cref{S:SerfatiGen}, we describe an extension of Serfati's result to a family of vector fields like those of Chemin in \cite{C1998}, in \cref{S:EquivInitConditions} showing that the resulting hypotheses on the initial data are equivalent to those of \cite{C1998}. In \cref{S:Examples}, we discuss several examples of initial data that satisfy the hypotheses of \cref{T:Serfati}. \ifbool{ForSubmission}{}{In \cref{S:CircularVortexPatch} we give the form of the matrix $A$ for two simple examples. In \cref{S:Conclusion} we make some concluding remarks.}In \cref{A:ProofsOfLemmas}, we prove the lemmas stated in \cref{S:Lemmas}. In \cref{A:graduCalcs}, we detail some calculations involving $\grad u$. Finally, in \cref{A:TransportEstimates}, we discuss our use of weak transport equations.

\section{Notation, conventions, and definitions}\label{S:Notation}

\noindent
We define
\begin{align*} 
    \grad u
        := \matrix{\prt_1 u^1 & \prt_2 u^1}{\prt_1 u^2 & \prt_2 u^2}
        = Du,
\end{align*}
the Jacobian matrix of $u$,
and define the gradient of other vector fields in the same manner.
\Ignore{ 
It is then natural to treat $u$ as a column vector, allowing us to rewrite \cref{e:Eu}$_1$ as
\begin{align*}
    \prt_t u + \grad u \, u + \grad p
        = \prt_t u + Du \, u + \grad p
        = 0.
\end{align*}
(Were we to treat $u$ as a row vector, it would be more natural to define $\grad u = (Du)^T$, in which case $u \cdot \grad u = u \grad u$. There is no universal convention, as the need to define $\grad u$ explicitly seldom arises, but our choice does agree with that used implicitly in \cite{ConstantinBertozzi1993}.)

Further, we define the gradient of any vector field as in \cref{e:gradu}.
We will need, however, the gradient of only two other vector fields: of $\grad \Cal{F}$ and of the Biot-Savart kernel, $K = \grad^\perp \Cal{F}$. Both of these vector fields are symmetric (the latter on domains bounded away from the origin, which is the only use we will make of it), so there is no ambiguity anyway in their definition. (The gradient of other vector fields, such as $Y$, occurs only in the form $u \cdot \grad Y := (u \cdot \grad) Y$, so $\grad Y$ needs no separate definition.)
} 


We follow the common convention that 
the gradient
and divergence operators apply only to the spatial variables.

Another common convention we follow is that constants denoted by $C$ depend only on the quantities specified in the context or stated explicitly, such as in $C(\omega_0, Y_0)$, but do not depend on other parameters, such as $\al$ or $\eps$. In a series of inequalities, the value of $C$ can vary with each inequality. 

All vectors are to be treated as column vectors for linear algebra operations, even when written in the form $(v^1, v^2)$.

We write $\abs{v}$ for the Euclidean norm of $v = (v^1, v^2)$, $\abs{v}^2 = (v^1)^2 + (v^2)^2$. We use the same notation for matrices, but will find it more convenient to define, for $M = (M^{i j})$,
\begin{align*}
    \abs{M}
        := \max_{i, j} \abs{M^{i j}},
\end{align*}
so that $\abs{I} = 1$ for the identity matrix.
Of course, all norms on finite-dimensional spaces are equivalent, so these choices make little difference as they just change the values of constants.

If $X$ is a function space, we define
\begin{align*}
    \norm{v}_X
        := \norm{\abs{v}}_X, \quad
    \norm{M}_X
        := \norm{\abs{M}}_X.
\end{align*}

\begin{definition}[\Holder and Lipschitz spaces]\label{D:HolderSpaces}
        Let $\alpha \in(0, 1)$ and $U \subseteq \R^d$, $d \ge 1$, be open.
        Then $C^\al(U)$ is the space of all measurable functions for which
        \begin{align*}
            \norm{f}_{C^\alpha(U)}
                := \norm{f}_{L^\iny(U)} + \norm{f}_{\dot{C}^\alpha(U)}
                    < \iny, \quad
            \norm{f}_{\dot{C}^\alpha(U)}
                := \sup_{\substack{x, y \in U \\ x \ne y}}
                    \frac{\abs{f(x) - f(y)}}{\abs{x - y}^\alpha}.
        \end{align*}
        
        For $\al = 1$, we obtain the Lipschitz space, which is not called $C^1$
        but rather $Lip(U)$. We also define $lip(U)$ for the homogeneous space.
        Explicitly, then,
        \begin{align*}
            \norm{f}_{Lip(U)}
                := \norm{f}_{L^\iny(U)} + \norm{f}_{lip(U)}, \quad
            \norm{f}_{lip(U)}
                := \sup_{\substack{x, y \in U \\ x \ne y}}
                    \frac{\abs{f(x) - f(y)}}{\abs{x - y}}.
        \end{align*}
        
        For any positive integer $k$, $C^{k + \al}(U)$ is the space
        of $k$-times continuously differentiable functions on $U$ for which
        \begin{align*}
            \norm{f}_{C^{k + \al}(U)}
                := \sum_{\abs{\beta} \le k} \smallnorm{D^\beta f}_{L^\iny(U)}
                    + \sum_{\abs{\beta} = k} \smallnorm{D^\beta f}_{C^\al(U)}
                    < \iny.
        \end{align*}
        We define the negative \Holder space, $C^{\al - 1}(U)$, by
    \begin{align*}
        C^{\al - 1}(U)
            &= \set{f + \dv v \colon f, v \in C^\al(U)}, \\
        \norm{h}_{C^{\al - 1}(U)}
            &= \inf\set{
                \norm{f}_{C^\al(U)}
                    + \norm{v}_{C^\al(U)}
                    \colon
                        h = f + \dv v; \, f, \, v \in C^\al(U)
                }.
    \end{align*}
\end{definition}

It follows immediately from the definition of $C^{\al - 1}$ that
\begin{align}\label{e:dvvCalBound}
    \norm{\dv v}_{C^{\al - 1}}
        \le \norm{v}_{C^\al}.
\end{align}
We also have the elementary inequalities,
\begin{align}\label{e:CdotIneq}
    \begin{split}
        \norm{f \circ g}_{\dot{C}^\al}
            &\le \norm{f}_{\dot{C}^\al} \norm{\grad g}_{L^\iny}^\al, \\
        \norm{f g}_{C^\al}
            &\le \norm{f}_{C^\al} \norm{g}_{C^\al}.
    \end{split}
\end{align}

\Ignore{ 
\begin{definition}\label{D:CurveRegularity}
    We say that $\gamma$ is a $C^{k + \al}$ curve if
    $\gamma \colon I \to \R^2$ is $C^{k + \al}(I)$ and injective, where $I$ is an
    interval and the $C^{k + \al}(I)$ space is as in \cref{D:HolderSpaces}.
    If $\gamma$ is periodic (or is defined on the unit circle $\mathbb{S}^1$) then
    the curve is closed. We say that a one-dimensional manifold,
    such as the boundary of an open set, is
    $C^{k + \al}$ if it is the image of some 
    $C^{k + \al}$ curve.
\end{definition}
} 

\begin{definition}[The inf ``norm'']\label{D:infNorm}
    For any measurable subset $\Lambda \subseteq \R^2$ and measurable
    function $f$ on $\Lambda$, we define
    \begin{align*}
        \norm{f}_{\inf(\Lambda)} = - \norm{-\abs{f}}_{L^\iny(\Lambda)}.
    \end{align*}
\end{definition}

\begin{definition}\label{D:Sigmadelta}
    For any $\Lambda \subseteq \R^2$, we define
    $\Ndelta{\Lambda}{\delta} := \set{x\in \R^2: d(x, \Lambda) < \delta}$, where $d$ is the
    Euclidean distance in $\R^2$.  
\end{definition}

\begin{definition}[Radial cutoff functions]\label{D:Radial}
    We make an arbitrary, but fixed, choice of a radially symmetric function
    $a \in C_C^\iny(\R^2)$ taking values in $[0, 1]$ with $a = 1$ on $B_1(0)$
    and $a = 0$ on
    $B_2(0)^C$. For $r > 0$, we define the rescaled cutoff function, $a_r(x) = a(x/r)$,
    and for $r, h > 0$ we define
    \begin{align*}
        \mu_{rh} = a_r(1 - a_h).
    \end{align*}
\end{definition}

\begin{remark}\label{R:Radial}
When using the cutoff function $\mu_{rh}$ we will be fixing $r$ while taking $h \rightarrow  0$, in which case we can safely assume that $h$ is sufficiently smaller than $r$ so that $\mu_{rh}$ vanishes outside of $(h, 2r)$ and equals 1 identically on $(2h, r)$. It will then follow that
\begin{align*} 
    \left\{
    \begin{array}{rl}
        \halfspacer
        \abs{\grad \mu_{rh}(x)} \le C h^{-1} \le C \abs{x}^{-1}
            & \text{for } \abs{x} \in (h, 2h), \\
        \halfspacer
        \abs{\grad \mu_{rh}(x)} \le C r^{-1} \le C \abs{x}^{-1}
            & \text{for } \abs{x} \in (r, 2r), \\
        \grad \mu_{rh} \equiv 0
            & \text{elsewhere}.
    \end{array}
    \right.
\end{align*}
Hence, also, $\abs{\grad \mu_{rh}(x)} \le C \abs{x}^{-1}$ everywhere.
\end{remark}

\begin{definition}[Mollifier]\label{D:Mollifier}
    Let $\rho \in C_C^\iny(\R^2)$ with $\rho \ge 0$ have $\norm{\rho}_{L^1} = 1$.
    For $\eps > 0$, define $\rho_\eps(\cdot) = (\eps^{-2}) \rho(\cdot / \eps)$.
\end{definition}

\begin{definition}[Principal value integral]\label{D:PV}
For any measurable integral \textit{kernel}, $L \colon \R^2\times \R^2 \to \R$, and any measurable function, $f \colon \R^2 \to \R$, define the integral transform $L[f]$ by
\begin{align*}
    L[f](x)
        := \PV \int_{\R^2} L(x, y) f(y) \, dy
        &:= \lim_{h \to 0^+} \int_{\abs{x - y} > h} L(x, y) \, f(y) \, dy,
\end{align*}
whenever the limit exists.
\end{definition}

When dealing with algebraic manipulations of principal value integrals, we will find it convenient to introduce the notation in \cref{D:PValg}.
\begin{definition}\label{D:PValg}
    \begin{align*}
        f \starp g (x)
                := \PV \int f(x - y) g(y) \, dy.
    \end{align*}
\end{definition}

%
%
\section{Key lemmas}\label{S:Lemmas}

\noindent \crefrange{L:SerfatiLemma1}{L:Gronwall} are six key lemmas we will need in the proof of \cref{T:Serfati}. \Cref{L:SerfatiLemma1,L:SerfatiLemma2} are the two lemmas of \cite{SerfatiVortexPatch1994}: a simple, if seemingly unmotivated, linear algebra lemma and an estimate on an integral transform that includes singular integrals. \cref{L:SerfatiLemma2Inf} is a variant on \cref{L:SerfatiLemma2}, while \cref{L:SerfatiKernels} gives explicit estimates on the four kernels to which we will apply \cref{L:SerfatiLemma2,L:SerfatiLemma2Inf}. \cref{L:EquivalentConditions} is used to establish the equivalence of \cref{e:SerfatiInitialData}$_3$ and \cref{e:Assumption}$_3$, and \cref{L:Gronwall} is the form of Gronwall's lemma that we will need. We give the proof of \cref{L:SerfatiKernels} in this section and defer the proof of the other lemmas (except for Gronwall's lemma, which is classical)  to \cref{A:ProofsOfLemmas}.

\begin{lemma}\label{L:SerfatiLemma1}
    Let
    $
        M
            = \matrix{a & b}{c & d}
    $
    be an invertible matrix.
    For any  $2 \times 2$ symmetric matrix $B$,
    \begin{align*}
        \displaystyle \abs{B}
            \le C \frac{\abs{M}}{\det M}
                 \abs{B M_1}
                + C \abs{\tr B},
    \end{align*}
    \Ignore{ 
    \begin{align*}
        \norm{B}
            \le C \brac{
                \frac{a^{5}+b^{5}+c^{5}+d^{5}}{(a^2+c^2)  \det M}
                \abs{BM_1} +\abs{\tr B}
                },
    \end{align*}
    }
    where $M_1 =(a,c)^T$ is the first column of $M$.
\end{lemma}

\begin{lemma}\label{L:SerfatiLemma2}
    Let $L \colon \R^2\times \R^2 \to \R$ be an integral kernel for which
    \begin{align*}
        \norm{L}_*
            := \sup_{x, y \in \R^2}
            \bigset{\abs{x - y}^2 \abs{L(x, y)}
                + \abs{x - y}^{3} \abs{\grad_x L(x, y)}}
        < \infty   
    \end{align*}
    and for which
    \begin{align}\label{e:LPVinL1}
        \abs{\PV \int_{\R^2} L(x, y) \, dy}
            < \iny
            \text{ for all } x \in \R^2.
    \end{align}
    Let $L[f]$ be as in \cref{D:PV}.
    Then
    \begin{align}\label{e:KernelEstimate}
        \begin{split}
        \norm{L[f-f(x)](x)}_{\dot{C}_x^\alpha}
            &= \norm{\PV \int_{\R^2} L(x, y)
                \brac{f(y) - f(x)} \, dy}_{\dot{C}_x^\alpha} \\
            &\le C \al^{-1} (1 - \al)^{-1}
                \norm{L}_* \norm{f}_{\dot{C}^\alpha}.
        \end{split}
    \end{align}
    If 
    \begin{align}\label{e:LHomo}
        \PV \int_{\R^2} L(\cdot, y) \, dy \equiv 0
    \end{align}
    then
    \begin{align}\label{e:KernelEstimateHomo}
        \norm{L[f]}_{\dot{C}^\alpha}
            \le C \al^{-1} (1 - \al)^{-1}
                \norm{L}_* \norm{f}_{\dot{C}^\alpha}.
    \end{align}
\end{lemma}

The inequality in \cref{e:KernelEstimateHomo} is a classical result relating a Dini modulus of continuity of $f$ to a singular integral operator applied to $f$ in the special case where the modulus of continuity is $r \mapsto C r^\al$. (See, for instance, the lemma in \cite{KNV2007}, and note that applying that lemma to a $C^\al$ function gives the same factor of $\al^{-1}(1 - \al)^{-1}$ that appears in \cref{L:SerfatiLemma2}. This reflects the fact that the integral transform in \cref{e:KernelEstimate} applied to a $C^1$-function gives only a log-Lipschitz function, and applied to a $C^0$-function yields no modulus of continuity.)

\cref{L:SerfatiLemma2Inf} allows us to bound the full $C^\al$ norm.

\begin{lemma}\label{L:SerfatiLemma2Inf}
    Let $L$ be as in \cref{L:SerfatiLemma2} and suppose further that
    \begin{align*}
        \norm{L}_{**}
            := \norm{L}_*
                + \sup_{x \in \R^2} \norm{L(x, \cdot)}_{L^1(B_1(x)^C)}
            < \iny.
    \end{align*}
    Then the conclusions of \cref{L:SerfatiLemma2} hold with each $\dot{C}^\al$
    replaced by $C^\al$ and $\norm{L}_*$ replaced by $\norm{L}_{**}$.
    \Ignore{ 
    Moreover, $\norm{L[f]}_{L^\iny}$ is bounded as in \cref{e:KernelEstimate}
    or \cref{e:KernelEstimateHomo}, but with the factor of $\al^{-1}$ in place
    of $\al^{-1} (1 - \al)^{-1}$.
    } 
\end{lemma}

We shall apply \cref{L:SerfatiLemma2} to the kernel $L_2$ of \cref{L:SerfatiKernels} and apply \cref{L:SerfatiLemma2Inf} to the kernels $L_1$, $L_3$, and $L_4$.
Note that for $L_2$, $L_3$, and $L_4$, we are actually applying \cref{L:SerfatiLemma2} to each of their components.

\begin{lemma}\label{L:SerfatiKernels}
    Consider the four kernels,
    \begin{enumerate}
        \item
            $L_1(x, y) = \rho_\eps(x - y) \omega_0(y)$;
            
        \item
            $L_2(x,y) = \grad (a_r K)(x - y)$ for some fixed $r > 0$;

        \item
            $L_3(x, y) = \grad K(x - y) \omega(y)$, where $\omega \in C_C^\iny(\R^2)$;
            
        \item  
            $L_4(x, y) = \rho_\eps(x - y) \grad u_0(y)$.
                \end{enumerate}
    Here, $K$ is the Biot-Savart kernel of \cref{e:BS}.
    Then
    $\norm{L_1}_{**} \le C \norm{\omega_0}_{L^\iny}$ for $C$
    independent of $\eps$; $L_2$ satisfies \cref{e:LHomo} with
    $\norm{L_2}_* \le C$ independently of $r$;
    $\norm{L_3}_{**} \le C V(\omega)$ with
    \begin{align}\label{e:Vomega}
        V(\omega)
            :=
            \norm{\omega}_{L^\iny}
                + \norm{\PV \int \grad K(x - y) \omega(y) \, dy}_{L^\iny};
    \end{align}
    and $\norm{L_4}_{**} \le C \norm{\grad u_0}_{L^\iny}$ for $C$
    independent of $\eps$.
\end{lemma}
\begin{proof}
    The bounds on the $*$-norms of $L_1$, $L_2$, and $L_4$ are easily verified, the
    key points being their $L^1$-bound uniform in $x$,
    the decay of $K(x - y)$ and $\grad_x K(x - y)$,
    and the scaling of $\rho_\eps(x - y)$ and $\grad_x \rho_\eps(x - y)$ in terms of $\eps$.
    For $L_3$, $\norm{L_3}_*$
    is bounded as for $L_2$, with the $\PV$ integral in
    \cref{e:Vomega} coming from the final term in $\norm{L}_{**}$ .
\end{proof}

\begin{lemma}\label{L:EquivalentConditions}
    Let $\alpha \in (0,1)$ and $Z\in L^{\infty}(\R^2)$.
    Then, $\dv Z \in C^{\alpha-1}(\R^2)$ if and only if
    $\grad \mathcal{F} \ast \dv Z \in C^\alpha(\R^2)$ (equivalently,
    $K * \dv Z \in C^\al(\R^2)$).
    Moreover,
    \begin{align}\label{e:EquivalentConditionsBound}
        \norm{\dv Z}_{C^{\al - 1}}
            \le \norm{\grad \Cal{F} * \dv Z}_{C^\al}
            \le C \pr{\norm{Z}_{L^\iny} + \norm{\dv Z}_{C^{\al - 1}}}.
    \end{align}
\end{lemma}

\begin{lemma}[Gronwall's lemma and reverse Gronwall's lemma]\label{L:Gronwall}
    Suppose $h \ge 0$ is a continuous nondecreasing or nonincreasing function on $[0,T]$,
    $g \ge 0$ is an
    integrable function on $[0,T]$, and
    \begin{align*}
        f(t)
            \le h(t) + \int_0^t g(s) f(s) \, ds
                \ \text{ or } \ 
        f(t)
            \ge h(t) - \int_0^t g(s) f(s) \, ds     
    \end{align*}
    for all $t \in [0, T]$. Then
    \begin{align*}
        f(t) \le h(t) \exp \int_0^t g(s) \, ds
            \ \text{ or } \ 
        f(t) \ge h(t) \exp \pr{-\int_0^t g(s) \, ds},
    \end{align*}
    respectively, for all $t \in [0, T]$.
\end{lemma}

\section{Approximate solutions and transport equations}\label{S:Transport}

\noindent We first regularize the initial data by setting $u_{0, \eps} = \rho_\eps * u_0$, where $\rho_\eps$ is the standard mollifier of \cref{D:Mollifier}, letting $\eps$ range over values in $(0, 1]$. It follows that $\omega_{0, \eps} = \rho_\eps * \omega_0$. Then there exists a solution, $\omega_\eps(t)\in C^\iny(\R^2)$, to the Euler equations \cref{e:Eomega} for all time with $C^\iny$ velocity field, $u_\eps$ (\cite{Leray1933, Wolibner1933} or see Theorem 4.2.4 of \cite{C1998}). These solutions converge to a solution $\omega(t)$ of \cref{e:Eomega}. (We say more about convergence in \cref{S:Convergence}.)

The flow map, $\eta_\eps$, is given in \cref{e:etaDef} with $u_\eps$ in place of $u$. Moreover, all the $L^p$-norms of $\omega_\eps$ are conserved over time with
\begin{align}\label{e:omegaNormp}
    \norm{\omega_\eps(t)}_{L^{p}}
        = \norm{\omega_{\eps, 0}}_{L^{p}}
        \le \norm{\omega_0}_{L^{p}}
        \le \norm{\omega_0}_{L^1 \cap L^\iny}
        =: \norm{\omega_0}_{L^1} + \norm{\omega_0}_{L^\iny}
\end{align}
for all $p \in [1, \iny]$. Also,
\begin{align}\label{e:uepsBound}
    \norm{u_\eps(t)}_{L^\iny} \le C \norm{\omega_0}_{L^1 \cap L^\iny}    
\end{align}
(see Proposition 8.2 of \cite{MB2002}) so $\norm{u_\eps}_{L^\iny(\R \times \R^2)}$ is uniformly bounded in $\eps$.
\OptionalDetails{
    To derive the last inequality in \cref{e:omegaNormp}, we start with the interpolation
    (or convexity) inequality for Lebesgue spaces,
    \begin{align*}
        \norm{f}_{L^p} \le \norm{f}_{L^s}^\theta \norm{f}_{L^r}^{1 - \theta},
    \end{align*}
    where $1 \le s \le p \le r \le \iny$ and
    \begin{align*}
        \frac{1}{p} = \frac{\theta}{s} + \frac{1 - \theta}{r}.
    \end{align*}
    (See, for instance, (II.2.7) of \cite{G1994}.)
    Choosing $s = 1$, $r = \iny$ gives $\theta = p^{-1}$. Hence,
    \begin{align*}
        \norm{f}_{L^p} 
            \le \norm{f}_{L^1}^{\frac{1}{p}}
                \norm{f}_{L^\iny}^{1 - \frac{1}{p}}
            = \norm{f}_{L^1}^{\frac{1}{p}}
                \norm{f}_{L^\iny}^{\frac{p - 1}{p}}.
    \end{align*}
    Applying Young's inequality, we have
    \begin{align*}
        \norm{f}_{L^p} 
            \le \frac{\pr{\norm{f}_{L^1}^{\frac{1}{p}}}^p}{p}
                + \frac{\pr{\norm{f}_{L^\iny}^{\frac{p - 1}{p}}}^\frac{p}{p - 1}}
                    {\frac{p}{p - 1}}.   
            = \frac{\norm{f}_{L^1}}{p}
                + \frac{\norm{f}_{L^\iny}}
                    {\frac{p}{p - 1}}
            \le \norm{f}_{L^1} + \norm{f}_{L^\iny}.  
    \end{align*}

    }

For the most of the proof we will use these smooth solutions, passing to the limit as $\eps \to 0$ in the final steps in \cref{S:Convergence}.

We let
\begin{align}\label{e:Yeps}
    Y_\eps(t, \eta_\eps(t, x)) = Y_0(x) \cdot \nabla \eta_\eps(t, x)
\end{align}
be the pushforward of $Y_0$ under the flow map $\eta_\eps$ (as in \cref{e:Y}). (Note the slight notational collision between $Y_\eps$ and $Y_0$ and $\omega_\eps$ and $\omega_0$; this should not, however, cause any confusion.)

Standard calculations show that
\begin{align}\label{e:YTransportAlmost}
    \prt_t Y_\eps + u_\eps \cdot \nabla Y_\eps
        = Y_\eps \cdot \nabla u_\eps
\end{align}
and that
\begin{align}\label{e:YTransportDiv}
    \begin{split}
        \prt_t \dv Y_\eps + u_\eps \cdot \grad \dv Y_\eps &= 0, \\
        \prt_t \dv (\omega_\eps Y_\eps)
            + u_\eps \cdot \grad \dv (\omega_\epsilon Y_\eps) &= 0,
    \end{split}
\end{align}
the latter equality using that the vorticity is transported by the flow map.
Hence,
\begin{align}\label{e:divomegaY}
    \begin{split}
		\dv Y_\eps(t, x) = \dv Y_0(\eta_\eps^{-1}(t, x)), \\
        \dv (\omega_\eps Y_\eps)(t, x) = \dv (\omega_{0, \eps} Y_0)(\eta_\eps^{-1}(t, x)).
    \end{split}
\end{align}

\begin{remark}\label{R:WeakTransport}
Actually, the transport equations in \cref{e:YTransportAlmost,e:YTransportDiv}, and others we will state later, are satisfied in a weak sense, since $Y_0$ and $\dv (\omega_{0, \eps} Y_0)$ only lie in $C^\al$. We refer to Definition 3.13 of \cite{BahouriCheminDanchin2011} for the notion of weak transport. With the exception of the use of Theorem 3.19 of \cite{BahouriCheminDanchin2011} in the proof of \cref{L:RegularitywY}, we will treat all transport equations as though they are satisfied in a strong sense, however, justifying such use in \cref{A:TransportEstimates}. (See also \cref{R:CannotMollifyY0}.)
\end{remark}

We can also write \cref{e:YTransportDiv} as
\begin{align}\label{e:TwoTransports}
    \begin{split}
        &\diff{}{t} Y_\eps(t, \eta_\eps(t, x))
            = (Y_\eps \cdot \nabla u_\eps)(t, \eta_\eps(t, x)), \\
        &\diff{}{t} (\omega_\eps Y_\eps)(t, \eta_\eps(t, x))
            = 0.
    \end{split}
\end{align}

Define the initial vector field
\begin{align}\label{e:R0Def}
    R_{0, \eps}
        = \omega_{0, \eps} Y_0
            + \rho_\eps * \nabla \mathcal{F} * \dv (\omega_0 Y_0)
            - \rho_\eps * \left(\omega_0 Y_0\right)
\end{align}
and observe that
\begin{align*}
    \dv R_{0, \eps}
        &= \dv (\omega_{0, \eps} Y_0)
            + \dv \bigpr{\rho_\eps * \pr{\nabla \mathcal{F}
            * \dv \left(\omega_{0} Y_0\right)
            - \omega_{0} Y_0}}
        = \dv (\omega_{0, \eps} Y_0),
\end{align*}
where we used that $\Delta \mathcal{F}$ is the Dirac delta function. Hence, pushing forward $R_{0, \eps}$ will have the same effect on $\dv (\omega_{0, \eps} Y_0)$ as does pushing forward $\omega_{0, \eps} Y_0$ itself; that is, letting
\begin{align*}
    R_\eps(t, \eta_\eps(t, x)) = R_{0, \eps}(x) \cdot \nabla \eta_\eps(t, x),
\end{align*}
we have
\begin{align*} 
    \prt_t R_\eps + u_\eps \cdot \nabla R_\eps
        = R_\eps \cdot \nabla u_\eps
\end{align*}
and
\begin{align}\label{e:wYREquiv}
    \dv  (\omega_\eps Y_\eps)(t, x)
        = \dv R_\eps(t, x)
        = \dv R_{0, \eps} \left(\eta_\eps^{-1}(t, x)\right)
        = \dv (\omega_{0, \eps} Y_0)(\eta_\eps^{-1}(t, x)).
\end{align}
\Ignore{ 
We note also that as in \cref{e:YTransportAlmost,e:divomegaY},
\begin{align*} 
    \begin{split}
        \prt_t R_\eps + u_\eps \cdot \nabla R_\eps
            &= R_\eps \cdot \nabla u_\eps, \\
        \prt_t \dv R_\eps + u_\eps \cdot \grad \dv R_\eps &= 0.
    \end{split}
\end{align*}
} 

Although $R_\eps$ and $\omega_\eps Y_\eps$ have the same divergence, we will see in the proof of \cref{T:Serfati} that $R_\eps$ is bounded in $C^\al$ uniformly in $\eps$ in (\ref{e:YRCalphaBound}), which is not true of $\omega_\eps Y_\eps$. We will take advantage of this fact in the estimate in 
(\ref{e:ekBound}).

\begin{lemma}\label{L:R0Decomp}
The vector field $R_{0, \eps}$, defined in \cref{e:R0Def}, is in $C^\alpha(\R^2)$, with
\begin{align*}
    \norm{R_{0, \eps}}_{C^\al} \le C_\al,
\end{align*}
uniformly over $\eps$ in $(0, 1]$, where $C_\al$ is as in \cref{e:Calpha}.
\end{lemma}
\begin{proof}
We rewrite $R_{0, \eps}$ in the form,
\begin{align*}
    R_{0, \eps}
        = \rho_\eps * \nabla \mathcal{F} * \dv(\omega_0 Y_0)
            +\bigbrac{(\rho_\eps * \omega_0)Y_0 - \rho_\eps * (\omega_0 Y_0)}.
\end{align*}
Since $\displaystyle \nabla \mathcal{F} * \dv (\omega_0 Y_0)\in C^\alpha(\R^2)$ by \cref{L:EquivalentConditions} (noting that $\omega_0 Y_0 \in L^\iny$), we have 
\begin{align}\label{e:R0InternalBound}
    \norm{\rho_\eps * \nabla \mathcal{F} * \dv(\omega_0 Y_0)}_{C^{\alpha}}
        \le C \norm{\nabla \mathcal{F} * \dv(\omega_0 Y_0)}_{C^{\alpha}}
        \le C(\omega_{0}, Y_{0}).
\end{align}
Since $Y_0 \in C^\alpha(\R^2)$, applying \cref{L:SerfatiLemma2Inf} with the kernel $L_1$ of \cref{L:SerfatiKernels}, we have
\begin{align}\label{e:rhoepsY0Est}
    \begin{split}
    &\norm{(\rho_\eps * \omega_0)Y_0 - \rho_\eps * (\omega_0 Y_0)}_{C^\al}
        = \norm{\int_{\R^2} \rho_\eps(x - y)\omega_0(y) \brac{Y_0(x)-Y_0(y)}
                \, dy}_{C^\al} \\
        &\qquad
        \le C(\omega_{0}, Y_{0}) \pr{\al^{-1} (1 - \al)^{-1}}
        = C_\al.
    \end{split}
\end{align}
This completes the proof. 
\end{proof}

Finally, we prove the propagation of regularity of $\dv (\omega_\eps Y_\eps)$.

\begin{lemma}\label{L:RegularitywY}
We have $\dv (\omega_\eps Y_\eps)(t) \in C^{\alpha-1}(\R^2)$ with 
\begin{align*}
    \norm{\dv (\omega_\eps Y_\eps)(t)}_{C^{\alpha-1}}
        \le C_\al
            \exp\int_0^t \norm{\nabla u(s)}_{L^\iny} \, ds.
\end{align*}
\end{lemma}

\begin{proof}
\Ignore{ 
As observed in \cref{e:wYREquiv}, $\dv (\omega_\eps Y_\eps)$ and $\dv R_\eps$ are the same.

As noted above,
\begin{align*} 
    \begin{split}
        &\prt_t \dv (\omega_\eps Y_\eps) + u_\eps \cdot \nabla \dv(\omega_\eps Y_\eps)
            = 0, \\
        &\dv (\omega_\eps Y_\eps)(0, x)
            =\dv (\omega_{0, \eps} Y_0) \in C^{\alpha-1}(\R^2). 
    \end{split}
\end{align*}
} 
Noting that $C^{\alpha-1}(\R^2)$ is equivalent to the Besov space $B^{\alpha-1}_{\infty,\infty}(\R^2)$, Theorem 3.14 of \cite{BahouriCheminDanchin2011} applied to the weak transport equation in \cref{e:YTransportDiv} (see \cref{R:WeakTransport}) gives
\begin{align*}
    \norm{\dv (\omega_\eps Y_\eps)(t)}_{C^{\alpha-1}}
        \le C \norm{\dv (\omega_{0, \eps} Y_0)}_{C^{\alpha-1}}
            \exp\int_0^t \norm{\nabla u(s)}_{L^\iny} \, ds.
\end{align*}
We must still, however, bound $\norm{\dv (\omega_{0, \eps} Y_0)}_{C^{\alpha-1}}$ uniformly in $\eps$. 

From the triangle inequality,
\begin{align*}
    \norm{\dv (\omega_{0, \eps} Y_0)}_{C^{\alpha-1}}
        \le \norm{\dv(\omega_{0, \eps} Y_0) - \rho_\eps * \dv (\omega_0 Y_0)}_{C^{\al - 1}}
            + \norm{\rho_\eps * \dv (\omega_{0, \eps} Y_0)}_{C^{\alpha-1}}.
\end{align*}
Now,
\begin{align*}
    \norm{\dv(\omega_{0, \eps} Y_0) - \rho_\eps * \dv (\omega_0 Y_0)}_{C^{\al - 1}}
        \le \norm{\omega_{0, \eps} Y_0 - \rho_\eps * (\omega_0 Y_0)}_{C^\al}
        \le C_\al,
\end{align*}
the first inequality following from \cref{e:dvvCalBound}, the second from 
\cref{e:rhoepsY0Est}.
Also,
\begin{align*}
    \norm{\rho_\eps * \dv (\omega_0 Y_0)}_{C^{\al - 1}}
        &\le C\norm{\grad \Cal{F} * (\rho_\eps * \dv (\omega_0 Y_0))}_{C^\al}
        = C\norm{\rho_\eps * (\grad \Cal{F} * \dv (\omega_0 Y_0))}_{C^\al} \\
        &\le C\norm{\grad \Cal{F} * \dv (\omega_0 Y_0)}_{C^\al}
        \le C \pr{\norm{\omega_0 Y_0}_{L^\iny}
            + \norm{\dv (\omega_0 Y_0)}_{C^{\al - 1}}}.
\end{align*}
For the first inequality we applied \cref{L:EquivalentConditions}, for the second inequality we used $\norm{\rho_\eps * f}_{C^\al} \le \norm{f}_{C^\al}$, and for the third we applied \cref{L:EquivalentConditions} once more.
\Ignore{ 
\begin{align*}
    \norm{\rho_\eps * f}_{\dot{C}^\al}
        = \sup_{x \ne z} \abs{\int \rho_\eps(y)
            \frac{f(x - y) - f(z - y)}{\abs{x - z}^\al} \, dy}
        \le \norm{f}_{\dot{C}^\al} \int \rho_\eps(y) \, dy
        = \norm{f}_{\dot{C}^\al}.
\end{align*}
} 
Hence,
\begin{align*}
    \norm{\dv (\omega_{0, \eps} Y_0)}_{C^{\alpha-1}}
        \le C_\al
            + \pr{\norm{\omega_0 Y_0}_{L^\iny}
                + \norm{\dv (\omega_0 Y_0)}_{C^{\al - 1}}}
        \le C_\al.
\end{align*}
\end{proof}

\begin{remark}\label{R:CannotMollifyY0}
    It would be natural to let $Y_{0, \eps} = \rho_\eps * Y_0$ and pushforward
    $Y_{0, \eps}$ rather than $Y_0$ in the definition of $Y_\eps$, and also
    use $Y_{0, \eps}$ rather than $Y_0$ in the definition of $R_{0, \eps}$.
    This would allow us to use transport equations purely in strong form.
    It is the bound in \cref{e:R0InternalBound}, however, that prevents us from
    doing this, as the equivalent bound with $Y_{0, \eps}$ in place of $Y_0$
    may not hold true. Instead, we take the approach described in
    \cref{A:TransportEstimates}.  
\end{remark}

\section{Proof of Serfati's Theorem Part I}
\label{S:ProofOfSerfatisResult}

\noindent In this section we prove the first part of \cref{T:Serfati}; namely, \cref{e:MainBounds}. Before proceeding to the fairly long and technical proof, let us first sketch the overall strategy.

The proof hinges on the transport equation for $Y_\eps$ in \cref{e:YTransportAlmost} and the identity in \cref{C:graduCor}. Together, they allow us to relate the four quantities,
\begin{align*}
    \begin{array}{lll}
        Q_1 := \norm{Y_\eps \cdot \grad u_\eps}_{C^\al}, & &
        Q_2 := \norm{\grad u}_{L^\iny}, \\
        Q_3 := \norm{Y_\eps}_{C^\al}, & &
        Q_4 :=
        \norm{K * \dv(\omega_\eps Y_\eps)}_{C^\al}.
    \end{array}
\end{align*}

A bound on $Q_4$ comes essentially for free via \cref{L:RegularitywY}, so we may as well take it as given. From \cref{e:YTransportAlmost} we obtain a bound on $Q_3$ in terms of $Q_1$, and from \cref{C:graduCor} we obtain a bound on $Q_1$ in terms of $Q_2$ and $Q_3$. Obtaining these estimates will occupy \crefrange{S:EstimationGraduAndY}{S:EstYR}, and will also involve estimates on the flow map $\eta_\eps$.

At this point, we could close the estimates if we could obtain a bound on $Q_2$ in terms of $Q_3$. This is the subject of \cref{S:Refinedgradueps}, which will lead in \cref{S:ClosingWithGronwall} to a bound on $Q_2$ in terms of itself.  Once we have a bound on $Q_2$, we easily obtain a bound on $Q_1$ and $Q_3$. In \cref{S:Convergence}, we show that in the limit we obtain \cref{e:MainBounds}.

Note that a coarse bound on $Q_2$ in terms of $Q_1$ and $Q_3$ is easily derived (along the lines of \cref{e:gradu0Bound}), but is inadequate: It will take a great deal more work to obtain in \cref{S:Refinedgradueps} a tight enough bound that  Gronwall's lemma can be successfully applied in \cref{S:ClosingWithGronwall}.

\subsection{Preliminary estimate of $\norm{\nabla u_\eps(t)}_{L^\iny}$}\label{S:EstimationGraduAndY}

By the expression for $\grad u_\eps$ in \cref{L:graduExp},
\begin{align*}
    \norm{\grad u_\eps(t)}_{L^\iny}
        \le V_\eps(t),
\end{align*}
where
\begin{align}\label{e:VepsDef}
        V_\eps(t)
            :=
            \norm{\omega_0}_{L^\iny}
                + \norm{\PV \int \grad K(\cdot - y) \omega_\eps(t, y) \, dy}_{L^\iny}.
\end{align}
Here,
we used \cref{e:omegaNormp} to replace $\norm{\omega_\eps(t)}_{L^\iny}$ by $\norm{\omega_0}_{L^\iny}$ in the first term.

\subsection{Estimate of $\norm{\nabla \eta_\eps(t)}_{L^\iny}$ and $\norm{\nabla \eta^{-1}_\eps(t)}_{L^\iny}$}\label{S:etaetainv}

As in \cref{e:etaDef}, the defining equation for $\eta_\eps$ is
\begin{align}\label{e:etaDefRedux}
    \prt_t \eta_\eps(t, x)
        = u_\eps(t, \eta_\eps(t, x)),
            \quad
        \eta_\eps(0,x) = x,
\end{align}
or, in integral form,
\begin{align}\label{e:etaepsIntForm}
    \eta_\eps(t, x) = x + \int_0^t u_\eps(s, \eta_\eps(s, x)) \, ds.
\end{align}
This immediately implies that 
\begin{align}\label{e:gradetaBound}
    \norm{\nabla \eta_\eps(t)}_{L^\iny}
        \le \exp \int_0^t V_\eps(s) \, ds. 
\end{align}
Similarly,
\begin{align}\label{e:gradetaInvBound}
    \norm{\nabla \eta^{-1}_\eps(t)}_{L^\iny}
        \le \exp \int_0^t V_\eps(s) \, ds. 
\end{align}
The bound in \cref{e:gradetaInvBound} does not follow as immediately as that in \cref{e:gradetaBound} because the flow is not autonomous. For the details, see, for instance, the proof of Lemma 8.2 p. 318-319 of \cite{MB2002} (applying the argument there to $\grad \eta_\eps^{-1}$ rather than to $\eta_\eps^{-1}$).

Applying \cref{L:Gronwall} to \cref{e:etaepsIntForm} yields
\begin{align*}
    \abs{x - y} \exp \pr{-\int_0^t V_\eps(s) \, ds}
        \le \abs{\eta_{\eps}(t, x) - \eta_{\eps}(t, y)}
        \le \abs{x - y} \exp \int_0^t V_\eps(s) \, ds.
\end{align*}
Hence, for any $\delta_0 > 0$,
\begin{align}\label{e:NbdExpansion}
    \Ndelta{\eta_\eps(t, \Sigma)}{\delta_t^-}
        \subseteq \eta_\eps(t, \Ndelta{\Sigma}{\delta_0})
        \subseteq \Ndelta{\eta_\eps(t, \Sigma)}{\delta_t^+},
            \quad
        \delta_t^\pm := \delta_0 \exp \pr{\pm \int_0^t V_\eps(s) \, ds}.
\end{align}

\subsection{Estimate of $Y_\eps$ and $R_{\epsilon}$}\label{S:EstYR}
%
%
Taking the inner product of \cref{e:TwoTransports}$_1$ with $Y_\eps(t, \eta_\eps(t, x))$ gives
\begin{align*}
    \diff{}{t} Y_\eps(t, \eta_\eps(t, x)) \cdot Y_\eps(t, \eta_\eps(t, x))
        = (Y_\eps \cdot \grad u_\eps)(t, \eta_\eps(t, x)) \cdot Y_\eps(t, \eta_\eps(t, x)).
\end{align*}
The left-hand side equals
\begin{align*}
    \frac{1}{2} \diff{}{t} \abs{Y_\eps(t, \eta_\eps(t, x))}^2
\end{align*}
so
\begin{align*}
    \abs{\diff{}{t} \abs{Y_\eps(t, \eta_\eps(t, x))}^2}
        &\le 2 \norm{\grad u_\eps(t, \eta_\eps(t, \cdot))}_{L^\iny}
            \abs{Y_\eps(t, \eta_\eps(t, x))}^2 \\
        &= 2 \norm{\grad u_\eps(t)}_{L^\iny}
            \abs{Y_\eps(t, \eta_\eps(t, x))}^2
        \le 2 V_\eps(t)
            \abs{Y_\eps(t, \eta_\eps(t, x))}^2.
\end{align*}
It follows that 
\begin{align*}
    \diff{}{t} \abs{Y_\eps(t, \eta_\eps(t, x))}^2
        \le 2 V_\eps(t)
            \abs{Y_\eps(t, \eta_\eps(t, x))}^2.
\end{align*}
Similarly, 
\begin{align*}
    \diff{}{t} \abs{Y_\eps(t, \eta_\eps(t, x))}^2
       \ge - 2 V_\eps(t)
            \abs{Y_\eps(t, \eta_\eps(t, x))}^2.
\end{align*}
Integrating in time and applying \cref{L:Gronwall} gives
\begin{align*}
    \abs{Y_0(x)}
            e^{- \int_0^t \norm{\grad u_\eps(s)}_{L^\iny} \, ds}
        \le \abs{Y_\eps(t, \eta_\eps(t, x))}
        \le \abs{Y_0(x)}
            e^{\int_0^t \norm{\grad u_\eps(s)}_{L^\iny} \, ds}.
\end{align*}
Taking the $L^\iny$ norm in $x$, we conclude that
\begin{align}\label{e:YepsLinfAbove}
    \norm{Y_\eps(t)}_{L^\iny}
        \le \norm{Y_0}_{L^\iny} e^{\int_0^t V_{\epsilon}(s) \, ds}.
\end{align}
Also, for any measurable set $\Lambda \subseteq \R^2$,
\begin{align}\label{e:YepsBelow}
    \norm{Y_\eps(t)}_{\inf(\eta_\eps(t, \Lambda))}
        \ge \norm{Y_0}_{\inf(\Lambda)} e^{-\int_0^t V_{\epsilon}(s) \, ds}.
\end{align}

The estimate for $R_\eps$ corresponding to \cref{e:YepsLinfAbove} is, using \cref{L:R0Decomp},
\begin{align}\label{e:RepsLinfAbove}
    \norm{R_\eps(t)}_{L^\iny}
        \le C(\omega_0, Y_0) e^{\int_0^t V_{\epsilon}(s) \, ds}.
\end{align}

Integrating \cref{e:TwoTransports}$_1$ in time and substituting $\eta^{-1}_{\eps}(t, x)$ for $x$ yields
\begin{align*} 
    Y_\eps(t, x)
        = Y_0(\eta^{-1}_\eps(t, x))
            + \int_0^t (Y_\eps \cdot \grad u_\eps) (s, \eta_\eps(s,\eta^{-1}_\eps(t, x))) \, ds. 
\end{align*}
Taking the $\dot{C}^\alpha$ norm and applying \cref{e:CdotIneq}$_1$, we have 
\begin{align*} 
    \norm{Y_\eps(t)}_{\dot{C}^\alpha}
        &\le \norm{Y_0}_{\dot{C}^\alpha}
            \norm{\nabla \eta^{-1}_\eps(t)}^\alpha_{L^\iny}
            + \int_0^t \norm{(Y_\eps \cdot \nabla u_\eps)(s)}_{\dot{C}^\alpha}
            \norm{\nabla(\eta_\eps(s,\eta^{-1}_\eps(t, x)))}^\alpha_{L^\iny} \, ds.
\end{align*}

Now, by \cref{C:graduCor}, we have
\begin{align*} 
    Y_\eps \cdot \nabla u_\eps(s, x)
        = \PV \int
                &\grad K(x - y) \omega_\eps(s, y)
                \brac{Y_\eps(s, x)-Y_\eps(s, y)} \, dy \\
            &+ K * \dv (\omega_\eps Y_\eps)(s, x)
        =: \text{I} + \text{II}
\end{align*}
with
\begin{align*}
    \norm{\text{I}}_{C^\alpha}
        &\le C \norm{Y_\eps(s)}_{C^\alpha} V_\eps(s).
\end{align*}
By \cref{L:EquivalentConditions,L:RegularitywY}, we have
\begin{align*}
    \norm{\text{II}}_{C^\alpha}
        \le C_\al \exp\int_0^s V_\eps(\tau) \, d \tau.
\end{align*}
It follows that
\begin{align}\label{e:YgraduepsBound}
    \norm{Y_\eps \cdot \grad u_\eps(t)}_{C^\al}
        \le \norm{Y_\eps(t)}_{C^\alpha} V_\eps(t)
            + C_\al \exp\int_0^t V_\eps(\tau) \, d \tau.
\end{align}

To estimate $\smallnorm{\nabla(\eta_\eps(s,\eta^{-1}_\eps(t, x)))}_{L^\iny}$, we start with
\begin{align*}
    \prt_\tau \eta_\eps(\tau, \eta^{-1}_\eps(t, x))
        = u_\eps(\tau, \eta_\eps(\tau, \eta^{-1}_\eps(t, x))),
\end{align*}
which follows from \cref{e:etaDefRedux}. Applying the spatial gradient and the chain rule gives
\begin{align*}
    \prt_\tau \grad \pr{\eta_\eps(\tau, \eta^{-1}_\eps(t, x))}
        = \nabla u_\eps(\tau, \eta_\eps(\tau, \eta^{-1}_\eps(t, x)))
            \nabla(\eta_\eps(\tau, \eta^{-1}_\eps(t, x))).
\end{align*}
Integrating in time and using
$
    \nabla(\eta_\eps(\tau, \eta^{-1}_\eps(t, x)))|_{\tau = t}
        = I^{2\times 2},
$
the identity matrix, we have
\begin{align*}
    \grad \pr{\eta_\eps(s, \eta^{-1}_\eps(t, x))}
        &= I^{2 \times 2}
            - \int_s^t \nabla u_\eps(\tau, \eta_\eps(\tau, \eta^{-1}_\eps(t, x)))
            \nabla(\eta_\eps(\tau, \eta^{-1}_\eps(t, x))) \, d \tau.
\end{align*}

By \cref{L:Gronwall}, then,
\begin{align*}
    \norm{\nabla(\eta_\eps(s,\eta^{-1}_\eps(t, x)))}_{L^\iny}
            \le \exp\int_s^t \norm{\nabla u_\eps(\tau)}_{L^\iny} \, d\tau
        \le \exp\int_s^t V_\eps(\tau) \, d\tau.
\end{align*}

These bounds with \cref{e:gradetaInvBound}, and accounting for \cref{e:YepsLinfAbove}, give
\begin{align*} 
    \begin{split}
        &\norm{Y_\eps(t)}_{C^\alpha}
            \le \norm{Y_0}_{C^\al} \exp \pr{\al \int_0^t V_\eps(s) \, ds} \\
        &\qquad\qquad
                + \int_0^t \brac{\norm{Y_\eps(s)}_{C^\alpha} V_\eps(s)
                    + C_\al \exp\int_0^s V_\eps(\tau) \, d \tau
                    }
            \exp \pr{\al \int_s^t V_\eps(\tau) \, d\tau \, ds}
            \\
        &\qquad\le (\norm{Y_0}_{C^\al} + C_\al t) \exp\int_0^t V_\eps(s) \, ds
            + \int_0^t \norm{Y_\eps(s)}_{C^\alpha} V_\eps(s)
                \brac{\exp \int_s^t V_\eps(\tau) \, d\tau} \, ds.
    \end{split}
\end{align*}
Letting
\begin{align*}
    y_\eps(t)
        = \norm{Y_\eps(t)}_{C^\alpha} \exp\brac{-\int_0^t V_\eps(s) \, ds}
\end{align*}
it follows that $y_\eps$ satisfies the inequality, 
\begin{align*}
    y_\eps(t)
        \le \norm{Y_0}_{C^\al} + C_\al t + \int_0^t V_\eps(s) y_\eps(s) \, ds.
\end{align*}
Therefore, by \cref{L:Gronwall}, we obtain
\begin{align*} 
    y_\eps(t)
        &\le \pr{\norm{Y_0}_{C^\al} + C_\al t}
            \exp \pr{\int_0^t V_\eps(s) \, ds}
        \le C_\al (1 + t)
            \exp \pr{\int_0^t V_\eps(s) \, ds}
\end{align*}
and thus, with the similar bound for $R_\eps$,
\begin{align}\label{e:YRCalphaBound}
    \norm{Y_\eps(t)}_{C^\alpha}, \; \norm{R_\eps(t)}_{C^\alpha}
        \le C_\al (1 + t)
            \exp \pr{2 \int_0^t V_\eps(s) \, ds}.
\end{align}

\subsection{Refined estimate of $\nabla u_\eps$}\label{S:Refinedgradueps}
We split the second term in $V_{\epsilon}$ in \cref{e:VepsDef} into two parts, as (see \cref{D:PValg} for the meaning of $\starp$)
\begin{align}\label{e:pvgradKomegaFirstSplit}
    \begin{split}
        \PV \int &\grad K(x - y) \omega_\eps(t, y) \, dy
            = \grad K \starp \omega_\eps(t, x) \\
            &= \grad (a_r K) \starp \omega_\eps(t, x)
                + \grad ((1 - a_r) K) \starp \omega_\eps(t, x),
    \end{split}
\end{align}
where $r \in (0, 1]$ will be chosen later (in \cref{e:rChoice}).

On the support of $\grad (1 - a_r) = - \grad a_r$, $\abs{x - y} \le 2 r$, so
\begin{align}\label{e:OnearKBound}
    \abs{\grad ((1 - a_r) K)}
        \le \abs{(1 - a_r) \grad K} + \abs{\grad a_r \otimes K)}
        \le C \abs{x - y}^{-2}.
\end{align}
Hence, one term in \cref{e:pvgradKomegaFirstSplit} is easily bounded by 
\begin{align*}
    &\abs{\grad ((1 - a_r) K) \starp \omega_\eps(t, x)}
        \le C \int_{B_r^C(x)} \abs{x - y}^{-2} \abs{\omega_\eps(t, y)} \, dy \\
        &\qquad
        \le C \int_r^1 \frac{\norm{\omega_\eps}_{L^\iny}}{\rho^2} \rho \, d \rho
        + C \smallnorm{\abs{x - \cdot}^{-2}}_{L^\iny(B_1^C(x))}
            \norm{\omega_{\eps, 0}}_{L^1} \\
        &\qquad
        \le -C \log r \norm{\omega_0}_{L^\iny}
            + C \norm{\omega_0}_{L^1}
        \le C (-\log r + 1) \norm{\omega_0}_{L^1 \cap L^\iny}.
\end{align*}

For the other term in \cref{e:pvgradKomegaFirstSplit}, we decompose the vorticity
as follows. Let $\delta_0$ be as in \cref{e:Assumption}$_2$ and take a smooth function $\chi$ such that $\chi = 1$ on $\Ndelta{\Sigma}{\delta_0/4}$ and $\chi = 0$ on $\Ndelta{\Sigma}{\delta_0/2}^C$ (see \cref{D:Sigmadelta}). Let $b$ be a smooth function such that $b = 1$ on $\Ndelta{\Sigma}{3 \delta_0/4}$ and $b = 0$ on $\Ndelta{\Sigma}{\delta_0}^C$. 
Then $\abs{Y_0} \ge c$ on $\supp b$ by \cref{e:Assumption}$_2$, so
\begin{align}\label{e:Y0onsuppb}
    \abs{Y_0(\eta_\eps^{-1})} \ge c
        \text{ on } \supp b(\eta_\eps^{-1}) \supseteq \supp \chi(\eta_\eps^{-1}).
\end{align}

We now let
\begin{align*} 
    \omega^1_{0, \eps} = \chi \omega_{0, \eps},
        \quad
    \omega^2_{0, \eps} = (1 - \chi) \omega_{0, \eps},
\end{align*}
and then let
\begin{align} \label{e:DecompositionOfVorticity}
\begin{split}
    \omega^1_\eps(t, x)
        = \omega^1_{0, \eps} \pr{\eta^{-1}_\eps(t, x)}, \qquad
    \omega^2_\eps(t, x)
        = \omega^2_{0, \eps} \pr{\eta^{-1}_\eps(t, x)}.
  \end{split}
\end{align}

The vorticity $\omega^1_\eps$ is the ``bad vorticity,'' in that it is transported from a neighborhood of the set $\Sigma$ on which the initial vorticity fails to be $C^\al$. By contrast, $\omega^2_\eps$ is the ``good vorticity'' since
for all sufficiently small $\eps > 0$, we have
\begin{align}\label{e:omega20epsBound}
    \norm{\omega^2_{0, \eps}}_{C^\al}
        \le C \norm{\omega_0}_{C^\al(\R^2 \setminus \Sigma)}
        = C(\omega_0).
\end{align}

By \cref{e:NbdExpansion},
\begin{align*} 
    \supp \omega^1_\eps \subseteq \Ndelta{\eta_\eps(t, \Sigma)}{\delta/2},
        \quad
    \delta = \delta_t^+ := \delta_0 \exp \int_0^t V_\eps(s) \, ds.
\end{align*}

Now we split the other term in \cref{e:pvgradKomegaFirstSplit} into three parts, as
\begin{align*} 
    \begin{split}
        \grad (a_r K) \starp \omega_\eps
            &=(1 - b)(\eta^{-1}_\eps)
                \grad (a_r K) \starp \omega^1_\eps
                + (1 - b)(\eta^{-1}_\eps)
                \grad (a_r K) \starp \omega^2_\eps \\
            &\qquad\qquad
                + b(\eta^{-1}_\eps)
                 \grad (a_r K) \starp \omega_\eps \\
        &=: \text{III}_1 + \text{III}_2 + \text{III}_3.
    \end{split}
\end{align*}
Now choose
\begin{align}\label{e:rChoice}
    r
        = \min \bigset{
            1, \frac{\delta_0}{8} \exp \pr{- C' \int_0^t V_\eps(s) \, ds}
            },
\end{align}
leaving the choice of the constant $C' > 1$ until later (see \cref{e:CprimeChoice}).
We have $r < \delta_t^-/8$, where $\delta_t^-$ is defined in \cref{e:NbdExpansion}, so $\grad (a_r K) \starp \omega^1_\eps$ is supported on $\Ndelta{\eta_\eps(t, \Sigma)}{\delta}$ for all $\eps < \delta/8$. Hence, $\text{III}_1 = 0$.

Noting that the bound in \cref{e:OnearKBound} applies also to $\abs{\grad (a_r K)}$, we have
\begingroup
\allowdisplaybreaks
\begin{align*}
    \begin{split}
        \abs{\lim_{h \to 0} \text{III}_2}
            &= \abs{(1 - b)(\eta^{-1}_\eps) \lim_{h \to 0} \int_{\R^2}
                \grad (a_r K)(x - y)
                    \brac{\omega^2_\eps(y)-\omega^2_\eps(x)} \, dy} \\
            &\le C \norm{\omega^2_\eps(t)}_{\dot{C}^\alpha}
                \int_{\R^2} \abs{x - y}^\alpha
                \abs{\grad (a_r K)(x - y)} \, dy \\
            &\le C \norm{\omega^2_{0, \eps}}_{\dot{C}^\alpha}
                \norm{\nabla(\eta^{-1}_\eps)}^\alpha_{L^\iny}
                \int_{\abs{x - y} \le 2r} \abs{x - y}^{\alpha-2} \, dy \\
            &\le C(\omega_0) \al^{-1}
                \norm{\nabla \eta^{-1}_\eps }^\alpha_{L^\iny} r^\alpha.
    \end{split}
\end{align*}
\endgroup
In the second inequality we used \cref{e:CdotIneq}$_1$ and in the third we used \cref{e:omega20epsBound}.

To estimate $\text{III}_3$, we will find it slightly more convenient to use $\grad \Cal{F}$ in place of $K = \grad^\perp \Cal{F}$, the norms that result being identical. Letting $\mu_{rh}$ be as in \cref{D:Radial}, by virtue of \cref{L:ConvEq}, we can write
\begin{align*} 
    \abs{\text{III}_3}
        = \abs{b(\eta^{-1}_\eps) \lim_{h \to 0} \grad (\mu_{hr} K) \starp \omega_\eps}
        = \lim_{h \to 0} \abs{b(\eta^{-1}_\eps) B},
\end{align*}
where
\begin{align*}
    B = \grad \brac{\mu_{rh}
        \grad \mathcal{F}} * \omega_\eps.
\end{align*}

Because $\grad \brac{\mu_{rh} \grad \mathcal{F}}$ is not in $L^1$ uniformly in $h > 0$, we cannot estimate $B$ in $L^\iny$ directly. Instead, we will apply \cref{L:SerfatiLemma1} with
\begin{align*}
    M
        =
        \begin{pmatrix}
            Y_\eps & \pr{(\grad \eta_\eps)^T Y_0^\perp} \circ \eta^{-1}_\eps
        \end{pmatrix},
\end{align*}
so that $M_1 = Y_\eps$.

From
\begin{align*} 
\begin{split}
    Y_\eps \circ \eta^\eps
        &= \grad \eta_\eps \matrix{Y_0^1}{Y_0^2}
        = \matrix
            {\prt_1 \eta^1_\eps & \prt_2 \eta^1_\eps}
            {\prt_1 \eta^2_\eps & \prt_2 \eta^2_\eps} 
        \matrix{Y_0^1}{Y_0^2}
    = \matrix
        {Y^1_0 \prt_1 \eta^1_\eps + Y^2_0 \prt_2 \eta^1_\eps}
        {Y^1_0 \prt_1 \eta^2_\eps + Y^2_0 \prt_2 \eta^2_\eps}, \\
    (\nabla \eta_\eps)^T Y_0^\perp
        &= \matrix
            {\prt_1 \eta^1_\eps & \prt_1 \eta^2_\eps}
            {\prt_2 \eta^1_\eps & \prt_2 \eta^2_\eps} 
        \matrix{-Y^2_0}{Y^1_0}
    = \matrix
        {-Y^2_0 \prt_1 \eta^1_\eps + Y^1_0 \prt_1 \eta^2_\eps}
        {-Y^2_0 \prt_2 \eta^1_\eps + Y^1_0 \prt_2 \eta^2_\eps},
 \end{split}
\end{align*}
we have
\begin{align*}
    M
        = \matrix
            {Y^1_0 \prt_1 \eta^1_\eps + Y^2_0 \prt_2 \eta^1_\eps
                & -Y^2_0 \prt_1 \eta^1_\eps + Y^1_0 \prt_1 \eta^2_\eps}
            {Y^1_0 \prt_1 \eta^2_\eps + Y^2_0 \prt_2 \eta^2_\eps
                & -Y^2_0 \prt_2 \eta^1_\eps + Y^1_0 \prt_2 \eta^2_\eps} 
        \circ \eta^{-1}_\eps.
\end{align*}
A direct computation yields
\begin{align*}
    \det M(t, x)
        = \abs{Y_0}^2\pr{\eta^{-1}_\eps(t, x)}
            \det \nabla \eta_\eps \pr{\eta^{-1}_\eps(t, x)}
        = \abs{Y_0}^2\pr{\eta^{-1}_\eps(t, x)},
\end{align*}
since $\det \nabla \eta_\eps (\eta^{-1}_\eps(t, x)) = 1$.

For the rest of the analysis of the matrix $B$, we restrict our analysis to $\supp b(\eta^{-1}_\eps)$, on which III$_3$ is supported. By \cref{e:Y0onsuppb}, then, we have
\begin{align}\label{e:detMBelow}
    \det M=\abs{Y_0}^2\left(\eta^{-1}_\eps(t, x)\right)
        \ge c^{2}
        > 0.
\end{align}
Hence, $M$ is invertible, and applying \cref{L:SerfatiLemma1}, we have
\begin{align*}
    \begin{split}
        \abs{B}
            \le C \brac{\frac{\norm{Y_\eps}_{L^\iny}
                    + \norm{Y_0}_{L^\iny} \norm{\nabla \eta_\eps}_{L^\iny}}
                    {c^2}
                    \abs{BM_1}
                + \abs{\tr B}}.
    \end{split}
\end{align*}

\Ignore{ 
\begin{align*}
    \begin{split}
        \abs{B}
            \le C \brac{\frac{\brac{\norm{Y_\eps}^{5}_{L^\iny}
                + \norm{\nabla \eta_\eps }^{5}_{L^\iny}}\abs{BM_1 }}
                    {c^2\norm{Y_\eps(t)}^2_{\inf}}
                + \abs{\tr B}}.
    \end{split}
\end{align*}
} 

We now compute $\tr B$. We have,
\begin{align*} 
\begin{split}
    \tr B
        &= \brac{\prt_1 \mu_{rh} \prt_1 \mathcal{F}} * \omega_\eps
            + \brac{\prt_2 \mu_{rh} \prt_2 \mathcal{F}} * \omega_\eps
            + \brac{\mu_{rh} \Delta \mathcal{F}} * \omega_\eps \\
        &= \brac{\prt_1 \mu_{rh} \prt_1 \mathcal{F}} * \omega_\eps
            + \brac{\prt_2 \mu_{rh} \prt_2 \mathcal{F}} * \omega_\eps,
\end{split}
\end{align*}
using $\Delta \mathcal{F}=\delta_0$ and $\mu_{rh}(0)=0$ to remove the last term.

But, referring to \cref{R:Radial}, for $j = 1, 2$, we have
\begin{align*}
    \begin{split}
        \abs{\brac{\prt_j \mu_{rh} \prt_j \mathcal{F}} * \omega_\eps}
            &\le \frac{C}{r} \int_{r < \abs{x - y} < 2r}
                \frac{\abs{\omega_\eps(t, y)}}{\abs{x - y}} \, dy
            + \frac{C}{h} \int_{h < \abs{x - y} < 2h}
                \frac{\abs{\omega_\eps(t, y)}}{\abs{x - y}} \, dy \\
            &\le \frac{C}{r} \int_r^{2r} \frac{\norm{\omega_\eps(t)}_{L^\iny}}{\rho}
                \rho \, d \rho
            + \frac{C}{h} \int_h^{2h} \frac{\norm{\omega_\eps(t)}_{L^\iny}}{\rho}
                \rho \, d \rho \\
            &= C \norm{\omega_\eps(t)}_{L^\iny}
    \end{split}
\end{align*}
so that
\begin{align*}
    \lim_{h \to 0} \abs{\tr B}
        \le C \norm{\omega_0}_{L^\iny}.
\end{align*}

We next estimate $\abs{BM_1}$. Because 
\begin{align*}
    B
        = \matrix
            {\prt_1 \brac{\mu_{rh} \prt_1 \mathcal{F}} * \omega_\eps
                & \prt_2  \brac{\mu_{rh} \prt_1 \mathcal{F}} * \omega_\eps}
            {\prt_1 \brac{\mu_{rh} \prt_2  \mathcal{F}} * \omega_\eps
                & \prt_2  \brac{\mu_{rh} \prt_2  \mathcal{F}} * \omega_\eps}
\end{align*}
we have 
\begin{align*} 
    BM_1
        = \matrix{F_1}{F_2}
        := \matrix
            {(\prt_1 \brac{\mu_{rh} \prt_1 \mathcal{F}} * \omega_\eps) Y^1_\eps
                + (\prt_2  \brac{\mu_{rh} \prt_1 \mathcal{F}} * \omega_\eps)Y^2_\eps}
            {(\prt_1 \brac{\mu_{rh} \prt_2  \mathcal{F}} * \omega_\eps) Y^1_\eps
                + (\prt_2  \brac{\mu_{rh} \prt_2 \mathcal{F}} * \omega_\eps) Y^2_\eps}.
\end{align*}
We now decompose $F_1$ and $F_2$ into two parts as $F_k = d_k + e_k$, where
\begin{align*}
    d_k
        &= \sum_{j = 1}^2 (\prt_j \brac{\mu_{rh} \prt_k \mathcal{F}} * \omega_\eps)
                    Y^j_\eps
        - \prt_j \brac{\mu_{rh} \prt_k \mathcal{F}} * (\omega_\eps Y^j_\eps), \\ 
    e_k
        &= \prt_1 \brac{\mu_{rh} \prt_k \mathcal{F}} * (\omega_\eps Y^1_\eps)
        + \prt_2  \brac{\mu_{rh} \prt_k \mathcal{F}} * (\omega_\eps Y^2_\eps) = \dv \bigpr{\mu_{rh} \prt_k \mathcal{F} * (\omega_\eps Y_\eps)}.
\end{align*} 
\Ignore { 
$d_{i}$ is of the form
\begin{align*}
    (\prt_l \brac{\mu_{rh} \prt_k \mathcal{F}} * \omega_\eps) Y^{m}_\eps
        - \prt_l \brac{\mu_{rh} \prt_k \mathcal{F}} * (\omega_\eps Y^{m}_\eps) 
\end{align*}
and $e_{i}$ is in a divergence form
\begin{align*}
    \prt_1 \brac{\mu_{rh} \prt_{l} \mathcal{F}} * (\omega_\eps Y^1_\eps)
        + \prt_2  \brac{\mu_{rh} \prt_{l} \mathcal{F}} * (\omega_\eps Y^2_\eps),
    \quad l = 1, 2.
\end{align*}
} 
As we can see from \cref{R:Radial}, $\abs{\mu_{rh} \grad \mathcal{F}(x - y)} \le C \abs{x - y}^{-2}$, so
\begingroup
\allowdisplaybreaks
\begin{align*}
    \begin{split}
        \sum_{k = 1, 2} \abs{\lim_{h \to 0} d_k}
            &\le 2\abs{\lim_{h \to 0} \int_{\R^2}
                \nabla \brac{\mu_{rh} \nabla \mathcal{F}}(x - y) (Y_\eps(x)-Y_\eps(y))
                    \omega_\eps(y) \, dy}\\ 
            &\le C \norm{Y_\eps(t)}_{\dot{C}^\alpha} \norm{\omega_\eps(t)}_{L^\iny}
                \int_{\abs{x - y} \le 2r} \abs{x - y}^{\alpha-2} \, dy  \\
            &\le C \al^{-1} \norm{Y_\eps(t)}_{C^\alpha} \norm{\omega_0}_{L^\iny} r^\alpha
    \end{split}
\end{align*}
\endgroup
and
\begingroup
\allowdisplaybreaks
\begin{align} \label{e:ekBound}
    \begin{split}
        \sum_{k = 1, 2} \abs{\lim_{h \to 0} e_k}
            &\le 2 \abs{\lim_{h \to 0} \int_{\R^2} \brac{\mu_{rh}\nabla\mathcal{F}}(x - y)
                \dv (\omega_\eps Y_\eps)(y) \, dy} \\
        &= 2 \abs{\lim_{h \to 0} \int_{\R^2}
            \brac{\mu_{rh}\nabla\mathcal{F}}
                \big[\dv (\omega_\eps Y_\eps)(y)- \dv (\omega_\eps Y_\eps)(x)\big] \, dy} \\
        &= 2 \abs{\lim_{h \to 0} \int_{\R^2}
            \brac{\mu_{rh}\nabla\mathcal{F}}
                \big[\dv R_{\epsilon}(y)- \dv R_{\epsilon}(x)\big] \, dy} \\
        &= 2 \abs{\lim_{h \to 0} \int_{\R^2}
            \grad \brac{\mu_{rh}\nabla\mathcal{F}}
                \big[R_{\epsilon}(y)- R_{\epsilon}(x)\big] \, dy} \\
        &\le C \norm{R_\eps(t)}_{\dot{C}^\alpha}
                \int_{\abs{x - y} \le 2r} \abs{x - y}^{\alpha-2} \, dy  \\
        &\le C \al^{-1}\norm{R_\eps(t)}_{C^\alpha}  r^\alpha,
    \end{split}
\end{align}
\endgroup
where we used \cref{e:wYREquiv} and the regularity of $R_{\eps}$ in (\ref{e:YRCalphaBound}).
Thus,
\begin{align}\label{e:BFinalEst}
    \begin{split}
    \lim_{h \to 0} \abs{B}
        \le C &\al^{-1} \frac{
                    \norm{Y_\eps}_{L^\iny}
                        + \norm{Y_0}_{L^\iny} \norm{\nabla \eta_\eps }_{L^\iny}
                        }
                    {c^2}
                    \pr{
                        \norm{Y_\eps}_{C^\alpha} \norm{\omega_0}_{L^\iny}
                        + \norm{R_\eps}_{C^\alpha}}
                        r^\alpha \\
            &\qquad\qquad\qquad\qquad
                + C \norm{\omega_0}_{L^\iny}.
    \end{split}
\end{align}

Collecting all the bounds we have obtained so far, we conclude that
\begin{align}  \label{e:VepsBound}
    \begin{split}
        V_\eps(t)
            &\le C (1 - \log r) \norm{\omega_0}_{L^1 \cap L^\iny}
                + C(\omega_0) \al^{-1}
                \norm{\nabla \eta^{-1}_\eps(t)}^\alpha_{L^\iny} r^\alpha \\
                &\;
                + \frac{C}{c^2 \al}
                \frac{\norm{Y_\eps(t)}_{L^\iny}
                + \norm{Y_0}_{L^\iny} \norm{\nabla \eta_\eps(t)}_{L^\iny}}
                    {\norm{Y_\eps(t)}_{\inf(\supp b(\eta_\eps^{-1}))}}
                    \left(\norm{\omega_0}_{L^\iny}
                    \norm{Y_\eps(t)}_{C^\alpha}
                    + \norm{R_\eps(t)}_{C^\alpha} \right)r^\alpha.
    \end{split}
\end{align}

\subsection{Closing the estimates using Gronwall's lemma}\label{S:ClosingWithGronwall}
We have, from \cref{e:rChoice}, that
\begin{align*}
    1 - \log r
        &\le \abs{1 - \log \frac{\delta_0}{8}} + C' \int_0^t V_\eps(s) \, ds
        \le C(Y_0) + C' \int_0^t V_\eps(s) \, ds, \\
    r^\al
        &\le \frac{\delta_0^\al}{8^\al} \exp \pr{-C' \al \int_0^t V_\eps(s) \, ds}
        \le C(Y_0) \exp \pr{-C' \al \int_0^t V_\eps(s) \, ds}.
\end{align*}
Returning to \cref{e:VepsBound}, then, these bounds on $1 - \log r$ and $r^\al$, along with the bounds in \cref{e:gradetaBound,e:gradetaInvBound,e:YepsLinfAbove,e:RepsLinfAbove,e:YepsBelow}, and \cref{e:YRCalphaBound}, yield the estimate,
\begin{align*}
    V_\eps(t)
        &\le C(\omega_0, Y_0) + C' C(\omega_0, Y_0)
            \int_0^t V_\eps(s) \, ds
                + \frac{C(\omega_0, Y_0)}{\al}
                \exp \pr{\al(1 - C') \int_0^t V_\eps(s) \, ds} \\
        &\qquad
            + C_\al(1 + t)
                \exp \pr{\al(4 - C' \al) \int_0^t V_\eps(s) \, ds}
        \\ 
        &\le {C_\al (1 + t)} + \frac{C(\omega_0, Y_0)}{\al}
            \int_0^t V_\eps(s) \, ds
\end{align*}
as long as we choose
\begin{align}\label{e:CprimeChoice}
    C'
        = 4 \al^{-1}.
\end{align}
We note that, as required, $C' > 1$.

By \cref{L:Gronwall}, we conclude that
\begin{align*}
    \norm{\grad u_\eps(t)}_{L^\iny}
        \le V_\eps(t)
        \le C_\al (1 + t) e^{C(\omega_0, Y_0) \al^{-1} t}
        \le C_\al e^{c_\al t}.
\end{align*}
If $\al > 1/2$, we can apply the above bound with $1/2$ in place of $\al$, eliminating the factor of $(1 - \al)^{-1}$ that appear in $C_\al$. This gives
\begin{align}\label{e:graduepsLinffBound}
    \norm{\grad u_\eps(t)}_{L^\iny}
        \le V_\eps(t)
        = C(\omega_0, Y_0) \al^{-1} e^{c_\al t}
        \le c_\al e^{c_\al t}.
\end{align}

Then
\begin{align*}
    \int_0^t V_\eps(s) \, ds
        < \frac{c_\al}{c_\al} e^{c_\al t}
        = e^{c_\al t}
\end{align*}
so by virtue of \cref{e:YRCalphaBound},
\begin{align}\label{e:YRepsBounds}
    \begin{split}
    & \norm{Y_\eps(t)}_{C^\alpha}, \norm{R_\eps(t)}_{C^\alpha} 
        \le C_\al e^{e^{c_\al t}}. \\
    \end{split}
\end{align}
It follows from \cref{e:YgraduepsBound} that $\norm{Y_\eps \cdot \grad u_\eps(t)}_{C^\al}$ has this same bound.

\subsection{Convergence of approximate solutions}\label{S:Convergence}

\noindent In this section we show that in the limit as $\eps \to 0$, the estimates in \cref{e:MainBounds} hold for the solution $\omega$ to \cref{e:Eomega} with velocity $u$. For the delicate parts of the proof we follow the argument on pages 105-106 of \cite{C1998}, but beginning in a slightly different manner.

That the approximate solutions $(u_\eps)$ converge to the solution $u$ for bounded initial vorticity is by now classical (see Section 8.2 of \cite{MB2002}, for instance). Because $(\grad u_\eps)$ is uniformly bounded in $L^\iny$, however, we can obtain stronger convergence, as follows.

Fix $T > 0$. It follows from \cite{Serfati1995A}, under the assumption only that the initial vorticity and velocity are both in $L^\iny$, that $\grad p_\eps$ is bounded in $L^\iny([0, T] \times \R^2)$ (see \cite{KBounded2014} for details). Then since
$
    \prt_t u_\eps = - u_\eps \cdot \grad u_\eps - \grad p_\eps
$
it follows that $(\prt_t u_\eps)$ is bounded uniformly in $\eps$. So, in fact, $(u_\eps)$ is a bounded equicontinuous family on $[0, T] \times L$ for any compact subset $L$ of $\R^2$ with a Lipschitz modulus of continuity in time and space. By the Arzela-Ascoli theorem, a subsequence converges uniformly on $[0, T] \times L$ to some $\ol{u}$, with $\ol{u} = u$ by the uniqueness of limits. This also shows that $\grad u \in L^\iny([0, T] \times \R^2)$, with the bound in \cref{e:MainBounds}. (Here and in what follows, we take subsequences as necessary without relabeling the indices.) It then follows that $\grad \eta$, and for that matter $\grad \eta_\eps$, are in $L^\iny([0, T] \times \R^2)$.

It follows from \cref{L:uUniformDecay}, which we prove below, that $u_\eps$ decays in space uniformly in time and in $\eps$.  This same uniform decay applies to $u$ by the convergence we showed above. Thus, in fact, $u_\eps \to u$ in $L^\iny([0, T] \times \R^2)$ since we can control the size of $\abs{u_\eps - u}$ outside of a sufficiently large compact subset. By interpolation it follows that $u_\eps \to u$ in $L^\iny(0, T; C^\beta(\R^2))$ for all $\beta < 1$.

From \cref{e:etaepsIntForm}, then, we can estimate,
\begin{align*} 
    \begin{split}
    &\abs{\eta_\eps(t, x) - \eta(t, x)} \\
        &\qquad
        \le \int_0^t \abs{u_\eps(s, \eta_\eps(s, x))
            - u(s, \eta_\eps(s, x)} \, ds
            + \int_0^t \abs{u(s, \eta_\eps(s, x) - u(s, \eta(s, x)} \, ds \\
        &\qquad
        \le \int_0^t \norm{u_\eps(s) - u(s)}_{L^\iny}
            + \int_0^t \norm{\grad u}_{L^\iny([0, T] \times \R^2)}
                \abs{\eta_\eps(s, x) - \eta(s, x)} \, ds.
    \end{split}
\end{align*}
It follows from \cref{L:Gronwall} that $\eta_\eps - \eta \to 0$ in $L^\iny([0, T] \times \R^2)$ and, similarly, that $\eta_\eps^{-1} - \eta^{-1} \to 0$ in $L^\iny([0, T] \times \R^2)$. By interpolation it follows that $\eta_\eps - \eta \to 0$ in $L^\iny(0, T; C^\beta(\R^2))$ for all $\beta < 1$. 


We now argue along the lines of pages 105-106 of \cite{C1998}.

We can write \cref{e:Yeps} as
\begin{align*}
	Y_0 \cdot \grad \eta_\eps
    	= Y_\eps \circ \eta_\eps.
\end{align*}
By \cref{e:CdotIneq}$_2$ and \cref{e:YRepsBounds}, then, $Y_0 \cdot \grad \eta_\eps$ is uniformly bounded in $L^\iny(0, T; C^\al(\R^2))$. But $C^\al(\R^2)$ is compactly embedded in $C^\beta(\R^2)$ for all $\beta < \al$ so a subsequence of $(Y_0 \cdot \grad \eta_\eps)$ converges in $L^\iny(0, T; C^\beta(\R^2))$ to some $f$ for all $\beta < \al$, and it is easy to see that $f \in L^\iny(0, T; C^\al(\R^2))$.

To show that $f = Y_0 \cdot \grad \eta$, we need only show convergence of $Y_0 \cdot \grad \eta_\eps \to Y_0 \cdot \grad \eta$ in some weaker sense. To do this, observe that 
\begin{align*}
    (Y_0 \cdot \grad \eta_\eps)^j
    	= Y_0 \cdot \grad \eta_\eps^j
        = \dv(\eta_\eps^j Y_0) - \eta_\eps^j \dv Y_0.
\end{align*}
But $\eta_\eps - \eta \to 0$ in $L^\iny(0, T; C^\beta(\R^2))$ for all $\beta < 1$ as we showed above so $\eta_\eps^j Y_0 - \eta^j Y_0 \to 0$ in $L^\iny(0, T; C^\al(\R^2))$. And, by assumption \cref{e:Assumption}$_4$, $\eta_\eps^j \dv Y_0 - \eta^j \dv Y_0 \to 0$ in $L^\iny(0, T; C^\al(\R^2))$. By the definition of negative \Holder spaces in \cref{D:HolderSpaces} it follows that $Y_0 \cdot \grad \eta_\eps \to Y_0 \cdot \grad \eta$ in $L^\iny(0, T; C^{\al - 1}(\R^2))$.
Hence, $f = Y_0 \cdot \grad \eta$, so we can conclude that $Y_0 \cdot \grad \eta \in L^\iny(0, T; C^\al(\R^2))$ and $Y_0 \cdot \grad \eta_\eps \to Y_0 \cdot \grad \eta$ in $L^\iny(0, T; C^\beta(\R^2))$ for all $\beta < \al$.

Then, since
$
    Y_\eps = (Y_0 \cdot \grad \eta_\eps) \circ \eta_\eps^{-1}
$
and
$
    Y = (Y_0 \cdot \grad \eta) \circ \eta^{-1}
$
(see \cref{e:Y,e:Yeps}),
we have,
\begin{align*}
    \begin{split}
    \norm{Y_\eps - Y}_{L^\iny}
        &\le
            \norm{(Y_0 \cdot \grad \eta_\eps) \circ \eta_\eps^{-1}
                - (Y_0 \cdot \grad \eta_\eps) \circ \eta^{-1}}_{L^\iny} \\
            &\qquad
            +
                \norm{(Y_0 \cdot \grad \eta_\eps) \circ \eta^{-1}
                - (Y_0 \cdot \grad \eta) \circ \eta^{-1}}_{L^\iny} \\
        &\le
            \norm{Y_0 \cdot \grad \eta_\eps}_{C^\al}
             \smallnorm{\eta_\eps^{-1} - \eta^{-1}}_{L^\iny}^\al
            +
                \norm{Y_0 \cdot \grad \eta_\eps
                - Y_0 \cdot \grad \eta}_{L^\iny} \\
		&\to 0 \text{ as } \eps \to 0,
    \end{split}
\end{align*}
where we used \cref{e:CdotIneq}$_1$.
Here the $L^\iny$ norms are over $[0, T] \times \R^2$ for any fixed $T > 0$. Arguing as for $Y_0 \cdot \grad \eta$, it also follows that $Y \in L^\iny(0, T; C^\al(\R^2))$ and that the bound on $Y(t)$ in \cref{e:MainBounds} holds. Then \cref{e:gradetaMainBound} follows from \cref{e:MainBounds} as in \cref{e:gradetaBound,e:gradetaInvBound}.

The proofs of \cref{e:YdivBound,e:omegaYdivBound}, which we suppress, follow much the same course as the bounds above.
Finally,
\begin{align*}
    (Y_\eps \cdot \grad u_\eps)^j
        = \dv(u_\eps^j Y_\eps) - u_\eps^j \dv Y_\eps,
\end{align*}
and given that we now know that $Y_\eps \to Y$ in $C^\beta(\R^2)$ for all $\beta < \al$ with $Y \in C^\al(\R^2)$, \cref{e:YgraduBound} can be proved much the way we proved the convergence of $Y_0 \cdot \grad \eta_\eps \to Y_0 \cdot \grad \eta$, above (taking advantage of \cref{e:YdivBound}).

This completes the proof of the first part of \cref{T:Serfati}.

\begin{remark}\label{R:WeakerAssumption}
    Had we only assumed that $\dv Y_0 \in C^{\al'}(\R^2)$ for some $\al' \in (0, \al]$
    then the argument above that showed
    $Y_0 \cdot \grad \eta_\eps \to Y_0 \cdot \grad \eta$ in
    $L^\iny(0, T; C^{\al - 1}(\R^2))$ would yield
    $Y_0 \cdot \grad \eta_\eps \to Y_0 \cdot \grad \eta$
    in $L^\iny(0, T; C^{\al' - 1}(\R^2))$. This would be sufficient to conclude
    that $f = Y_0 \cdot \grad \eta$, and the proof would proceed unchanged.
\end{remark}

\begin{lemma}\label{L:uUniformDecay}
    The approximating solutions, $u_\eps$, decay in space uniformly in time and in $\eps$.
\end{lemma}
\begin{proof}
Let $\abs{x} > 2$ and $R = \abs{x}/2$.
From \cref{e:Eomega}$_2$, we can write
\begin{align*}
    u_\eps
        = \pr{\int_{B_1(x)} + \int_{B_R(x) \setminus B_1(x)} + \int_{B_R(x)^C}}
            K(x - y) \omega_\eps(y) \, dy.
\end{align*}
Fixing $p, q$ \Holder conjugate with $p \in [1, 2)$ and applying \Holder's inequality to each of the three terms above, we have
\begin{align}\label{e:uUniformDecay}
    \begin{split}
    \abs{u_\eps(t, x)}
        &\le \left(
            \norm{K(x - \cdot)}_{L^p(B_1(x))} \norm{\omega_\eps(t)}_{L^q(B_1(x))}
            \right. \\
        &\qquad
            \left.
            + \norm{K(x - \cdot)}_{L^\iny(B_R(x) \setminus B_1(x))}
                \norm{\omega_\eps(t)}_{L^1(B_R(x) \setminus B_1(x))}
            \right) \\
        &\qquad
            + \norm{K(x - \cdot)}_{L^\iny(B_R(x)^C)} \norm{\omega_\eps(t)}_{L^1(B_R(x)^C)} \\
        &\le C
            \norm{\omega_\eps(t)}_{(L^1 \cap L^q)(B_R(x))}
                    + C R^{-1} \norm{\omega_\eps(t)}_{L^1(\R^2)} \\
        &\le C
            \norm{\omega_\eps(t)}_{(L^1 \cap L^q)(B_{\abs{x}/2}(0)^C)}
                    + C \norm{\omega_0}_{L^1(\R^2)} \abs{x}^{-1}.
    \end{split}
\end{align}

We claim that for constants $C_1, C_2 > 0$,
\begin{align*}
    \norm{\omega_\eps(t)}_{(L^1 \cap L^q)(B_{\abs{x}/2}(0)^C)}
        &\le \norm{\omega_{0, \eps}}_{(L^1 \cap L^q)(B_{\abs{x}/2 - C_{1}T}(0)^C)} \\
        &\le \norm{\omega_0}_{(L^1 \cap L^q)(B_{\abs{x}/2 - C_1 T - C_2}(0)^C)}
\end{align*}
for all $\abs{x}/2 \ge C_1 T + C_2$. The first inequality holds because $\norm{u_\eps}_{L^\iny([0, T] \times \Omega)} \le C_1 = \norm{\omega_0}_{L^1 \cap L^\iny}$ by \cref{e:uepsBound} and the vorticity is transported by the flow map. Then since $\omega_{0, \eps} = \rho_\eps * \omega_0$ and $\rho_\eps$ is supported within a ball of radius $C_2 \eps$, and we have assumed that $\eps \le 1$, the second inequality holds as well. Thus, $\omega_\eps$ decays in space uniformly in time and in $\eps$. It then follows from \cref{e:uUniformDecay} that $u_\eps$ decays in space uniformly in time and in $\eps$. 
\end{proof}

\section{Proof of Serfati's Theorem Part II}
\label{S:MatrixA}
\noindent
In this section we prove the second part of \cref{T:Serfati}, finding a matrix $A\in C^\alpha(\R^2)$ such that $\nabla u - \omega A \in C^\alpha(\R^2)$.

We start with the expression
\begin{align}\label{e:graduepsAgain}
    \grad u_\eps(x)
        = \frac{\omega_\eps(x)}{2} \matrix{0 & -1}{1 & 0}
            + \PV \int \grad K(x - y) \omega_\eps(y) \, dy
\end{align}
for $\grad u_\eps$ given by \cref{L:graduExp}. This expression may not hold in the limit as $\eps \to 0$, but to remove from $\grad u$ the discontinuities inherent in $\omega$ we will clearly need, before taking $\eps$ to zero, to subtract from $\grad u_\eps$ its antisymmetric part, $\omega_\eps A^{(1)}$, where
\begin{align*}
    A^{(1)}
        :=  \frac{1}{2}
                \matrix{0 & -1}{1 & 0}.
\end{align*}
What remains, then, is the principal value integral in \cref{e:graduepsAgain}: the symmetric part of $\grad u_\eps$. We first show that away from $\eta_\eps(t, \Sigma)$, the symmetric part of $\grad u_\eps$ has $C^\al$ regularity.

Let $\delta_0$ be as in \cref{e:Assumption}$_2$ and $r > 0$ and the cutoff function $\chi$ be as in \cref{S:Refinedgradueps}. Then 
\begin{align*}
    (1 - \chi(&\eta^{-1}_\eps(x)))\PV \int \grad K(x - y) \omega_\eps(y) \, dy \\
        &= (1 - \chi(\eta^{-1}_\eps(x)))
        \PV \int_{\R^2} \grad (a_r K)(x - y)
            \omega_\eps(y) \, dy \\
        &\qquad
        + (1 - \chi(\eta^{-1}_\eps(x)))
            \int_{\R^2} \grad ((1 - a_r) K)(x - y)\omega_\eps(y) \, dy \\
    &=: \text{I}_1 +\text{I}_{2}.
\end{align*}

By (\ref{e:LHomo}) in \cref{L:SerfatiLemma2}  applied with the kernel $L_2$ of \cref{L:SerfatiKernels}, we have
$\text{I}_{1}\in C^\al(\R^2)$ with a $C^\al$ bound that is uniform in $\eps$. (Note that in the integrals above, $y$ remains bounded away from $\eta_\eps(\Sigma, t)$ on the support of $(1 - \chi(\eta^{-1}_\eps(x)))$.) By \cref{e:omegaCalLocalBound} we have the $\eps$-independent bound,
\begin{align*}
    \norm{I_1}_{C^{\al}}
        \le \norm{\omega_0}_{C^\al(\Sigma^C)}
                e^{e^{c_\al t}}.
\end{align*}

Since the kernel $\nabla\left((1 - a_r)K\right)$ is smooth and bounded, $\text{I}_2$ is smooth, lying in $C^k(\R^2)$ for all $k$ with a norm that depends only upon $\norm{\omega_\eps}_{L^1} = \norm{\omega_0}_{L^1}$. In particular, $\norm{\text{I}_2}_{C^\al} \le C(\omega_0)$ with a norm independent of $\eps$. (Cutting off with $\chi$ was not needed to conclude this.)

We will complete the proof by establishing the following lemma:
\begin{lemma}\label{L:AMatrixCentralLemma}
We have,
\begin{align*} 
    \chi(\eta^{-1}_\eps) \PV \int_{\R^2} \grad K(x - y) \omega_\eps(y) \, dy
        =  \omega_\eps B_\eps + D_\eps,
\end{align*}
where
\begin{align*}
    B_\eps
        = \frac{\chi(\eta^{-1}_\eps)}{2 \abs{Y_\eps}^2}
                \matrix{2 Y_\eps^1 Y_\eps^2 & (Y_\eps^2)^2 - (Y_\eps^1)^2}
                {(Y_\eps^2)^2 - (Y_\eps^1)^2 & -2 Y_\eps^1 Y_\eps^2}
\end{align*}
and $D_{\eps}$ is presented in the proof. $B_\eps$ and $D_\eps$ lie in $C^\al(\R^2)$ with
\begin{align}\label{e:BepsBound}
    \norm{B_\eps(t)}_{C^\al}, \, \norm{D_\eps(t)}_{C^\al}
            &\le C_\al e^{e^{c_\al t}}.
\end{align}
\end{lemma}

Before proving \cref{L:AMatrixCentralLemma}, we show how it completes the proof of the second part of \cref{T:Serfati}. Let
\begin{align*}
    \ol{A}_\eps = A^{(1)} + B_\eps,
\end{align*}
which lies in $C^\al$ with a $C^\al$ norm that is uniform in $\eps$.
Then
\begin{align*}
    \grad u_\eps - \omega_\eps \ol{A}_\eps
        &= \PV \int \grad K(x - y) \omega_\eps(y) \, dy
            - \omega_\eps B_\eps \\
        &= \text{I}_1 + \text{I}_2
            + \chi(\eta^{-1}_\eps) \PV \int_{\R^2} \grad K(x - y) \omega_\eps(y) \, dy
            - \omega_\eps B_\eps \\
        &= \text{I}_1 + \text{I}_2
            + D_\eps
\end{align*}
lies in $C^\al$ with a $C^\al$ norm that is uniform in $\eps$. The final equality is where we used \cref{L:AMatrixCentralLemma}.

This shows that $\ol{A}_\eps$ is a candidate for our matrix $A$ (in the limit), but for aesthetic reasons we prefer to apply the cutoff function $\chi$ to $A^{(1)}$ as well, using 
\[
A_\eps = \chi(\eta^{-1}_\eps) A^{(1)} + B_\eps.
\]
This is valid, since the bound in \cref{e:omegaCalLocalBound} was established without using the matrix $A$, so we know at this point that $(1 - \chi(\eta^{-1}_\eps)) \omega_\eps$ lies in $C^\al$ with a $C^\al$ norm that is uniform in $\eps$. A simple calculation shows that
\begin{align*}
    A_\eps
        = \frac{\chi(\eta^{-1}_\eps)}{\abs{Y_\eps}^2}
                \matrix
                    {Y_\eps^1 Y_\eps^2 & - (Y_\eps^1)^2}
                    {(Y_\eps^2)^2 & -Y_\eps^1 Y_\eps^2}.
\end{align*}
The bounds in \cref{e:ABound} then follow from \cref{e:BepsBound}, the bounds on I$_1$ and I$_2$, above, and \cref{L:ProductCalBound}.

An examination of each of the components of $D_\eps$ shows, arguing as in \cref{S:Convergence}, that for some subsequence, $(\eps_k)_{k = 1}^\iny$, $D_{\eps_k} \to D$ in $L^\iny(0, T; C^\beta)$ for all $\beta < \al$ and that $D$ lies in $L^\iny(0, T; C^\al)$. Similarly, $A_{\eps_k}$ converges in $C^\beta$ for all $\beta < \al$ to
\begin{align}\label{e:AExplicit}
    A
        = \frac{\chi(\eta^{-1})}{\abs{Y}^2}
                \matrix
                    {Y^1 Y^2 & - (Y^1)^2}
                    {(Y^2)^2 & -Y^1 Y^2},
\end{align}
which lies in $C^\al$, and also in the limit \cref{e:ABound} holds.

\begin{proof}[\textbf{Proof of \cref{L:AMatrixCentralLemma}}]
The proof comes down to understanding the regularity of the principal value integral in \cref{e:graduepsAgain} on the support of $\chi(\eta_\eps^{-1})$. To do this, we forcefully inject into this integral the vector field $Y_\eps$, which characterizes the discontinuities in the vorticity field. This approach has already been used in our application of \cref{C:graduCor} in \cref{S:EstYR}. Now, however, the injection of $Y_\eps$ will be deeper. Loosely speaking, $Y_\eps$ was injected linearly in \cref{C:graduCor}; now, we will inject $Y_\eps$ quadratically.

Using the notation in \cref{D:PValg}, we can write
\begin{align}\label{e:pvInt}
    \PV \int \grad K(x - y) \omega_\eps(y) \, dy
        = \matrix
            {\prt_1 K^1 \starp \omega_\eps &
            \prt_1 K^2 \starp \omega_\eps}
            {\prt_2 K^1 \starp \omega_\eps &
            \prt_2 K^2 \starp \omega_\eps}.
\end{align}

Now, $\dv K = 0$ so 
\[
\prt_1 K^1 \starp \omega_\eps  = - \prt_2 K^2 \starp \omega_\eps.
\]
 Also,
\[
\curl K = - \dv K^\perp
    = - \dv (\grad^\perp \Cal{F})^\perp
    = \Delta \Cal{F} = \delta.
\]
But since the  $\starp$ operator avoids the origin in the integrand,
    we have 
\[
    \prt_1 K^2 \starp \omega_\eps
    = \prt_2 K^1 \starp \omega_\eps.
\]    
Thus, the $\PV$ integral in \cref{e:pvInt} is, as we know, the symmetric part of $\grad u_\eps$. Moreover,
\begin{align*}
        \chi(\eta^{-1}_\eps)
            \PV \int \grad K(x - y) \omega_\eps(y) \, dy
        &= \frac{\chi(\eta^{-1}_\eps) }{\abs{Y_\eps}^2}
            \matrix{a_{\eps} & b_{\eps}}{b_{\eps} & -a_{\eps}},
\end{align*}
where
    \begin{align*}
        a_{\eps}
            &:= \abs{Y_\eps}^2 \prt_1 K^1 \starp \omega_\eps
            = (Y_\eps^1)^2 \prt_1 K^1 \starp \omega_\eps
                + (Y_\eps^2)^2 \prt_1 K^1 \starp \omega_\eps \\
            &= Y_\eps^1 \brac{Y_\eps^1 \prt_1 K^1 \starp \omega_\eps
                    + Y_\eps^2 \prt_1 K^2 \starp \omega_\eps}
                + Y_\eps^2 \brac{Y_\eps^2 \prt_1 K^1 \starp \omega_\eps
                    - Y_\eps^1 \prt_1 K^2 \starp \omega_\eps} \\
            &= Y_\eps^1 \brac{Y_\eps^1 \prt_1 K^1 \starp \omega_\eps
                    + Y_\eps^2 \prt_2 K^1 \starp \omega_\eps}
                - Y_\eps^2 \brac{Y_\eps^2 \prt_2 K^2
                        \starp \omega_\eps
                    + Y_\eps^1 \prt_1 K^2
                        \starp \omega_\eps}
    \end{align*}
and
    \begin{align*}
        b_{\eps}
            &:=
            \abs{Y_\eps}^2 \prt_1 K^2 \starp \omega_\eps
            = (Y_\eps^1)^2 \prt_1 K^2 \starp \omega_\eps
                + (Y_\eps^2)^2 \prt_1 K^2 \starp \omega_\eps \\
            &= Y_\eps^1 \brac{Y_\eps^1 \prt_1 K^2 \starp \omega_\eps
                    + Y_\eps^2 \prt_2 K^2 \starp \omega_\eps}
                + Y_\eps^2 \brac{Y_\eps^2 \prt_1 K^2 \starp \omega_\eps
                    - Y_\eps^1 \prt_2 K^2 \starp \omega_\eps} \\
            &= Y_\eps^1 \brac{Y_\eps^1 \prt_1 K^2 \starp \omega_\eps
                    + Y_\eps^2 \prt_2 K^2 \starp \omega_\eps}
                + Y_\eps^2 \brac{Y_\eps^2 \prt_2 K^1 \starp \omega_\eps
                    + Y_\eps^1 \prt_1 K^1 \starp \omega_\eps}.
    \end{align*}
 
Observe that we can write
\begin{align*}
    a_{\eps}
        &= Y_\eps \cdot
            \matrix{Y_\eps \cdot (\grad K^1 \starp \omega_\eps)}
            {- Y_\eps \cdot (\grad K^2 \starp \omega_\eps)},
            \qquad
    b_{\eps}
        = Y_\eps \cdot
            \matrix{Y_\eps \cdot (\grad K^2 \starp \omega_\eps)}
            {Y_\eps \cdot (\grad K^1 \starp \omega_\eps)}.
\end{align*}
Defining, for vector-valued functions $v$ and $w$,
    \begin{align*}
        v \starpdot w
            := v^j \starp w^j,
    \end{align*}
we have,
    \begin{align*}
        Y_\eps \cdot &(\grad K^j \starp \omega_\eps)
            =\grad K^j \starpdot (\omega_\eps Y_\eps)+\left[ (\grad K^j \starp \omega_\eps) \cdot Y_\eps
                - \grad K^j \starpdot (\omega_\eps Y_\eps)\right].
    \end{align*}
\cref{L:graduYLikeLemma} gives
\begin{align*}
    Y_\eps\cdot \nabla K^{1}\ast \omega_\eps
        &= \frac{1}{2}\omega Y^{2} +K^{1}\ast \dv(\omega Y) 
                + \left[(\grad K^1 \starp \omega_\eps) \cdot Y_\eps
                - \grad K^1 \starpdot (\omega_\eps Y_\eps)\right], \\
    Y_\eps\cdot \nabla K^{2}\ast \omega_\eps
        &= -\frac{1}{2}\omega Y^{1} +K^{2}\ast \dv(\omega Y) 
                + \left[(\grad K^2 \starp \omega_\eps) \cdot Y_\eps
                - \grad K^2 \starpdot (\omega_\eps Y_\eps)\right].
\end{align*}
Therefore, 
\begin{align*}
    a_\eps
        &=Y^{1}_\eps ( Y_\eps\cdot \nabla K^{1}\ast \omega_\eps)
            - Y^{2}_\eps ( Y_\eps\cdot \nabla K^{2}\ast \omega_\eps)\\
        &= \omega_\eps  Y^{1}_\eps Y^{2}_\eps  \\
            &\qquad
            + Y^{1}_\eps  \left(K^{1}\ast \dv(\omega Y) 
            + (\grad K^1 \starp \omega_\eps) \cdot Y_\eps
            - \grad K^1 \starpdot (\omega_\eps Y_\eps)\right)\\
            &\qquad
            -Y^{2}_\eps  \left(K^{2}\ast \dv(\omega Y) 
            + (\grad K^2 \starp \omega_\eps) \cdot Y_\eps
            - \grad K^2 \starpdot (\omega_\eps Y_\eps)\right) \\
        & =: \omega_\eps Y^{1}_\eps Y^{2}_\eps + \bar{a}_{\eps}      
\end{align*}
and
\begin{align*}
    b_\eps
        &= Y^{1}_\eps
            (Y_\eps\cdot \nabla K^{2}\ast \omega_\eps)
            + Y^{2}_\eps (Y_\eps\cdot \nabla K^{1}\ast \omega_\eps) \\
        &= \frac{1}{2}\omega_{\eps} \left[(Y^{2}_{\eps})^{2} -(Y^{1}_{\eps})^{2}\right] \\
            &\qquad
            + Y^{1}_\eps  \left(K^{2}\ast \dv(\omega Y) 
            + (\grad K^2 \starp \omega_\eps) \cdot Y_\eps
            - \grad K^2 \starpdot (\omega_\eps Y_\eps)\right) \\
            &\qquad
                + Y^{2}_\eps \left(K^{1}\ast \dv(\omega Y) 
                + (\grad K^1 \starp \omega_\eps) \cdot Y_\eps
                - \grad K^1 \starpdot (\omega_\eps Y_\eps)\right) \\
        &=: \frac{1}{2}\omega_{\eps} \left[(Y^{2}_{\eps})^{2}
            - (Y^{1}_{\eps})^{2}\right] + \bar{b}_{\eps}.   
\end{align*}
This gives $B_{\eps}$ as stated above and
\begin{align*}
    D_\eps
        = \frac{\chi(\eta^{-1}_\eps)}{\abs{Y_\eps}^2}
                \matrix{\bar{a}_{\eps}  & \bar{b}_{\eps} }
                {\bar{b}_{\eps}  & -\bar{a}_{\eps} }.
\end{align*} 
Now,
    \begin{align*}
        (\grad K^j \starp &\omega_\eps) \cdot Y_\eps
                - \grad K^j \starpdot (\omega_\eps Y_\eps)
            = \PV \int \grad K^j(x - y) \cdot
                \brac{Y_\eps(x) - Y_\eps(y)} \omega_\eps(y) \, dy.
    \end{align*}
Applying \cref{L:SerfatiLemma2Inf} with the kernel $L_3$ of \cref{L:SerfatiKernels} and using \cref{e:graduepsLinffBound} and \cref{e:YRepsBounds}$_1$ gives
\begin{align*}
    &\norm{(\grad K^j \starp \omega_\eps) \cdot Y_\eps
            - \grad K^j \starpdot (\omega_\eps Y_\eps)}_{C^{\alpha}}
        \le C \norm{Y_\eps(t)}_{C^\alpha} V_\eps(t)
        \le C_\al e^{c_2 e^{c_1 t}}.
\end{align*}
This with (\ref{e:omegaYdivBound}) gives the the bound on $D_\eps$ in \cref{e:BepsBound}.

For the regularity of $B_\eps$, first note that
\begin{align*}
    \smallnorm{\abs{Y_\eps}^{-1}}_{\dot{C}^\al(\supp \eta^{-1})}
        \le \frac{\norm{Y_\eps}_{\dot{C}^\al}}
            {\norm{Y_\eps}_{\inf(\supp \chi(\eta^{-1}))}^2}
        \le C_\al e^{e^{c_\al t}}
\end{align*}
by \cref{e:MainBounds,e:YepsBelow,e:Y0onsuppb}. The bound on $B_\eps$ in \cref{e:BepsBound} then follows from \cref{e:CdotIneq}$_2$ and \cref{L:ProductCalBound}. 
\end{proof}

We used the following elementary lemma above:
\begin{lemma}\label{L:ProductCalBound}
    Assume that $\varphi \in C_C^\iny(\R^2)$ takes values in $[0, 1]$
    and $f \in C^\al(\R^2)$.
    Then
    \begin{align*}
        \norm{\varphi f}_{C^\al(\R^2)}
            \le \norm{f}_{L^\iny(\supp \varphi)}
                + \norm{\varphi}_{\dot{C}^\al} \norm{f}_{C^\al(\supp \varphi)}.
    \end{align*}
\end{lemma}
\OptionalDetails{
\begin{proof}
    Clearly, $\norm{\varphi f}_{L^\iny} \le \norm{f}_{L^\iny(\supp \varphi)}$, so it remains
    only to bound the homogeneous norm. But the bound on this follows from the
    observation that for any $g \in C^\al(\R^2)$,
    \begin{align*}
        \norm{g}_{C^\al(\R^2)}
            = \norm{g}_{C^\al(\supp \varphi)}.
    \end{align*}
    \Ignore { 
    
    Let $K = \supp \varphi$ and for $x, y \in \R^2$
    define
    \begin{align*}
        \Delta(x, y)
            := \frac{\abs{(\varphi f)(x) - (\varphi f)(y)}}{\abs{x - y}^\al}.
    \end{align*}
    There are four cases.
    
    If $x, y \in K$ then
    \begin{align*}
        \Delta(x, y)
            &\le \frac{\abs{\varphi(x) (f(x) - f(y))}}{\abs{x - y}^\al} 
                + \frac{\abs{f(y) (\varphi(x) - \varphi(y))}}{\abs{x - y}^\al}
            \le \norm{f}_{\dot{C}^\al(K)}
                + \norm{\varphi}_{\dot{C}^\al} \norm{f}_{L^\iny(K)}.
    \end{align*}
    
    If $x \notin K$, $y \notin K$ then $\Delta(x, y) = 0$.
    
    If $x \in K$, $y \notin K$ then
    \begin{align*}
        \Delta(x, y)
            &= \frac{\abs{\varphi(x) f(x)}}{\abs{x - y}^\al} 
            = \frac{\abs{(\varphi(x) - \varphi(y))f(x)}}{\abs{x - y}^\al}
            \le \norm{\varphi}_{\dot{C}^\al} \norm{f}_{L^\iny(K)}
    \end{align*}
    and the same inequality holds for $x \notin K$, $y \in K$.
    
    Together, these bounds yield the stated inequality.
    } 
\end{proof}
} 

%
%
\newcommand{\RadiallySymmetricExample}
{

Suppose that $\omega_0$ is radially symmetric, so that $\omega_0(x) = g(\abs{x})$ for some measurable function, $g$. Then the solution to the Euler equations is stationary, with
\begin{align*}
    u(x)
        &= \pr{-\frac{x_2}{r^2} \intg, \,
                \frac{x_2}{r^2} \intg},
\end{align*}
where $r = \abs{x}$. Then
\begingroup
\allowdisplaybreaks
\begin{align*}
    \grad u(x)
        &= \matrix{
            \displaystyle
            \frac{2 x_1 x_2}{r^4} \intg - \frac{x_1 x_2}{r^2} g(r) &
            \displaystyle
            \pr{\frac{2 x_2^2}{r^4} - \frac{1}{r^2}} \intg - \frac{x_2^2}{r^2} g(r) \\}
            {\displaystyle
            \pr{\frac{-2 x_1^2}{r^4} + \frac{1}{r^2}} \intg + \frac{x_1^2}{r^2} g(r) &
            \displaystyle
            \frac{-2 x_1 x_2}{r^4} \intg + \frac{x_1 x_2}{r^2} g(r)} \\
        &= \frac{1}{r^4} \intg
            \matrix{
            2 x_1 x_2 &
            x_2^2 - x_1^2 \\}
            {x_2^2 - x_1^2 &
            -2 x_1 x_2}
            + \frac{g(r)}{r^2}
                \matrix{-x_1 x_2 & -x_2^2}{x_1^2 & x_1 x_2}.
\end{align*}
\endgroup

For simplicity, add the assumption that $g(r) = 0$ on $B_\delta(0)$ for some $\delta > 0$. Then, choosing
\begin{align*}
    Y
        = (1 - a_{\delta/4}) \e_\theta
        = (1 - a_{\delta/4}) \pr{\frac{-x_2}{r}, \frac{x_1}{r}},
\end{align*}
with $a$ as in \cref{D:Radial}, and letting $\chi(\abs{x}) = 1 - a_{\delta/2}(x)$, \cref{e:AExplicit} gives
\begin{align*}
    A(x)
        = \frac{\chi(r)}{r^2}
                \matrix
                    {-x_1 x_2 & - x_2^2}
                    {x_1^2 & x_1 x_2}.
\end{align*}

We see, then, that
\begin{align*}
    \grad u(x) - \omega(x) A(x)
        &= \frac{1}{r^4} \intg
            \matrix{
            2 x_1 x_2 &
            x_2^2 - x_1^2 \\}
            {x_2^2 - x_1^2 &
            -2 x_1 x_2} \\
        &= \frac{1}{r^2} \intg
            \matrix{2 \cos \theta \sin \theta & \sin^2 \theta - \cos^2 \theta}
                        {\sin^2 \theta - \cos^2 \theta & - 2 \cos \theta \sin}
\end{align*}
in polar coordinates. This is $C^\iny$ in $\theta$ and is as smooth in $r$ as $\rho$ allows, but is in any case always at least Lipschitz continuous. (Across the boundary of a classical vortex patch, for instance, it is only Lipschitz continuous.)
} 

\ifbool{ForSubmission} {
    We now give a simple example where $A$ can be explicitly calculated.
        
    \RadiallySymmetricExample
    
As simple as this example is, it provides useful insight into how $\omega(t)$ and $A(t)$ combine to cancel the singularities in $\grad u$. In particular, it highlights how the matrix $A$ has no direct dependence on the magnitude of $\omega$, only upon its irregularities as described by $Y$. So the same pattern of irregularities in the \textit{initial} vorticity would yield the same $A(0)$. For a nonstationary solution $A(t)$ would, of course, evolve in a way that depends upon the magnitude of $\omega$ at time zero.

    }
    {
    }

\section{Persistence of regularity of a vortex patch boundary}\label{S:ProofOfVortexPatch}
\noindent We now prove \cref{T:VortexPatch} using \cref{T:Serfati}, first  reformulating the vortex patch problem using level sets as in \cite{ConstantinBertozzi1993}.  We take $\displaystyle \phi_0 \in C^{1,\alpha}(\R^2)$ such that 
\begin{align*} 
    &\phi_0(x) > 0 \quad \text{in $\Omega$}, \\
    &\phi_0(x) = 0 \quad \text{on $\prt \Omega$ \ \ and}
        \ \  \inf_{x\in \partial \Omega} \abs{\nabla \phi_0(x)}\ge 2c > 0.
\end{align*}
We will be applying \cref{e:LevelCurveSmoothness} of \cref{T:SerfatiAdditional} with the closed set $\Sigma = \prt \Omega$.

Let $Y_0 = \grad^\perp \varphi_0 \in C^\al$ and note that $Y_0$ is tangential to $\prt \Omega$ and $\abs{Y_0} \ge c > 0$ on $\Ndelta{\prt \Omega}{\delta_0}$ for some $\delta_0 > 0$, since $\prt \Omega$ is compact. Also, formally, 
\[
\dv (\omega_0 Y_0) = \omega_0 \dv Y_0 + \grad \omega_0 \cdot Y_0 = 0 + \grad \omega_0 \cdot Y_0 = 0.
\]
More precisely, let $\varphi$ be any test function in $\Cal{D}(\R^2)$. Then
\begin{align*} 
    (\dv (\omega_0 Y_0), \varphi)
        &= - (\omega_0 Y_0, \grad \varphi)
        = - \int_\Omega \omega_0 Y_0 \cdot \grad \varphi
        = - \int_\Omega Y_0 \cdot \grad \varphi
        = \int_\Omega \dv Y_0 \, \varphi
        = 0.
\end{align*}
The integration by parts is valid since $Y_0$ lies in $C^\al(\Omega) \subseteq L^2(\Omega)$ and $\dv Y_0 = 0$ also lies in $L^2(\Omega)$ (see, for instance, Theorem I.1.2 of \cite{T2001}), the boundary integral vanishing since $Y_0 \cdot \n = 0$.

Let $\phi(t)$ be $\phi_0$ transported by the flow map, so that 
\[
\prt_t \phi + u \cdot \grad \phi = 0.
\]
Applying $\grad^\perp$ to both sides of this equation gives
\begin{align*} 
    \prt_t \grad^\perp \phi + u \cdot \nabla \grad^\perp \phi
        = \grad^\perp \phi \cdot \nabla u.
\end{align*}
Comparing this to \cref{e:YTransportAlmost}, we see that $Y(t) = \grad^\perp \phi(t)$. Since $Y(t) \in C^\al$ by \cref{T:Serfati} applied with $\Sigma = \prt \Omega$, we have $\phi(t) \in C^{1 + \al}$.
Then, because $\prt \Omega_t$ remains a level set of $\phi(t)$, $Y(t) = \grad^\perp \phi(t)$ is tangential to $\prt \Omega$. Hence, the boundary of the vortex patch remains in $C^{1 + \al}$. (This argument also proves \cref{e:LevelCurveSmoothness}.)

\Ignore{ 
Then, the vortex patch problem is equivalent to seeking a solution $\phi(t, x)$ of
\begin{subequations} 
\begin{align}
      & \phi_t + u\cdot \nabla \phi = 0, \\
    &u(t, x)
        = \frac{1}{2\pi}\int_{\Omega_t}\nabla^\perp \mathcal{F}(x - y) \, dy, \quad
    \Omega_t
        := \left\{x \colon \phi(t, x)>0 \right\}. 
\end{align}
\end{subequations}
We note that $\omega_0 \in C^\alpha(\R^2\setminus \prt \Omega)$ and $\dv (\omega_0 Y_0) = 0$.  Moreover, $Y_0 = (-\prt_2 \phi_0 , \prt_1 \phi_0 )$ is a $C^\alpha$ vector field in $\R^2$ such that $\abs{Y_0}\ge c > 0$ on $\prt \Omega$. Therefore, there exists a unique solution $u$ satisfying the bound in \cref{S:EstimationGraduAndY}. We next show $\gamma(t) = \eta(t, \gamma_0 ) \in C^{1 + \alpha}(\mathbb{S}^1)$ which is equivalent to showing \cite{ConstantinBertozzi1993} 
\begin{align*}
    W(t, \eta(t, x)) := \nabla^\perp \phi (t, \eta(t,x)) \in C^\alpha(\R^2).
\end{align*}
Since $W$ satisfies the equation 
\[
W_{t}+u\cdot \nabla W=\nabla u W
\]
we obtain the same regularity of $W$ as $Y$:   
\begin{align*}
    W(t, \eta(t, \cdot))\in C^\alpha(\R^2)
\end{align*}
which completes the proof.
} 

%
%
\section{An extension of Serfati's result}\label{S:SerfatiGen}

\noindent It is actually slightly easier to prove a more general form of \cref{T:Serfati}, making assumptions on the initial data using a family of vector fields, much as Chemin does in \cite{C1998}. At the expense of a little extra bookkeeping, the decomposition of the vorticity into ``good'' and ``bad'' parts is no longer needed, and all the associated estimates in \cref{S:Refinedgradueps} go away.

We start with a family $Y_0 = (Y_0^{(\la)})_{\la \in \Lambda}$ of vector fields in $C^\al(\R^2)$ and define, for any $s \in \R$,
\begin{align*}
    \norm{f(Y_0)}_{C^s}
        &:= \sup_{\la \in \Lambda} \norm{f \pr{Y_0^{(\la)}}}_{C^s(\R^2)}, \\
    I(Y_0)
        &:= \inf_{x \in \R^2} \sup_{\la \in \Lambda} \abs{Y_0^{(\la)}(x)}.
\end{align*}
Here, $f$ is any function on vector fields (such as the divergence) and we define 
\[
    f(Y_0) = \pr{f(Y_0^{(\la)})}_{\la \in \Lambda}.
\]
When $\norm{f(Y_0)}_{C^s} < \iny$ we say that $f(Y_0) \in C^s$. 

Then, in place of \cref{e:Assumption}, we assume that
\begin{align}\label{e:AssumptionGen}
	\left\{
	\begin{array}{l}
		\halfspacer \omega_0 \in (L^1 \cap L^\iny)(\R^2), \\
		\halfspacer \norm{Y_0}_{C^\al} < \iny, \quad
          I(Y_0)
            > 0, \\
		\halfspacer \dv(\omega_0 Y_0) \in C^{\al - 1}, \\
		\dv Y_0 \in C^\al.
	\end{array}
	\right.
\end{align}

We define the pushforward of the family, $Y_0$, by
\begin{align}\label{e:YFamily}
    Y(t) = (Y^{(\la)}(t))_{\la \in \Lambda}, \quad
        Y^{(\la)}(t, \eta(t, x)) := (Y_0^{(\la)}(x) \cdot \nabla) \eta(t, x).
\end{align}

Then the first part of \cref{T:Serfati} holds with these new assumptions, reproducing, as we show in the next section, the result of Chemin in \cite{C1998}:

\begin{theorem}\label{T:SerfatiGen}
Suppose that $\omega_0$ is such that the conditions in \cref{e:AssumptionGen} are satisfied for some family of vector fields $Y_0$. For the unique solution $\omega$ in $L^\iny(\R; (L^1 \cap L^\iny)(\R^2))$ to the Euler equations in vorticity formulation \cref{e:Eomega}, the bound in \cref{e:MainBounds} (with $Y$ as in \cref{e:YFamily}) holds,
as do \crefrange{e:YdivBound}{e:gradetaMainBound} and \cref{e:LevelCurveSmoothness}.
\end{theorem}
\begin{proof}
We outline only the changes that are needed to the proof of \cref{T:Serfati} given in \cref{S:ProofOfSerfatisResult}, as well as the preliminary transport estimates in \cref{S:Transport}.

In place of the $Y_\eps$, $R_\eps$ vector fields corresponding to the mollified initial vorticity, we have entire families,
\begin{align*}
    Y_\eps = (Y_\eps^{(\la)})_{\la \in \Lambda}, \quad
    R_\eps = (R_\eps^{(\la)})_{\la \in \Lambda}.
\end{align*}
The calculations in \cref{S:Transport} now apply to each $Y_\eps^{(\la)}$, $R_\eps^{(\la)}$. \cref{L:R0Decomp} then gives a $C^\al$ bound on the family $R_{0, \eps}$ and \cref{L:RegularitywY} bounds $\norm{\dv (\omega_\eps Y_\eps)}_{C^{\al - 1}}$.

\cref{S:EstimationGraduAndY,S:etaetainv} require no changes. The estimates in \cref{S:EstYR} are now done for each $Y_\eps^{(\la)}$, and \cref{e:YRCalphaBound} become bounds on whole families. Also, \cref{e:YepsBelow} becomes a lower bound on $I(Y_\eps)$.

In \cref{S:Refinedgradueps} the initial decomposition in \cref{e:pvgradKomegaFirstSplit} of the second term in $V_\eps$ is unchanged, as is the estimate of the second part. We no longer make a decomposition of the initial vorticity as in \cref{e:DecompositionOfVorticity}, for that is the purpose of the family of vector fields, $Y_\eps$. What remains, then, is the estimate of the matrix $B$ without a cutoff function; that is, with $B = \grad [\mu_{rh} \grad \Cal{F}] * \omega_\eps$. To estimate $B$ at $x \in \R^2$, we choose arbitrarily any $\la \in \Lambda$ such that \cref{e:detMBelow} holds at $x$ with $Y_0^{(\la)}$ in place of $Y_0$ and $c = I(Y_0)$; this is always possible by \cref{e:AssumptionGen}$_2$. This leads to the same estimate on $B$ as in \cref{e:BFinalEst}, where now $Y_\eps$ and $R_\eps$ are families of vector fields.

\cref{S:ClosingWithGronwall,S:Convergence} proceed with no significant changes, except that the estimates now apply to families of vector fields.
\end{proof}

To obtain the second part of \cref{T:Serfati}, \cref{e:ABound}, and \cref{T:LocalPropagation}, we must strengthen the assumptions on the vector field, $Y_0$, giving more uniform-in-space control on it than the restriction that $I(Y_0) > 0$. For this purpose, we define, for any $\la \in \Lambda$,
\begin{align}\label{e:Ulambda}
    U_\la
        = \set{x \in \R^2 \colon \smallabs{Y_0^{(\la)}(x)} > c'}
\end{align}
with $0 < c' < c := I(Y_0)$. Each $U_\la$ is open since it is the inverse image of an open set under the continuous map, $\smallabs{Y_0^{(\la)}}$. There always exists a countable partition of unity, $(\varphi_n)_{n \in \N}$, for which $\supp \varphi_n \subseteq U_\la$ for some $\la \in \Lambda$ (see Theorem 13.10 of \cite{Tu2008}). We require that such a partition of unity exist for some $c' \in (0, c)$ with the further property that
\begin{align}\label{e:POUAssumption}
    I' \pr{(\varphi_n)_{n \in \N}}
        := \sup_{x \in \R^2}
        \sum_{n \in \N}
        \norm{\varphi_n}_{\dot{C}^\al} \CharFunc_{\supp \varphi_n}(x)
        < \iny.
\end{align}
This condition rules out families of vector fields, $Y_0$, having members whose magnitude drops arbitrarily quickly from $c$ to $0$.

With this added assumption, we have \cref{T:RestOfResults}.
\begin{theorem}\label{T:RestOfResults}
    With $\omega_0$ as in \cref{T:SerfatiGen} and
    adding the condition in \cref{e:POUAssumption}, we have
    \cref{e:AEq} and \cref{e:ABound}, as well as the conclusions of
    \cref{T:LocalPropagation}.
\end{theorem} 
\begin{proof}
We need only show that \cref{e:AEq} holds, for the remaining facts follow from it and \cref{e:MainBounds}, which we established in \cref{T:SerfatiGen}.

Returning to \cref{S:MatrixA} we employ our partition of unity to piece together $A_\eps$.
We do not make the decomposition of $\grad u_\eps$ using the cutoff function $\chi$. Instead, we construct an $A_\eps^{(n)}$ corresponding to $\varphi_n(\eta_\eps^{-1}) \grad u_\eps$ in the same manner that $A_\eps$ was constructed for $\chi(\eta_\eps^{-1})\grad u_\eps$ in \cref{S:MatrixA}. We then let
\begin{align*}
    A_\eps = \sum_{n \in \N} \varphi_n (\eta_\eps^{-1}) A_\eps^{(n)}
\end{align*}
and note that we have a doubly exponential in time bound on $\smallnorm{A_\eps^{(n)}}_{C^\al(\eta_\eps(\supp \varphi_n))}$ uniform in $n$ and $\eps$. The regularity of $A$ and $\grad u - \omega A$ in \cref{e:AEq} along with the bound in \cref{e:ABound} then follow from an application of \cref{L:ProductCalBoundSum} with $\psi_n = \varphi_n \circ \eta_\eps^{-1}$ and $f_n = A_\eps^{(n)}$. We can do this since
\begin{align*}
   I'\pr{(\psi_n)_{n \in \N}}
       \le \norm{\grad \eta_\eps^{-1}}_{L^\iny}^\al
           I'\pr{(\varphi_n)_{n \in \N}} 
\end{align*}
by \cref{e:CdotIneq}$_1$.
\end{proof}

\begin{remark}
Observe how in the proof of \cref{T:SerfatiGen} we had no need of a partition of unity when treating the matrix $B$ since the regularity of $B$ was not at issue.
\end{remark}

\begin{lemma}\label{L:ProductCalBoundSum}
    Let $(\psi_n)_{n \in \N}$ be a partition of unity and let $(f_n)_{n \in \N}$
    be a sequence of functions in $C^\al(\R^2)$.
    Then
    \begin{align*}
        \norm{\sup_{n \in \N} \psi_n f_n}_{C^\al(\R^2)}
            \le 3 \pr{1 + I'\pr{(\psi_n)_{n \in \N}}} \sup_{n \in \N}
                \norm{f_n}_{C^\al(\supp \psi_n)}.
    \end{align*}
\end{lemma}
\begin{proof}
    Let
    \begin{align*}
        F
            = \sum_{n \in \N} \psi_n f_n,
            \quad
            K_n = \supp \psi_n.
    \end{align*}
    First observe that
    \begin{align*}
        \norm{F}_{L^\iny}
            \le \sup_{n \in \N} \norm{f_n}_{L^\iny(K_n)},
    \end{align*}
    so it remains
    only to bound the homogeneous norm.
    
    Let $x, y \in \R^2$. There exists two finite sets, $N = \set{n_1, \dots n_j}$ and
    $M = \set{m_1, \dots, m_k}$ such that the only elements of $(\psi_n)_{n \in \N}$
    not vanishing at $x$ have indices in $N$ and the only
    elements of $(\psi_n)_{n \in \N}$ not vanishing at $y$ have indices in $M$.
    Defining
    \begin{align*}
        \Delta_n(x, y)
            := \frac{\abs{(\psi_n f_n)(x) - (\psi_n f_n)(y)}}{\abs{x - y}^\al},
    \end{align*}
    there are four cases.
    
    If $n \in N$ and $n \in M$ then
    \begin{align}\label{e:DeltaxyBound}
        \begin{split}
        \Delta_n(x, y)
            &\le \frac{\abs{\psi_n(x) (f_n(x) - f_n(y))}}{\abs{x - y}^\al} 
                + \frac{\abs{f_n(y) (\psi_n(x) - \psi_n(y))}}{\abs{x - y}^\al} \\
            &\le \psi_n(x) \norm{f_n}_{\dot{C}^\al(K_n)}
                + \norm{\psi_n}_{\dot{C}^\al} \norm{f_n}_{L^\iny(K_n)}.
            \end{split}
    \end{align}
    
    If $n \notin N$ and $n \notin M$ then $\Delta_n(x, y) = 0$.
    
    If $n \in n$ but $n \notin M$ then
    \begin{align*}
        \Delta_n(x, y)
            &= \frac{\abs{\psi_n(x) f_n(x)}}{\abs{x - y}^\al} 
            = \frac{\abs{(\psi_n(x) - \psi_n(y))f_n(x)}}{\abs{x - y}^\al}
            \le \norm{\psi_n}_{\dot{C}^\al} \norm{f_n}_{L^\iny(K_n)}
    \end{align*}
    and the same inequality holds for $n \notin N$ but $n \in M$.
    
    Thus, \cref{e:DeltaxyBound} holds for all four cases.
    
    Then,
    \begin{align*}
        \Delta(x, y)
            &:= \frac{\abs{F(x) - F(y)}}{\abs{x - y}^\al}
            \le \sum_{n \in M \cup N} \Delta_n(x, y) \\
            &\le 2 \sup_{n \in \N} \norm{f_n}_{\dot{C}^\al(K_n)}
                + 2 \sup_{n \in \N} \norm{f_n}_{L^\iny(K_n)}
                I'(Y_0),
    \end{align*}
    from which the stated bound follows.
\end{proof}

\section{Equivalence to Chemin's vortex patch result}\label{S:EquivInitConditions}

\noindent Using our notation, Chemin in \cite{Chemin1993Persistance,C1998} makes the same assumptions on the initial data as those in \cref{e:AssumptionGen} except that in place of \cref{e:AssumptionGen}$_3$ he assumes that
$
    Y_0 \cdot \grad \omega_0 \in C^{\al - 1}.
$
We show in this section that the assumptions of Chemin and Serfati are, in fact, equivalent, and that both are equivalent to assuming that $Y_0 \cdot \grad u_0 \in C^\al$. We state this more precisely as follows:
\begin{prop}\label{P:Equivalence}
    Assume that $\omega_0 = \curl u_0$ satisfies \cref{e:Assumption}$_{1, 2, 4}$
    or \cref{e:AssumptionGen}$_{1, 2, 4}$. Then
    \begin{align*}
        Y_0 \cdot \grad \omega_0 \in C^{\al - 1}
            \iff
            \dv (\omega_0 Y_0) \in C^{\al - 1}
            \iff
            Y_0 \cdot \grad u_0 \in C^\al.
    \end{align*}
\end{prop}
\begin{proof}
    Since $\omega_0 \dv Y_0 \in L^1 \cap L^\iny$ and
    \begin{align*}
        Y_0 \cdot \grad \omega_0
            = \dv(\omega_0 Y_0) - \omega_0 \dv Y_0,
    \end{align*}
    the first equivalence is immediate in light of \cref{L:L1LInfHolder}, which we
    prove below. (This equivalence continues to hold with the weaker assumption
    of \cref{R:WeakerAssumption}; in fact, $\dv Y_0 \in L^\iny$ is sufficient for this
    equivalence to hold.)
    
    We have already shown that $\dv (\omega_0 Y_0) \in C^{\al - 1}
    \implies Y_0 \cdot \grad u_0 \in C^\al$,
    since $Y_0 \cdot \grad u_0 \in C^\al$ is \cref{e:YgraduBound}
    at $t = 0$. We complete the proof by showing that
    $Y_0 \cdot \grad u_0 \in C^\al \implies \dv (\omega_0 Y_0) \in C^{\al - 1}$.
    We will do this only in the setting of the more general assumptions in
    \cref{e:AssumptionGen}, though there is a clear analog to those in \cref{e:Assumption}.

    So assume that $Y_0 \cdot \grad u_0 \in C^\al$ and that
    \cref{e:AssumptionGen}$_{1, 2, 4}$ hold.
    Let $Y = Y_0^{(\la)}$ be any element of $Y_0$ and let $U_\lambda$ be as given
    in \cref{e:Ulambda}. Then on $U_\lambda$ we can write $\grad u_0$ as
    \begin{align*}
        \grad u_0
            &= \grad u_0 \matrixone{Y & Y^\perp} \matrixone{Y & Y^\perp}^{-1}
            = \matrix{Y \cdot \grad u_0^1 & Y^\perp \cdot \grad u_0^1}
                     {Y \cdot \grad u_0^2 & Y^\perp \cdot \grad u_0^2}
                     \matrixone{Y & Y^\perp}^{-1}.
    \end{align*}
    Now, $Y \cdot \grad u_0^1$ and $Y \cdot \grad u_0^2$ both lie in $C^\al$, so using
    $\dv u_0 = 0$ and $\omega_0 = \prt_1 u_0^2 - \prt_2 u_0^1 \in L^\iny$, we have
    \begin{align*}
        \grad u_0^1 \cdot Y^\perp
            = - \prt_1 u_0^1 Y^2 + \prt_2 u_0^1 Y^1
            = \prt_2 u_0^2 Y^2 + \prt_1 u_0^2 Y^1 - \omega_0 Y^1
            = Y \cdot \grad u_0^2 - \omega_0 Y^1
            \in L^\iny, \\
        \grad u_0^2 \cdot Y^\perp
            = - \prt_1 u_0^2 Y^2 + \prt_2 u_0^2 Y^1
            = \omega_0 Y^2 - \prt_2 u_0^1 Y^2 - \prt_1 u_0^1 Y^1
            = \omega_0 Y^2 - Y \cdot \grad u_0^1
            \in L^\iny.
    \end{align*}
    Also,
    \begin{align*}
        \smallabs{\matrixone{Y & Y^\perp}^{-1}}
            = \abs{\matrixone{Y & Y^\perp}}^{-1}
            \le C \abs{Y}^{-1}.
    \end{align*}
    This shows that
    \begin{align}\label{e:gradu0Bound}
        \begin{split}
        \norm{\grad u_0}_{L^\iny(U_\la)}
            &\le C \abs{Y}^{-1} \brac{
                    \norm{Y \cdot \grad u_0}_{L^\iny(U_\la)}
                    + \norm{\omega_0}_{L^\iny} \norm{Y}_{L^\iny(U_\la)}
                } \\
            &\le C I(Y_0) \bigbrac{
                    \norm{Y_0 \cdot \grad u_0}_{C^\al}
                    + \norm{\omega_0}_{L^\iny} \norm{Y_0}_{C^\al}
                }.
        \end{split}
    \end{align}
    We conclude that $\grad u_0 \in L^\iny(\R^2)$.
    
    The result now follows immediately from \cref{C:graduCor},
    though this corollary requires regularity of $\omega_0$, so an
    approximation argument, which we leave to the reader, is required.
    \OptionalDetails{
    We let $u_{0, \eps} = \rho_\eps * u_0$ as in \cref{S:Transport} giving and
    $\omega_{0, \eps} = \rho_\eps * \omega_0$. Note that
    $\grad u_{0, \eps} \in L^\iny(\R^2)$ with a bound that
    is uniform in $\eps$.
    Then
    \begin{align*}
        Y \cdot \grad u_{0, \eps}
            &= Y \cdot (\rho_\eps * \grad u_0)
            = \brac{Y \cdot (\rho_\eps * \grad u_0) - \rho_\eps*(Y \cdot \grad u_0)}
                + \rho_\eps*(Y \cdot \grad u_0).
    \end{align*}
    Hence,
    \begin{align*}
        \norm{Y \cdot \grad u_{0, \eps}}_{C^\al}
            &\le \norm{\int \rho(x - y) \grad u_0(y) \brac{Y(x) - Y(y)} \, dy}_{C^\al}
                + \norm{\rho_\eps*(Y \cdot \grad u_0)}_{C^\al} \\
            &\le C \norm{\grad u_0}_{L^\iny} \norm{Y}_{C^\al}
                + \norm{Y \cdot \grad u_0}_{C^\al}
            \le C(\omega_0, Y_0),
    \end{align*}
    where we applied \cref{L:SerfatiLemma2Inf} with the kernel $L_4$ of
    \cref{L:SerfatiKernels}.
    \cref{C:graduCor} then gives
    \begin{align*}
        \norm{K * \dv(\omega_{0, \eps} Y)}_{C^\al}
            \le C(\omega_0, Y_0),
    \end{align*}
    since $V(\omega_{0, \eps}) \le V(\omega_0)$.
    (The constant's dependence on $\al$ is immaterial
    for our purposes here.)
    We conclude that some subsequence
    (which we do not relabel) of $(K * \dv(\omega_{0, \eps} Y))$ converges in
    $C^r(\R^2)$ for all $r < \al$ to some $f \in C^\al(\R^2)$.
    Since $\omega_{0, \eps} \to \omega_0$ as Schwartz-class distributions, it follows
    that $K * \dv(\omega_{0, \eps} Y) \to K * \dv(\omega_0 Y)$ (by \cref{e:KdivZDef},
    for instance) and that $f = K * \dv(\omega_0 Y)$.
    } 
\end{proof}

\begin{lemma}\label{L:L1LInfHolder}
    For all $\beta < 0$,
    \begin{align*}
        (L^1 \cap L^\iny)(\R^2) \subseteq C^\beta(\R^2).
    \end{align*}
\end{lemma}
\begin{proof}
    Let $\beta \in (-1, 0)$ and $f \in (L^1 \cap L^\iny)(\R^2)$. From Lemma 8.1
    and Proposition 8.2 of \cite{MB2002}, $K * f$ is log-Lipschitz and
    so lies in $C^{1 + \beta}(\R^2)$. But then
    \begin{align*}
        f
            = \curl (K * f)
            = - \dv (K * f)^\perp
    \end{align*}
    is the divergence of the $C^{1 + \beta}$-function, $- (K * f)^\perp$,
    and so lies in $C^\beta(\R^2)$.
\end{proof}


In both \cite{C1998} and \cite{SerfatiVortexPatch1994}, \cref{e:YTransportAlmost} is used to bound $\norm{Y_\eps(t)}_{C^\al}$ (or $\norm{X_{t, \la}}_{X^\eps}$ in \cite{C1998}) which leads each author to bound $\norm{Y_\eps \cdot \grad u_\eps}_{C^\al}$. Serfati does this\footnote{Rather, this is our interpretation of what he is doing, as the expression for $Y(x) \cdot \grad u(x)$ in \cref{C:graduCor} never appears in \cite{SerfatiVortexPatch1994}.} using \cref{C:graduCor}, introducing the quantity $\dv (\omega_\eps Y_0)$, which is transported by the flow and can be uniformly bounded in $C^{\al - 1}$. Chemin does this on p. 101 of \cite{C1998}. Since $\dv(\omega_\eps Y_\eps)$, $\omega_\eps$, and $\dv Y_\eps$ are each transported by the flow, it follows that $Y_\eps \cdot \grad \omega_\eps = \dv(\omega_\eps Y_\eps) - \omega_\eps \dv Y_\eps$ is also transported by the flow, and it turns out that it too can be uniformly bounded in $C^{\al - 1}$. The two proofs diverge sharply in how they manage all the estimates that result, but this dichotomy of choice in what is to be transported is the origin of the difference between both their initial hypotheses and their end results.

\Ignore{ 
Observe also that the assumption in \cref{e:AssumptionGen}$_4$ is used in \cite{C1998} both in obtaining the bound on  on $(Y_\eps \cdot \grad \omega_\eps)(t)$ in $C^\al$ and in proving the convergence of $Y_\eps \to Y$, while it is required in the proof of \cref{T:Serfati} only for the latter purpose.
} 

The condition $Y_0 \cdot \grad u_0 \in C^\al$ has a precise geometric interpretation: the initial velocity has $C^{1 + \al}$-regularity in the direction of $Y_0$, and this regularity persists over time. The condition $Y_0 \cdot \grad \omega_0 \in C^{\al - 1}$ does not mean that $\omega_0$ has $C^\al$-regularity in the direction of $Y_0$, except in a loose sense, and the condition $\dv (\omega_0 Y_0) \in C^{\al - 1}$ or Serfati's original form of this condition that $K * \dv (\omega_0 Y_0) \in C^\al$ are hard to interpret.

Using the condition $Y_0 \cdot \grad u_0 \in C^\al$ also allows one to view the result of Chemin in \cite{C1998} as an extension of the well-posedness of the Euler equations for $u_0 \in C^{1 + \al}$ (as in Chapter 4 of \cite{C1998}), showing that such regularity in one direction is sufficient and will persist over time. The constants $c_\al$ and $C_\al$ of \cref{T:Serfati}, however, do not depend only upon $\norm{Y_0 \cdot \grad u_0}_{C^\al}$, since if nothing else they also depend upon $\norm{\omega_0}_{L^1 \cap L^\iny}$. To have  well-posedness in the sense of Hadamard, then, would require a definition of the proper functions space and a closer evaluation of the manner in which $c_\al$ and $C_\al$ depend upon $\omega_0$ and $Y_0$.

\section{Examples satisfying the hypotheses of Serfati's theorem}\label{S:Examples}
\noindent We have already seen in \cref{S:ProofOfVortexPatch} that a classical vortex patch satisfies the hypotheses of \cref{T:Serfati} in \cref{e:Assumption}. The following are some additional examples:

\begin{enumerate}[label=(\arabic*), ref=Example \arabic*]  
\item
Suppose that $\omega_0 \in C^\al(\R^2)$. Then choose $\Sigma = \emptyset$ or choose $Y_0$ to be any nonzero constant vector on $\Sigma = \R^2$ with $\omega_0 Y_0 \in C^\al(\R^2)$ so $\dv(\omega_0 Y_0) \in C^{\al - 1}(\R^2)$; either way, \cref{e:Assumption} is satisfied.

\smallskip

\item\label{E:PatchConstantOnBoundary}
Let $\Sigma = \prt \Omega$, where $\Omega$ is a bounded domain having a $C^{1 + \al}$ boundary.
Let
$\omega_0 = f \CharFunc_\Omega$ for $f \in C^\al(\Omega)$ with $f|_{\prt \Omega} \equiv \gamma$, $\gamma$ being a constant. Choose $\phi_0$ and $Y_0 = \grad^\perp \phi_0$ as for a classical vortex patch (see \cref{S:ProofOfVortexPatch}). Now, $\omega_0 - \gamma \CharFunc_\Omega$ and $Y_0$ both lie in $C^\al$, so $\dv((\omega_0 - \gamma \CharFunc_\Omega) Y_0) \in C^{\al - 1}$. But,
\begin{align*}
    \dv((\omega_0 - \gamma \CharFunc_\Omega) Y_0)
        = \dv(\omega_0 Y_0)
            - \gamma \dv(\CharFunc_\Omega Y_0)
        = \dv(\omega_0 Y_0),
\end{align*}
since we showed that $\dv(\CharFunc_\Omega Y_0) = 0$ in \cref{S:ProofOfVortexPatch}. Hence, \cref{e:Assumption} holds and $\prt \Omega_t$ will remain $C^{1 + \al}$ for the same reason as for a classical vortex patch.

\smallskip

\Ignore{ 
\item
Let $\Sigma = \prt \Omega$ as in the previous example.
Let
$
    \omega_0 = f \CharFunc_\Omega,
    f \in C^{1 + \al}(\Omega)
$ and assume that $\grad f$ does not vanish on $\prt \Omega$. Let $\ol{f}$ be an extension of $f$ to $C^{1 + \al}(\R^2)$ and let $Y_0 = \grad^\perp \ol{f} \in C^\al(\R^2)$. Then $\dv (\omega_0 Y_0) $

the same calculations as in the previous example gives \cref{e:Assumption}, so the conclusions of \cref{T:Serfati} hold. We cannot conclude, however, that the boundary of the patch remains $C^{1 + \al}$.

\smallskip
} 

\Ignore { 
\item\label{E:SumVorticities}
A sum, $(\omega_0^k)_{k = 1}^N$, of initial vorticities satisfying \cref{e:Assumption} for $(\Sigma_k)_{k = 1}^N$ as long as the $\Sigma_k$'s are pairwise disjoint.

\smallskip
} 

\item

A finite sum of classical vortex patches or vorticities as in \ref{E:PatchConstantOnBoundary} as long as their boundaries are disjoint. The boundaries will remain $C^{1 + \al}$.

\smallskip

\item\label{E:ShearFlow}
    Let $\phi_0 \in C^{1 + \al}(\R^2)$ with $\abs{\grad \phi_0} \ge c > 0$ on
    all of $\R^2$ have level curves each of which crosses any given vertical line
    exactly once.
    Let $Y_0 = \grad^\perp \phi_0$. Then $Y_0 \in C^\al(\R^2)$, $\dv Y_0 = 0$,
    and its flow lines are level curves of $\phi_0$.
    $Y_0$ describes a shear flow deviating in a controlled way
    from horizontal.
    Define $f_{x_1}(x_2)$ so that the flow line that passes through
    $(x_1, f_{x_1}(x_2))$ also passes through $(0, x_2)$.
    
    Let $W \colon \R \to \R$ be any measurable bounded function supported on some nonempty
    bounded interval $[c, d]$.
    \Ignore { 
    with the additional property that
    \begin{align}\label{e:WProp}
        m\pr{W((x - \delta, x + \delta)) \setminus W(x)} > 0
            \text{ for all }\, x \in \R, \delta > 0.
    \end{align}
    As a specific choice of $W$, we can use, on $[c, d]$,
    \begin{align*}
        W(x)
            = \sum_{n = 1}^\iny \frac{\cos (b^n \pi x)}{n^2},
    \end{align*}
    where $b \in \R$ is nonzero.
    Then $W$ is, in fact, continuous but lies in no $C^\beta(\R)$ for any $\beta > 0$
    not even locally at any point (see \cite{Hardy1916}). In particular, \cref{e:WProp}
    holds.
    } 
    For some fixed $L > 0$ let
    \begin{align*}
        \omega_0(x_1, x_2)
            = \CharFunc_{[-L, L]}(x_1) W(f_{x_1}(x_2))
    \end{align*}
    and let
    \begin{align*}
        \Sigma
            = \set{(x_1, x_2) \colon (x_1, f_{x_1}(x_2)) \in [-L, L] \times [c, d]}.
    \end{align*}
    Observe that $\omega_0$ has the same level curves as $\phi_0$, which are
    all in $C^{1 + \al}$.
    
    Now, \cref{e:Assumption}$_{1, 2}$ are clearly satisfied. Also, formally,
    $\dv (\omega_0 Y_0) = \grad \omega_0 \cdot Y_0 + \omega_0 \dv Y_0 =  0$,
    and we can verify this as for a classical vortex patch.
    \OptionalDetails{
    Let $\varphi$ be any test function in $\Cal{D}(\R^2)$. Then
    \begin{align*}
        (\dv (\omega_0 Y_0), \varphi)
            &= - (\omega_0 Y_0, \grad \varphi)
            = - \int_\Sigma \omega_0(x) Y_0(x) \cdot \grad \varphi(x) \, dx.
    \end{align*}
    We make the change of variables, $x = (x_1, x_2) = (y_1, f^{-1}_{y_1}(y_2))$.
    Then in the integral,
    $
        \omega_0(x)
            = W(y_2)
    $, so
    \begin{align*}
        (\dv (\omega_0 Y_0), \varphi)
            = -\int_c^d W(y_2) \int_{-L}^L
                    (Y_0 \cdot \grad \varphi)\pr{y_1, f^{-1}_{y_1}(y_2)}
                    \abs{f_{y_1}'(f_{y_1}^{-1}(y_2))}
                    \, d y_1 \, d y_2.
    \end{align*}
    
    Let $(\wh{y}_1, \wh{y}_2)$ be a unit coordinate system determined by the
    variables $y_1, y_2$.
    Then $Y_0$ is parallel to $\wh{y}_1$ since the lines of constant $y_2$
    are the flow lines for $Y_0$. Thus,
    \begin{align*}
        (Y_0 \cdot \grad \varphi)\pr{y_1, f^{-1}_{y_1}(y_2)}
            = \pr{(Y_0 \cdot \wh{y}_1) \prt_{y_1} \varphi}
                \pr{y_1, f^{-1}_{y_1}(y_2)}.
    \end{align*}
    Hence, we can integrate by parts to obtain
    \begin{align*}
        (\dv (\omega_0 Y_0), \varphi)
            = \int_c^d W(y_2) \int_{-L}^L
                    \pr{\prt_{y_1} (Y_0 \cdot \wh{y}_1) \, \varphi}
                        \pr{y_1, f^{-1}_{y_1}(y_2)}
                    \abs{f_{y_1}'(f_{y_1}^{-1}(y_2))}
                    \, d y_1 \, d y_2.
    \end{align*}
    But in the $y$-coordinate system, $\prt_{y_1} (Y_0 \cdot \wh{y}_1)
    = \dv Y_0 = 0$, so the integral vanishes.
    } 
    
    Because $\phi_0$ and $\omega_0$ have the same level curves
    and level curves are transported
    by the flow, $\phi(t)$ and $\omega(t)$
    have the same level curves for all time, where $\phi(t)$ is $\phi_0$
    transported by the flow. We conclude from \cref{e:LevelCurveSmoothness}
    that all the level curves of $\omega$ remain $C^{1 + \al}$,
    including the top and bottom boundaries of $\supp \omega(t)$.
    That is, extreme lack of regularity of $\omega_0$ transversal
    to $Y_0$ does not disrupt the regularity of the flow lines.
    \Ignore{ 
   
    Now, $Y_0$ is tangent to the level curves of $\phi_0$.
    We conclude, using the same logic as used at the end of
    \cref{S:ProofOfVortexPatch}, that $Y(t) = \grad^\perp \phi(t)$ for all
    time, where $\phi$ is $\phi_0$ transported by the flow map,
    and hence that $\phi(t) \in C^{1 + \al}$. But $\phi_0$ and
    $\omega_0$ have the same level curves, and level curves are transported
    by the flow map, so $\phi$ and $\omega$
    have the same level curves for all time
    (\cref{e:WProp} makes this easy to see, though it is a stronger condition
    than necessary).
    Hence, we conclude that all the level curves of $\omega$ remain $C^{1 + \al}$,
    including the top and bottom boundaries of $\supp \omega(t)$.
    That is, the potentially extreme lack of regularity of $\omega_0$ transversal
    to $Y_0$ does not disrupt the regularity of the flow lines.
    }
\smallskip

\Ignore{ 
\item
    The idea of the previous example can also be applied to a vortex patch
    with initial vorticity having $C^\al$-regularity only in the direction
    tangential to the boundary. Dealing with the interior regularity
    requires either applying \cref{T:Serfati} twice with one $Y_0$ defined in
    the interior and one near the boundary or the application of
    \cref{T:SerfatiGen} with a family of three vector fields.
} 

\Ignore{ 
\smallskip

\item
    In both the previous examples, a finite number of jump discontinuities
    in the directions transversal to $Y$ could be allowed.
    For an infinite number of discontinuities, verifying that
    $\dv (\omega_0 Y_0) = 0$ becomes technically complex, and some
    additional condition would likely be required.
} 
    
\smallskip

\item
Any vector field satisfying \cref{e:Assumption} or \cref{e:AssumptionGen} plus a $C^\al(\R^2)$ vector field. Because this does not require the choice of the vector field or family of vector fields $Y_0$ to change, if $Y_0$ is divergence-free then \cref{e:LevelCurveSmoothness} will continue to hold. In particular, we conclude that the initially $C^{1 + \al}$ boundary of a classical vortex patch remains $C^{1 + \al}$ even if the initial vorticity is perturbed by a $C^\al(\R^2)$ vector field.
\end{enumerate}

%
%
\OptionalDetails{
Let us return to \ref{E:PatchConstantOnBoundary}, assuming now that $f \in C^{1 + \al}(\Omega)$, but assuming that $f$ is not constant on $\prt \Omega$. Let $\ol{f}$ be an extension of $f$ to $C^{1 + \al}(\R^2)$ and let $Y_0 = \grad^\perp \ol{f} \in C^\al(\Omega)$. Then $\dv (\omega_0 Y_0) = 0$ on $\Omega$.

Testing $\dv (\omega_0 Y_0)$ with a test function $\varphi \in \Cal{D}(\R^2)$ as in \cref{S:ProofOfVortexPatch}, we have
\begin{align*}
    (\dv (\omega_0 Y_0), \varphi)
        &= - (\omega_0 Y_0, \grad \varphi)
        = - \int_\Omega \omega_0 Y_0 \cdot \grad \varphi
        = \int_\Omega \dv(\omega_0 Y_0) \varphi
            - \int_{\prt \Omega} \omega_0 Y_0 \cdot \n \, \varphi \\
        &= - (f Y_0 \cdot \n) \mu, \varphi).
\end{align*}
Here, $\mu$ is the Radon measure on $\R^2$ supported on $\prt \Omega$ for which  $\mu|_\Gamma$ corresponds to Lebesgue (arc length) measure on $\prt \Omega$.
Hence, $\dv (\omega_0 Y_0) = - (f Y_0 \cdot \n) \mu$. Even though $\dv(\omega_0 Y_0) = 0$ on $\Omega$ and on $\Omega^C$, then, it does not follow that $\dv(\omega_0 Y_0) \in C^{\al - 1}(\R^2)$, so we cannot apply \cref{T:Serfati}; at least not with this choice of $Y_0$.


In all of our examples, $\dv Y_0 = 0$, where the assumptions on the initial data of \cite{SerfatiVortexPatch1994, C1998} coincide. Such examples are not only the easiest to create, but, because of \cref{e:LevelCurveSmoothness}, probably the most useful. Examples with $\dv Y_0 \ne 0$ can be created, though, by modifying some of the examples above in trivial ways. In \ref{E:PatchConstantOnBoundary}, for instance, we could make $Y_0$ non-divergence-free in the interior of $\Omega$. The argument in \cref{S:ProofOfVortexPatch} would then give $\dv(\omega_0 Y_0) = \dv Y_0$, which, though it does not vanish, lies in $C^\al(\R^2)$, so \cref{e:Assumption} still holds and $\prt \Omega_t$ still remains $C^{1 + \al}$.

For the 2D Euler equations, initial data for specific examples tends to be specified in terms of vorticity rather than velocity, meaning that the condition $\dv (\omega_0 Y_0) \in C^{\al - 1}$ tends to be easier to verify than $Y_0 \cdot \grad u_0 \in C^\al$.
} 


%
%
\OptionalDetails{
%
%
\section{Two examples of the matrix $A$}\label{S:CircularVortexPatch}
\noindent It is possible to calculate the matrix $A$ of \cref{e:AEq} for simple geometries when the solution is stationary. We calculate two specific examples: radially symmetric initial vorticity and a special case of the shear flow in \ref{E:ShearFlow} of \cref{S:Examples}.

As simple as these examples are, they provide useful insight into how $\omega(t)$ and $A(t)$ combine to cancel the singularities in $\grad u$. In particular, they highlight how the matrix $A$ has no direct dependence on the magnitude of $\omega$, only upon its irregularities as described by $Y$. So the same pattern of irregularities in the \textit{initial} vorticity would yield the same $A(0)$. Because these examples are both stationary, they do not exhibit $A(t)$ evolving over time: its evolution would, of course, depend upon the magnitude of $\omega$ at time zero.

%
%
\subsection{Radially symmetric vorticity} \RadiallySymmetricExample


\subsection{Shear flow}\label{S:ShearFlow} Consider \ref{E:ShearFlow} of \cref{S:Examples} for the special case of horizontal flow lines, but without applying $\CharFunc_{[-L, L]}$ to truncate the initial vorticity in the $x_1$ direction. The Eulerian solution is a stationary shear flow. Technically, \cref{T:Serfati} does not apply since the vorticity does not lie in $L^1(\R^2)$. We nonetheless calculate the matrix $A$, which will apply approximately for the original example (6) for sufficiently short time or sufficiently large $L$.

Choosing $\phi_0(x) = -x_2$, we have $Y = Y_0 = \grad^\perp \phi_0 = (1, 0)$, which never vanishes, so we need apply no cutoff function in \cref{e:AExplicit}. Thus,
\begin{align*}
    A
        = A(t)
        = \frac{1}{\abs{Y}^2}
                \matrix
                    {Y^1 Y^2 & - (Y^1)^2}
                    {(Y^2)^2 & -Y^1 Y^2}
        = \matrix{0 & -1}{0 & 0}.
\end{align*}

Now,
$
    \omega
        = \omega(t)
        = W(x_2).
$
The flow lines are horizontal, being the same as the flow lines for $Y$, and do not depend on $x_1$, so we can write $u = (g(x_2), 0)$ for some function $g$. It follows that $\prt_2 g(x_2) = - W(x_2)$, so that
\begin{align*}
    u(t, x_1, x_2)
        = \pr{C -\int_c^{x_2} W(s) \, ds, 0},
\end{align*}
where $C$ can be chosen arbitrarily. We can add the assumption that 
\[
\int_c^d W(s)ds = 0
\]
so that we can chose $C = 0$ and $u^1(x)$ will vanish outside of the infinite strip $\R \times [c, d]$. (The velocity will still not vanish as $x_1 \to \pm \iny$, but that reflects not truncating the support in the $x_1$ direction.) Thus,
\begin{align*}
    \grad u(x)
        = \grad u(t, x)
        = \matrix{0 & -W(x_2)}{0 & 0}.
\end{align*}

It follows that
\begin{align*}
    \grad u - \omega A
        = \matrix{0 & -W(x_2)}{0 & 0}
            - W(x_2) \matrix{0 & -1}{0 & 0}
        = 0,
\end{align*}
which certainly satisfies \cref{e:AEq,e:ABound}.
} 

%
%
\OptionalDetails{

%
%
\section{Concluding remarks}\label{S:Conclusion}

\noindent Serfati's results are most useful at obtaining the regularity of geometric structures in the Eulerian flow a la \cref{e:LevelCurveSmoothness} and at obtaining the local propagation of \Holder regularity as in \cref{T:LocalPropagation}. The first requires that $Y_0$ be divergence-free and that any lack of regularity in the initial vorticity be transversal to $Y_0$, the second requires that the initial vorticity be regular in some open set.

Because Serfati's approach is elementary, avoiding as it does paradifferential calculus, and relying primarily upon estimates of singular integrals (\cref{L:SerfatiLemma2}), it is more easily adapted to other purposes than Chemin's approach. As topics for future work, we discuss several such settings in the subsections that follow.

\subsection{Bounded domains and exterior domains}\label{S:BoundedDomains}
In \cite{Depauw1998, Depauw1999}, Depauw studies a vortex patch lying in a bounded domain, $\Omega$, with smooth boundary. He uses extension operators to transfer the problem, as much as possible, to the full plane, where he applies the tools of paradifferential calculus, somewhat in the spirit of \cite{C1998}. He is able to recover \cref{T:VortexPatch} for a bounded domain when the patch initially does not touch the boundary, and is able to obtain a short-time result when the patch does initially touch the boundary. (If a patch starts away from the boundary it remains away from the boundary for all time, since the flow map at time $t$ is a diffeomorphism.)

Serfati's approach adapts itself quite well to a bounded domain. The Biot-Savart law, $u = K * \omega$, is replaced by 
\[
u(t, x) = \int_\Omega K_\Omega(x, y) \omega(y) \, dy.
\]
When $\Omega$ is simply connected, $K_\Omega$ can be written explicitly. For simplicity, let $\Omega$ be the unit disk. Then 
\[
K_\Omega(x, y) = K(x - y) - K(x - y^*),
\]
where $y^* = y/\abs{y}^2$ is reflection across the unit disk. The key estimate required on $K_\Omega$ in applying \cref{L:SerfatiLemma2} is on its gradient, $\grad_x K_\Omega(x, y)$. The first term in $K_\Omega$ can be treated as in the full plane. For the second term, existing estimates on $\grad_x$ of it (see \cite{AKLL2014}, for instance) would give the required bounds as long as the patch starts away from the boundary. How much of Depauw's result for patches that touch the boundary can be obtained in this way, and whether his results could be strengthened, remain to be seen.

The same approach could also be taken for domains exterior to a simply connected obstacle.

Note there is one point in the proof of \cref{T:Serfati} where paradifferential calculus enters: in the application of Theorem 3.14 of \cite{BahouriCheminDanchin2011} in the proof of \cref{L:RegularitywY}. The proof of this theorem in \cite{BahouriCheminDanchin2011} relies on paradifferential calculus. The theorem itself, though, is easily extended to apply to a bounded (or exterior) domain by applying an extension operator and then restricting to the original domain. This works because the velocity is tangential to the boundary.

\subsection{Bounded vorticity, bounded velocity solutions}\label{S:BoundedBounded}
For a classical vortex patch, and for Chemin's and Serfati's more general results, the velocity field always decays at infinity. If we assume only that $\omega_0 \in L^\iny$ with $u_0 \in L^\iny$, existence and uniqueness of solutions is known to hold. This was shown originally by Serfati in \cite{Serfati1995A} (see also \cite{AKLL2014, KBounded2014, Taniuchi2004, TaniuchiEtAl2010}).

For such solutions, the Biot-Savart law no longer holds, but we do have the so-called \textit{Serfati identity},
\begin{align*}
        u^j(t&) - (u^0)^j
            = (a K^j) *(\omega(t) - \omega^0)
        - \int_0^t \pr{\grad \grad^\perp \brac{(1 - a) K^j}}
        \stardot (u \otimes u)(s) \, ds,
\end{align*}
where $M \stardot N = M^{mn} * N^{mn}$, and $a$ is as in \cref{D:Radial}.
Hence,
\begin{align*}
        \grad u^j(t) - (\grad u^0)^j
            = \grad \brac{(a K^j) *(\omega(t) - \omega^0)}
        - \int_0^t \grad \pr{\grad \grad^\perp
            \brac{(1 - a) K^j}}
        \stardot (u \otimes u)(s) \, ds.
\end{align*}
The first term can be handled as in the proof of \cref{T:Serfati}. The second lies in $C^\iny$, with the $C^k$ norm bounded by $C_k \norm{u}_{L^\iny}^2$.
Thus, it is likely that Serfati's vortex patch approach extends to bounded vorticity, bounded velocity solutions in the full plane.
One difficulty, however, is that $K * \dv Z \in C^\al$ is no longer equivalent to $\dv Z \in C^{\al - 1}$, since $K * \dv Z$ no longer makes sense.

Well-posedness of bounded vorticity, bounded velocity solutions exterior to a simply connected obstacle was proved in \cite{AKLL2014}. Though there are additional difficulties, the approach described above also applies in this case, at least if the discontinuities in $\omega_0$ are bounded away from the boundary.

\subsection{$C^{k + \al}$ regularity and vortex patches with singularities}
Vortex patches with singularities are studied in \cite{Chemin1995SingularVortexPatches, Danchin1997SingularVortexPatches, Danchin1997Cusp, Danchin2000Cusp}. We focus here on the result of Danchin's \cite{Danchin1997SingularVortexPatches}, in which it is shown that if the boundary of a vortex patch is \Holder continuous except on a closed subset of the boundary it remains so for all time as the boundary is transported by the flow. The subset on which the boundary (potentially) fails to be \Holder continuous is transported by the flow as well. 

Danchin obtains this result as a special case of a more general theorem in which he adapts Chemin's definition of a family of conormal vector fields in \cite{C1998} to both allow singularities and to deal with derivatives up to a fixed order, $k \ge 0$. It would be interesting to adapt Serfati's approach to obtain the same result as Danchin, and to then also apply it in the settings of \cref{S:BoundedDomains,S:BoundedBounded}.

} 

\appendix
\section{Proofs of lemmas}\label{A:ProofsOfLemmas}

\noindent In this section we prove the lemmas stated in \cref{S:Lemmas}.

\subsection{Proof of \cref{L:SerfatiLemma1}}
Let 
\begin{align*}
    E = \matrix{a & -c}{c & a}, \quad
    F =\matrix{d & -b}{-c & a}
    = (\det M) M^{-1},
\end{align*}
so that 
\begin{align*}
    E E^T
        =(\det E ) I, \quad
    M F
        = (\det M) I,
\end{align*}
$I$ being the $2 \times 2$ identity matrix.
Therefore, $B$ can be expressed as 
\begin{align}\label{e:B}
    B = \frac{E E^T B M F}{\det(M E )}.
\end{align}

We now compute $E^T B M$. Let $B = (B_{ij})$. Then, 
\begin{align*} 
    E^T B M
        &= \matrix{a^2B_{11}+acB_{12}+acB_{21}+c^2B_{22}
                    & abB_{11}+adB_{12}+bcB_{21}+cdB_{22}}
                {-acB_{11}-c^2B_{12}+a^2B_{21}+acB_{22}
                    & -bcB_{11}-cdB_{12}+abB_{21}+adB_{22}} \\
        &=: (l_{ij}).
\end{align*} 
Since $B$ is symmetric, we have 
\begin{align*} 
    &l_{11} =\matrix{a}{c}^T B \matrix{a}{c}, \quad
    l_{12}= \matrix{b}{d}^T B \matrix{a}{c}, \quad
    l_{21} =\matrix{-c}{a}^T B \matrix{a}{c}, \\
    &l_{22} =\det M \tr B - \matrix{d}{-b}^T B \matrix{a}{c}.
\end{align*} 
Therefore, we can rewrite $E^T B M$ as
\begin{align*} 
    E^T B M
        = \matrix
            {\matrix{a}{c}^T B \matrix{a}{c} & \matrix{b}{d}^T B \matrix{a}{c}}
            {\matrix{-c}{a}^T B \matrix{a}{c} & \matrix{d}{-b}^T B \matrix{a}{c}}
        +
            \matrix{0 & 0}{0 & \det M \tr B}.
\end{align*} 
Going back to \cref{e:B}, we obtain 
\begin{align*} 
    B
        &= \frac{E}{\det M \det E}
            \matrix
                {\matrix{a}{c}^T B \matrix{a}{c} & \matrix{b}{d}^T B \matrix{a}{c}}
                {\matrix{-c}{a}^T B \matrix{a}{c} & \matrix{d}{-b}^T B \matrix{a}{c}}
                F
                    + \frac{\tr B}{\det E} \matrix{c^2 & -ac}{-ac & a^2} \\
        &= \frac{E}{\det M (a^2 + c^2)}
            \matrix
                {(a, c) \cdot B M_1 & (b, d) \cdot B M_1}
                {(-c, a) \cdot B M_1 & (d, -b) \cdot B M_1}
            F
            + \frac{\tr B}{a^2 + c^2} \matrix{c^2 & -ac}{-ac & a^2}.
\end{align*}

Because all norms on finite-dimensional space are equivalent, $\abs{G H} \le C \abs{G} \abs{H}$ for any $2 \times 2$ matrices $G$, $H$. Therefore,
\begin{align*}
    \abs{B}
        &\le \frac{C}{\det M (a^2 + c^2)} \abs{E}
            \abs{\matrix{(a, c) \cdot B M_1 & (b, d) \cdot B M_1}
                {(-c, a) \cdot B M_1 & (d, -b) \cdot B M_1}
                }
            \abs{F}
            + C \frac{\abs{\tr B} (a^2 + c^2)}{a^2 + c^2}.
\end{align*}
But,
\begin{align*}
    \abs{E} \le C (a^2 + c^2)^{1/2}, \quad
    \abs{F} \le C (a^2 + b^2 + c^2 + d^2)^{1/2},
\end{align*}
and
\begin{align*}
    &\abs{\matrix{(a, c) \cdot B M_1 & (b, d) \cdot B M_1}
                {(-c, a) \cdot B M_1 & (d, -b) \cdot B M_1}
                }
        \le C \pr{\abs{(a, c)}^2 \abs{B M_1}^2 + \abs{(b, d)}^2 \abs{B M_1}^2}^{1/2} \\
        &\qquad
        \le C (a^2 + b^2 + c^2 + d^2)^{1/2} \abs{B M_1}.
\end{align*}
Hence,
\begin{align*}
    \displaystyle \abs{B}
        &\le C \frac{a^2 + b^2 + c^2 + d^2}{\det M (a^2 + c^2)^{1/2}}
             \abs{B M_1}
            + C \abs{\tr B} \\
        &\le C \frac{(a^2 + b^2 + c^2 + d^2)^{1/2}}{\det M}
             \abs{B M_1}
            + C \abs{\tr B}
        = C \frac{\abs{M}}{\det M}
             \abs{B M_1}
            + C \abs{\tr B}.
\end{align*}
\qed

\subsection{Proof of \cref{L:SerfatiLemma2}}
We need to show that 
\begin{align*}
    \abs{\int L(x,z) \brac{f(z) - f(x)} \, dz
            - \int L(y,z)\brac{f(z) - f(y)} \, dz}
        \le C \norm{L}_* \norm{f}_{C^\alpha} \abs{x - y}^\alpha.
\end{align*}
We set $h = \abs{x - y}$ and write
\begin{align*} 
    \int L&(x,z) \brac{f(z) - f(x)} \, dz
            -\int L(y,z) \brac{f(z) - f(y)} \, dz \\
    &=\int_{\abs{x - z} \le 2h} L(x,z) \brac{f(z) - f(x)} \, dz
            - \int_{\abs{x - z} \le 2h} L(y,z) \brac{f(z) - f(y)} \, dz \\
    &+ \brac{\int_{\abs{x - z} > 2h} L(x,z) \brac{f(z) - f(x)} \, dz
        - \int_{\abs{x - z} > 2h} L(y,z) \brac{f(z) - f(y)} \, dz} \\
    &=: \text{I + II + III}.
\end{align*}
We first estimate $\text{I}$ by 
\begin{align*} 
    \abs{\text{I}}
        &\le \int_{\abs{x - z} \le 2h} \abs{L(x,z)}
            \abs{x - z}^\alpha \frac{\abs{f(z) - f(x)}}{\abs{x - z}^\alpha} \, dz \\
        &\le \norm{f}_{C^\alpha} \int_{\abs{x - z} \le 2h}
            (\abs{L(x,z)} \abs{x - z}^2 ) \abs{x - z}^{\alpha-2} \, dz \\
        &\le  C \norm{f}_{\dot{C}^\alpha} \norm{L}_* \int_{\abs{x - z} \le 2h}
            \abs{x - z}^{\alpha-2} \, dz
        \le C \al^{-1} \norm{L}_* \norm{f}_{\dot{C}^\alpha} h^\alpha.
\end{align*}
If $\abs{x - z}\le 2h$ and $\abs{x - y}=h$, we have $\abs{y - z} \le 3h$ and thus 
\begin{align*}
    \abs{\text{II}}
        \le C \abs{\int_{\abs{y - z}\le 3h} L(y,z) \brac{f(z) - f(y)} \, dz}
        \le C \al^{-1} \norm{L}_* \norm{f}_{\dot{C}^\alpha} h^\alpha.
\end{align*}
To estimate $\text{III}$, we decompose it further (using \cref{e:LPVinL1}) into
\begin{align*} 
    \text{III}
        &= \brac{f(y)-f(x)} \int_{\abs{x - z} > 2h} L(x,z) \, dz
            + \int_{\abs{x - z} > 2h} \brac{f(z) - f(y)}
                (L(x,z) -L(y,z)) \, dz \\
        &=: \text{III}_1+\text{III}_2. 
\end{align*}
We immediately have that
\begin{align*}
    \abs{\text{III}_1}
        \le \norm{f}_{\dot{C}^\al} \int_{\abs{x - z} > 2h} \abs{x - z}^{\al - 2} \, dz
        \le C \al^{-1} \norm{L}_* \norm{f}_{C^\alpha} h^\alpha.
\end{align*}
We finally estimate $\text{III}_2$. Since
\begin{align*}
    \abs{L(x,z) - L(y,z)}
        \le \abs{\nabla_{x} L(\tilde{x},z)} \abs{x - y}, \quad
            \tilde{x}=ty + (1 - t)x, \ \text{for some}\ t\in [0,1], 
\end{align*}
we have 
\begin{align*} 
    \abs{\text{III}_2}
        &\le \int_{\abs{x - z} > 2h} \abs{\brac{f(z) - f(y)}
            (L(x,z) -L(y,z))} \, dz  \\
        &\le \int_{\abs{x - z} > 2h} \abs{f(z) - f(y)}
            \abs{\nabla_{x} L(\tilde{x},z)} \abs{x - y} \, dz \\
        &= h \int_{\abs{x - z} > 2h} \frac{\abs{f(z) - f(y)}}{\abs{y - z}^\alpha}
            (\abs{\nabla_{x} L(\tilde{x},z)}
            \abs{\tilde{x} - z}^{3}
            \frac{\abs{y - z}^\alpha}{\abs{\tilde{x} - z}^{3}} \, dz\\
        &\le \norm{L}_* \norm{f}_{C^\alpha} h  \int_{\abs{x - z} > 2h}
            \frac{\abs{y - z}^\alpha}{\abs{\tilde{x} - z}^{3}} \, dz
        \le C \norm{L}_* \norm{f}_{C^\alpha} h
            \int_{\abs{\tilde{x} - z}> h} \frac{1}{\abs{\tilde{x} - z}^{3-\alpha}} \, dz \\
        &\le C (1 - \al)^{-1} \norm{L}_* \norm{f}_{C^\alpha} h^\alpha ,
\end{align*}
where we used the inequalities $\abs{\tilde{x} - y} \le (1-t) \abs{x - y} \le h$ and $\abs{\tilde{x} - z} \ge \abs{x - z}-t\abs{x - y} \ge h$ to obtain $\abs{y - z}\le \abs{\tilde{x} - y} + \abs{\tilde{x} - z}\le 2\abs{\tilde{x} - z}$. Collecting all terms, we have \cref{e:KernelEstimate}.

Finally, if \cref{e:LHomo} holds then $\int_{\R^2} L(\cdot, z) f(\cdot) \, dz = 0$, and \cref{e:KernelEstimateHomo} follows from \cref{e:KernelEstimate}.
\qed

\subsection{Proof of \cref{L:SerfatiLemma2Inf}}

In light of \cref{L:SerfatiLemma2}, we need only bound the corresponding $L^\iny$ norms. We have,
\begin{align*}
    &\norm{\PV \int_{\R^2} L(\cdot, z) \brac{f(z)-f(\cdot)} \, dz}_{L^\iny} \\
        &\qquad
        \le \norm{f}_{\dot{C}^\al}
            \norm{\lim_{h \to 0} \int_{B_h(x)^C \cap B_1(x)}
            \abs{L(x, z)} \abs{x - z}^\al \, dz}_{L^\iny_x}
        + 2 \norm{f}_{L^\iny} \sup_{x \in \R^2} \norm{L(x, \cdot)}_{L^1(B_{1}(x)^C)}
        \\
        &\qquad
        \le \norm{L}_* \norm{f}_{\dot{C}^\al}
            \norm{\lim_{h \to 0} \int_{B_h(x)^C \cap B_1(x)}
            \abs{x - z}^{\al - 2} \, dz}_{L^\iny_x}
        + 2 \norm{L}_{**} \norm{f}_{L^\iny} \\
        &\qquad
        \le C \al^{-1} \norm{L}_{**} \norm{f}_{C^\alpha}.
\end{align*}

\subsection{Proof of \cref{L:EquivalentConditions}}\label{S:EqCond}
Suppose that $\dv Z\in C^{\alpha-1}(\R^2)$ with $Z \in L^{\infty} (\R^2)$. We have,
\begin{align*}
    \grad \Cal{F} * \dv Z
        = m(D) \dv Z
        = n_i(D) Z^i,
\end{align*}
where $m$ and $n_i$, $i = 1, 2$, are the Fourier-multipliers,
\begin{align*}
    m(\xi) = \frac{\xi}{\abs{\xi}^2}, \quad
    n(\xi) = \frac{\xi^i \xi}{\abs{\xi}^2},
\end{align*}
up to unimportant multiplicative constants.
We can thus write $\grad \Cal{F} * \dv Z$ using a Littlewood-Paley decomposition in the form,
\begin{align}\label{e:KdivZDef}
    \grad \Cal{F} * \dv Z
        = \sum_{j \ge -1} \Delta_j m(D) \dv Z
        = \Delta_{-1} n_i(D) Z^i
            + \sum_{j \ge 0} \Delta_j m(D) \dv Z,
\end{align}
where $\Delta_j$ are the nonhomogeneous Littlewood-Paley operators (dyadic blocks). We use the notation of \cite{BahouriCheminDanchin2011} and refer the reader to Section 2.2 of that text for more details. The sum in \cref{e:KdivZDef} will converge in the space $\Cal{S}'(\R^2)$ of Schwartz-class distributions as long as $\dv Z \in \Cal{S}'(\R^2)$.

Now, for any noninteger $r \in [-1, \iny)$,
\begin{align*}
    \sup_{j \ge -1} 2^{jr} \norm{\Delta_j f}_{L^\iny}
\end{align*}
is equivalent to the $C^r$ norm of $f$ (see Propositions 6.3 and 6.4 in Chapter II of \cite{Chemin2004Handbook}). Also,
\begin{align*}
    \norm{\Delta_j m(D) f}_{L^\iny}
        \le C 2^{-j} \norm{\Delta_j f}_{L^\iny}, \quad
    \norm{\Delta_{-1} n_i(D) f}_{L^\iny}
        \le C \norm{f}_{L^\iny}
\end{align*}
for all $j \ge 0$ and $i = 1, 2$. The first inequality follows from Lemma 2.2 of \cite{BahouriCheminDanchin2011} because $m$ is homogeneous of degree $-1$. The second inequality follows by a direct calculation, using only that $n_i$ is bounded.
\OptionalDetails{
    \begin{align*}
        \norm{\Delta_{-1} n_i(D) f}_{L^\iny}
            &= \norm{\FTR\pr{\chi \FTF(n_i(D) f)}}_{L^\iny}
            = \norm{\FTR\pr{\chi n_i \wh{Z}^i}}_{L^\iny} \\
            &\le \norm{\FTR\pr{\chi n_i}}_{L^1} \norm{f}_{L^\iny}
            \le C \norm{f}_{L^\iny}.
    \end{align*}
}

Hence,
\begin{align*}
    \norm{\grad \Cal{F} * \dv Z}_{C^{\alpha}}
        &\le \left\|\Delta_{-1} n_i(D) Z^i\right\|_{L^{\infty}}
            +\sup_{j\ge 0} 2^{j\al}\left\| \Delta_{j} m(D)\dv Z \right\|_{L^{\infty}} \\
        &\leq C \|Z\|_{L^{\infty}} + \sup_{j\ge 0} 2^{j(\al -1)}\left\| \Delta_{j}\dv Z \right\|_{L^{\infty}}\\
        & \leq C \|Z\|_{L^{\infty}}
            + C \left\| \dv Z\right\|_{C^{\alpha-1}},
\end{align*}
which gives the second inequality in \cref{e:EquivalentConditionsBound}.



Conversely, assume that $v :=\grad \mathcal{F}\ast \dv Z \in C^\alpha(\R^2)$. Then, 
\begin{equation*} 
 \begin{split}
   & \dv v=\Delta F \ast \dv Z=\dv Z.
 \end{split}
\end{equation*}
Therefore, we conclude that $\dv Z \in C^{\alpha - 1}(\R^2)$ and obtain the first inequality in \cref{e:EquivalentConditionsBound}.
\qed

%
%
\section{Calculations involving $\grad u$}\label{A:graduCalcs}

\noindent Recall from \cref{S:Notation} that $\grad u = Du$, the Jacobian matrix of $u$ (rather than its transpose, as it is sometimes defined).

\cref{L:graduExp} is a standard way of expressing $\grad u$; it is, in fact, the decomposition into its antisymmetric and symmetric parts. It follows, for instance, from Proposition 2.17 of \cite{MB2002}.
In \cref{L:graduYLikeLemma}, we inject the $C^\al$-vector field $Y$ into the formula given in \cref{L:graduExp}; the expression that results lies at the heart of the proof of \cref{T:Serfati}, via \cref{C:graduCor}. Finally, \cref{L:ConvEq} justifies switching between two ways of calculating principal value integrals.
\ifbool{ForSubmission}{
    We leave the proofs of \cref{L:graduYLikeLemma,L:ConvEq} to the reader.
    }
    {
    }

Our applications of these results are to our approximate solutions, which lie in $L^1 \cap C^\iny$.

\begin{lemma}\label{L:graduExp}
    Let $u$ be a divergence-free vector field in $(L^1 \cap C^\iny)(\R^2)$ with vorticity $\omega$. Then
    \begin{align*}
        \grad u(x)
            = \frac{\omega(x)}{2} \matrix{0 & -1}{1 & 0}
                + \PV \int \grad K(x - y) \omega(y) \, dy,
    \end{align*}
    where $K = \grad^\perp \Cal{F}$ is the Biot-Savart kernel.
    The first term is the antisymmetric, the second term the symmetric part
    of $\grad u(x)$.
\end{lemma}
\OptionalDetails{ 
\vspace{-1.25em}
\begin{proof}
We prove this first for $\omega \in C^\iny$ compactly supported.

Let $M$ be any test function in $(\Cal{D}(\R^2))^{2 \times 2}$. Then
\begin{align*}
    (\grad u, M)
        &= (Du, M)
        = (\prt_j u^i, M^{ij})
        = - (u^i, \prt_j M^{ij})
        = - (u, \dv M)
        = - (K * \omega, \dv M) \\
        &= - \int \int K(x - y) \omega(y) \, dy \dv M(x) \, dx \\
        &= - \int \int K(x - y) \omega(y) \dv M(x) \, dx \, dy \\
        &= - \int \lim_{\delta \to 0} \int_{B_\delta^C(y)} K(x - y) \dv M(x) \, dx
                \, \omega(y) \, dy \\
        &= \int \lim_{\delta \to 0} \int_{B_\delta^C(y)}
                \grad_x K(x - y) \cdot M(x) \, dx
                \, \omega(y) \, dy \\
        &\qquad
            + \int \lim_{\delta \to 0} \int_{\prt B_\delta(y)}
                (M(x) \n_x) \cdot K(x - y) \, dx \, \omega(y) \, dy
        =: I + II.
\end{align*}
We were able to apply Fubini's theorem to change the order of integration because the integrand lies in $L^1(\R^2 \times \R^2)$.

Consider now the boundary integral. In polar coordinates, 
\begin{align*}
    \n_x
        = (\cos \theta, \sin \theta),
    \qquad
    K(x - y)
        = \frac{\n_x^\perp}{2 \pi \delta}
        = \frac{(-\sin \theta, \cos \theta)}{2 \pi \delta}.
\end{align*}
Now, $\abs{M(x) - M(y)} \abs{K(x - y)} \le C$, so
\begin{align*}
    &\abs{\int \lim_{\delta \to 0} \int_{\prt B_\delta(y)}
                ((M(x) - M(y)) \n_x) \cdot K(x - y) \, dx \, \omega(y) \, dy} \\
        &\qquad
        \le \int \lim_{\delta \to 0} \int_{\prt B_\delta(y)}
                \abs{M(x) - M(y)} \abs{K(x - y)} \, dx \, \abs{\omega(y)} \, dy \\
        &\qquad
        \le C \norm{\omega}_{L^\iny} \int \lim_{\delta \to 0} \int_{\prt B_\delta(y)}
                \, dx \, dy
        = C \norm{\omega}_{L^\iny} \int \lim_{\delta \to 0} \delta \, dy
        = 0,
\end{align*}
so we can replace $M(x)$ with $M(y)$ in \textit{I}. Writing
\begin{align*}
    M(y)
        = \matrix{a & b}{c & d},
\end{align*}
the boundary integral becomes, in the limit as $\delta \to 0$,
\begin{align*}
    \int_0^{2 \pi} &\brac{\matrix{a & b}{c & d} \matrix{\cos \theta}{\sin \theta}}
            \cdot \matrix{-\sin \theta}{\cos \theta} \, d \theta \\
        &= \int_0^{2 \pi} (a \cos \theta + b \sin \theta,
                            c \cos \theta + d \sin \theta)
            \cdot (-\sin \theta, \cos \theta) \, d \theta \\
        &= \int_0^{2 \pi} (-a \cos \theta \sin \theta - b \sin^2 \theta
                            + c \cos^2 \theta + d \sin \theta \cos \theta)
                \, d \theta \\
        &= \frac{c - b}{2}
        = \frac{M^{21}(y) - M^{12}(y)}{2}.
\end{align*}
Thus,
\begin{align*}
    II
        = \int \frac{M^{21}(y) - M^{12}(y)}{2} \omega(y) \, dy
        = \pr{\frac{\omega}{2} \matrix{0 & -1}{1 & 0}, M}.
\end{align*}

Now write $I$ as
\begin{align*}
    I
        = \int \lim_{\delta \to 0} F_\delta(y) \, \omega(y) \, dy,
\end{align*}
where
\begin{align*}
    F_\delta(y)
        = \int_{B_\delta^C(y)} \grad_x K(x - y) \cdot M(x) \, dx.
\end{align*}
Let $R > 0$ be such that $\supp M \subseteq B_R$. If $\delta < 1$ then
\begin{align*}
    F_\delta(y)
        &= \int_{B_\delta^C(y) \cap B_1(y)} \grad_x K(x - y) \cdot M(x) \, dx
            + \int_{B_1^C(y)} \grad_x K(x - y) \cdot M(x) \, dx \\
        &=: I_1 + I_2.
\end{align*}

Because $\grad K$ is a singular integral kernel, each component of it integrates to zero over circles centered at the origin. Also, $\abs{M(x) - M(y)} \le C \abs{x - y}$, so
\begin{align*}
    \abs{I_1}
        &= \abs{\int_{B_\delta^C(y) \cap B_1(y)} \grad_x K(x - y)
            \cdot \pr{M(x) - M(y)} \, dx} \\
        &\le \int_{B_\delta^C(y) \cap B_1(y)} \frac{C}{\abs{x - y}} \, dx
        \le C.
\end{align*}
As for $I_2$,
\begin{align*}
    \abs{I_2}
        \le C \int_{B_1^C} \abs{M(x)} \, dx
        \le C.
\end{align*}
This means that $F_\delta(y)$ is bounded as a function over $\delta < 1$ so we can apply the dominated convergence theorem to conclude that
\begin{align*}
    I
        = \lim_{\delta \to 0} \int F_\delta(y) \, \omega(y) \, dy
        = \lim_{\delta \to 0} \int \int_{B_\delta^C(y)}
                \grad_x K(x - y) \cdot M(x) \, dx
                \, \omega(y) \, dy.
\end{align*}
Since the singularity in $\grad_x K(x - y)$ is removed for each fixed $\delta$, we can apply Fubini's theorem followed by another application of the dominated convergence theorem---using $\omega$ to play the role $M$ played before---to give
\begin{align*}
    I
        &= \int \lim_{\delta \to 0} \int_{B_\delta^C(x)}
                \grad_x K(x - y) \omega(y) \, dy \, M(x) \, dx
        = \pr{\PV \int \grad K(\cdot - y) \omega(y) \, dy, M(\cdot)} \\
        &= \pr{\PV \int \grad K(\cdot - y) \omega(y) \, dy, M(\cdot)},
\end{align*}
since $\grad K$ is symmetric.

This establishes the result for compactly supported vorticity $\omega \in C^\iny$. Now drop the assumption that $\omega$ is compactly supported, and define $\omega_n = a_{1/n} \omega$ and let $u_n = K * \omega_n$. Then $(\omega_n)$ is a sequence of compactly supported vorticities in $C^\iny$ converging to $\omega$ in $L^{1} \cap C^k$ for all $k \ge 0$ with $\omega_n = \omega$ on $B_{1/n}(0)$. From the expression for $\grad u_n(x)$ we just proved, it is easy to see that $\grad u_n$ is Cauchy in many different norms, so that $\grad u_n(x) \to U$, say pointwise, for some $U \in L^\iny$. Moreover,
\begin{align*}
    \norm{u_n - u}_{L^2}
        = \norm{K * \omega_n - K * \omega}_{L^2}
        \le \norm{K}_{L^2(B_{1/n}^C)} \norm{\omega_n - \omega}_{L^1}
        \to 0
\end{align*}
as $n \to \iny$. This shows that $\grad u_n \to \grad u$ in $H^{-1}$. By the uniqueness of limits, it follows that $U = \grad u$ and that the expression for $\grad u(x)$ in the lemma holds.

\end{proof}
} 

%
%
\begin{lemma}\label{L:graduYLikeLemma}
    Let $\omega \in (L^1 \cap C^\iny)(\R^2)$ and let $Y$ be a vector field in $C^\al(\R^2)$.
    Then
    \begin{align*} 
        \PV \int \grad K(x - y) Y(y) \, \omega(y) \, dy
            = - \frac{\omega(x)}{2} \matrix{0 & -1}{1 & 0} Y(x)
                + \brac{K * \dv(\omega Y)}(x),
    \end{align*}
    where $K = \grad^\perp \Cal{F}$ is the Biot-Savart kernel.
    Alternately, if for $j = 1, 2$ we let $i = 2, 1$ then
    \begin{align*}
        \PV \int \grad K^j(x - y) Y(y) \, \omega(y) \, dy
            &= \frac{(-1)^i}{2} \omega_\eps(x) Y_\eps^i(x)
                + \brac{K^j * \dv(\omega_\eps Y_\eps)}(x).
    \end{align*}
\end{lemma}

\OptionalDetails{
\begin{proof}
We assume first that $Y \in C^\iny(\R^2)$.
For $j = 1, 2$ let $i = 2, 1$. Noting that
\begin{align*}
    \brac{\grad K(x - y) Y(y)}^j
        &= (\grad K)^{jk}(x - y) Y^k(y)
        = \prt_k K^j(x - y) Y^k(y)
        = \grad K^j(x - y) \cdot Y(y),    
\end{align*}
integrating by parts gives
\begin{align*}
    &\brac{\PV \int \grad K(x - y) Y(y) \, \omega(y) \, dy}^j
        = \PV \int \grad K^j(x - y) \cdot Y(y) \, \omega(y) \, dy \\
        &\qquad
        = - \PV \int \prt_{y_k} K^j(x - y) Y^k(y) \, \omega(y) \, dy \\
        &\qquad
        = \lim_{\delta \to 0} \int_{B_\delta^C(x)}
             K^j(x - y) \dv (\omega Y)(y) \, dy
        + \lim_{\delta \to 0} \int_{\prt B_\delta(x)}
             K^j(x - y) \omega(y) (Y(y) \cdot \n) \, dy.    
\end{align*}
The area integral converges in the limit to $K^j * \dv(\omega Y)$ since $K$ is locally integrable. For the boundary integral, since $\n = (y - x)/\delta = - (x - y)/\delta$, we can write
\begin{align*}
    K(x - y)
        = \frac{1}{2 \pi \delta^2} (- (x_2 - y_2), x_1 - y_1)
        = \frac{1}{2 \pi \delta} (n_2, -n_1)
\end{align*}
so that $K^j(x - y) = (-1)^i (2 \pi \delta)^{-1} n_i$. Thus, we can express the boundary integral as
\begin{align*}
    I
        &:= (-1)^i \lim_{\delta \to 0} \frac{1}{2 \pi \delta}
            \int_{\prt B_\delta}
             n_i \omega(y) (Y(y) \cdot \n) \, dy.
\end{align*}

Let
\begin{align*}
    II
        &= (-1)^i \lim_{\delta \to 0} \frac{1}{2 \pi \delta}
            \int_{\prt B_\delta}
             n_i \omega(x) (Y(x) \cdot \n) \, dy.
\end{align*}
Writing $Y(x) = Y^1(x) \bm{e}_1 + Y^2(x) \bm{e}_2$, we have
\begin{align*}
    II
        &= (-1)^i \frac{\omega(x)}{2 \pi} \int_0^{2 \pi}
            n_i (Y^1(x) \cos \theta + Y^2(x) \sin \theta) \, d \theta
        = (-1)^i \frac{\omega(x) Y^i(x)}{2} \\
        &= - \frac{\omega(x)}{2} \brac{\matrix{0 & -1}{1 & 0} Y}^j,
\end{align*}
since $n_1 = \cos \theta$ and $n_2 = \sin \theta$. Then,
\begin{align*}
    I - II
        &= (-1)^i \lim_{\delta \to 0} \frac{1}{2 \pi \delta}
            \int_{\prt B_\delta}
             n_j \omega(y) ((Y(y) - Y(x)) \cdot \n \, dy \\
        &\qquad
            + (-1)^i \lim_{\delta \to 0} \frac{1}{2 \pi \delta}
            \int_{\prt B_\delta}
             n_j (\omega(y) - \omega(x)) Y(x) \cdot \n \, dy \\
        &= (-1)^i \lim_{\delta \to 0} \frac{1}{2 \pi \delta^{1 - \al}}
            \int_{\prt B_\delta}
             n_j \omega(y)
             \pr{\frac{Y(y) - Y(x)}{\abs{y - x}^\al}} \cdot \n \, dy \\
        &\qquad
            + (-1)^i \lim_{\delta \to 0} \frac{1}{2 \pi}
            \int_{\prt B_\delta}
             n_j \frac{\omega(y) - \omega(x)}{\abs{y - x}}
                 Y(x) \cdot \n \, dy.
\end{align*}
Therefore,
\begin{align*}
    \abs{I - II}
        \le (2 \pi)^{-1}
            \norm{\omega}_{L^\iny}
            \norm{Y}_{C^\al}
            \lim_{\delta \to 0} \frac{2 \pi \delta}{\delta^{1 - \al}}
                + (2 \pi)^{-1}
                \norm{Y}_{L^\iny}
                \norm{\grad \omega}_{L^\iny}
                \lim_{\delta \to 0}
                (2 \pi \delta)
        = 0
\end{align*}
so that
$
    I
        = II
$.
This leads to
    \begin{align}\label{e:pvKYomegaExp}
        \PV \int \grad K(x - y) Y(y) \, \omega(y) \, dy
            = - \frac{\omega(x)}{2} \matrix{0 & -1}{1 & 0} Y(x)
                + \brac{K * \dv(\omega Y)}(x),
    \end{align}
for $K * \dv(\omega Y)$ when $Y \in C^\iny$.

Now let $(Y_n)$ be a sequence in $C^\iny$ converging to $Y$ in $C^\al$. Fix $p \in [1, \iny)$. The lefthand side of \cref{e:pvKYomegaExp} converges in $L^p$ since it is a singular integral operator applied to $\omega Y_n$, which converges in $L^p$ to $\omega Y$. The first term on the righthand side converges in the same manner. And
\begin{align*}
    Y_n \to Y \text{ in } C^\al
        &\implies \omega Y_n \to \omega Y \text{ in } C^\al
        \implies \dv(\omega Y_n) \to \dv(\omega Y) \text{ in } C^{\al - 1} \\
        &
        \implies K * \dv(\omega Y_n) \to K * \dv(\omega Y) \text{ in } C^\al
\end{align*}
by \cref{L:EquivalentConditions}. Hence, all three terms converge in $L^p_{loc}(\R^2)$, giving \cref{e:pvKYomegaExp} for $Y \in C^\al$.
\end{proof}
} 

The following is a corollary of \cref{L:graduExp,L:graduYLikeLemma}.
\begin{cor}\label{C:graduCor}
    Let $\omega \in (L^1 \cap C^\iny)(\R^2)$ and let $Y$ be a vector field in $C^\al(\R^2)$.
    Then
    \begin{align*}
        Y(x) \cdot \grad u(x)
            &= \PV \int \grad K(x - y) \brac{Y(x) - Y(y)} \omega(y) \, dy
                + \brac{K * \dv(\omega Y)}(x),
    \end{align*}
    where $K = \grad^\perp \Cal{F}$ is the Biot-Savart kernel. Moreover,
    \begin{align*}
        \norm{\PV \int \grad K(x - y) \brac{Y(x) - Y(y)} \omega(y) \, dy}_{C^\al}
            \le C V(\omega) \norm{Y}_{\dot{C}^\al},
    \end{align*}
    $V(\omega)$ being given in \cref{e:Vomega}.
\end{cor}
\begin{proof}
The expression for $Y(x) \cdot \grad u(x)$ follows from comparing the expressions in \cref{L:graduExp,L:graduYLikeLemma}. The $C^\al$-bound follows from applying \cref{L:SerfatiLemma2Inf} with the kernel $L_3$ of \cref{L:SerfatiKernels}.
\end{proof}

\begin{lemma}\label{L:ConvEq}
    Let $f \in C^\beta(\R^2)$ for $\beta > 0$ be such that $\grad (a_r K) \starp f(x)$
    is defined for some $r > 0$, $x \in \R^2$. Then
    \begin{align*}
        \grad (a_r K) \starp f(x)
            = \lim_{h \to 0} \grad (\mu_{rh} K) \starp f(x).
    \end{align*}
\end{lemma}
\OptionalDetails{ 
\vspace{-1.25em}
\begin{proof}
    Referring to \cref{R:Radial}, 
    for all sufficiently small $h > 0$, $\mu_{rh} = a_r$ on $B_{2h}^C$.
    Hence, we need only show that
    \begin{align}\label{e:ConvEqMustShow}
        \lim_{h \to 0} \int_{B_{2h}}
            \grad \brac{(a_r - \mu_{rh}) K}(x - y) f(y) \, dy
                = 0.
    \end{align}
    Now, on $B_{2h}$,
    \begin{align*}
        \grad \brac{(a_r - \mu_{rh}) K}
            = (a_r - \mu_{rh}) \grad K + (\grad a_r - \grad \mu_{rh}) \otimes K
            = (1 - \mu_{rh}) \grad K - \grad \mu_{rh} \otimes K
    \end{align*}
    so
    \begin{align*}
        \abs{\grad \brac{(a_r - \mu_{rh}) K}(z)}
            \le \abs{\grad K(z)} + \frac{C}{h} \abs{K(z)}
            \le \frac{1}{2 \pi \abs{z}^2} + \frac{C}{h \abs{z}}
            \le \frac{C}{\abs{z}^2}.
    \end{align*}
    Also,
    \begin{align*}
        \int_{B_{2h}}
            \grad \brac{(a_r - \mu_{rh}) K}(x - y) f(x) \, dy
                = 0.
    \end{align*}
    
    Thus,
    \begin{align*}
        &\abs{\int_{B_{2h}}
                \grad \brac{(a_r - \mu_{rh}) K}(x - y) f(y) \, dy} \\
            &\qquad\qquad
            = \abs{\int_{B_{2h}} 
                \grad \brac{(a_r - \mu_{rh}) K}(x - y) (f(y) - f(x)) \, dy
                } \\
            &\qquad\qquad
            \le \int_0^{2h} \frac{C}{\rho^2} \norm{f}_{C^\beta} \rho^\beta \rho \, d \rho
            \le C h^\beta
    \end{align*}
    so that \cref{e:ConvEqMustShow} holds.
\end{proof}
} 

%
%
\section{On transport equation estimates}\label{A:TransportEstimates}

\noindent Together, \cref{L:f0,L:ForY0} justify our use of strong transport equations in obtaining estimates in the $C^\al$-norm of the transported and pushed-forward quantities. First, the initial data is mollified using a mollification parameter $\delta$ independent of $\eps$, the strong transport equation estimates are made, then $\delta$ is taken to zero. This is all while $\eps$ is held fixed. \cref{L:f0} is used to obtain the $C^\al$-bound on $\dv Y_\eps(t)$ (leading to \cref{e:YdivBound}), while \cref{L:ForY0} is used to obtain the $C^\al$-bounds on the vector fields, $Y_\eps(t)$, $R_\eps$(t), and $Y_\eps \cdot \grad u_\eps(t)$.

The proofs of \cref{L:f0,L:ForY0}, which are left to the reader, employ only \cref{e:CdotIneq}$_{1, 2}$, the boundedness of $\grad \eta_\eps^{-1}(t)$ in $L^\iny$ over time (for fixed $\eps$), and the convergence in $C^\al$ of a mollified function to the function itself. 

\noindent \begin{lemma}\label{L:f0}
    For $f_0 \in C^\al$, let
    \begin{align*}
        f(t, x)
            &:= f_0(\eta_\eps^{-1}(t, x)), \\
        f^{(\delta)}(t, x)
            &= (\rho_\delta * f_0)(\eta_\eps^{-1}(t, x))
    \end{align*}
    for $\delta > 0$.
    Then
    \begin{align*}
        \smallnorm{f^{(\delta)} - f}_{L^\iny([0, T]; C^\al)}
            \to 0 \text{ as } \delta \to 0.
    \end{align*}
    \Ignore{ 
    Similarly, for $f_0 \in C^{\al - 1}$
    \begin{align*}
        \smallnorm{f^{(\delta)} - f}_{L^\iny([0, T]; C^{\al - 1}}
            \to 0 \text{ as } \delta \to 0.
    \end{align*}
    } 
\end{lemma}
\OptionalDetails{
\vspace{-1.25em}
\begin{proof}
    We have,
    \begin{align*}
        \smallnorm{(f^{(\delta)} - f)(t)}_{\dot{C}^\al}
            &= \norm{(\rho_\delta * f_0 - f_0) \circ \eta_\eps^{-1}(t)}_{\dot{C}^\al}
            \le \norm{\rho_\delta * f_0 - f_0}_{\dot{C}^\al}
                \norm{\eta_\eps^{-1}(t)}_{L^\iny}^\al \\
            &\le \norm{\rho_\delta * f_0 - f_0}_{\dot{C}^\al}
                e^{\al \int_0^t V_\eps(s) \, ds}.
    \end{align*}
    And, more simply,
    \begin{align*}
        \smallnorm{(f^{(\delta)} - f)(t)}_{L^\iny}
            &= \norm{(\rho_\delta * f_0 - f_0) \circ \eta_\eps^{-1}(t)}_{L^\iny}
            = \norm{\rho_\delta * f_0 - f_0}_{L^\iny}    
    \end{align*}
    so
    \begin{align*}
        \smallnorm{(f^{(\delta)} - f)(t)}_{C^\al}
            &\le \norm{\rho_\delta * f_0 - f_0}_{C^\al}
                e^{\al \int_0^t V_\eps(s) \, ds},
    \end{align*}
    from which $\smallnorm{f^{(\delta)} - f}_{L^\iny([0, T]; C^\al)} \to 0$ follows.
    (Note that we do not have a uniform bound on $V_\eps$ in $\eps$,
    but all we need is that it is finite for all $\eps > 0$.)
    \Ignore{ 
    We need to work a little harder for the $C^{\al - 1}$ convergence.
    Since $f_0 \in C^{\al - 1}$, $f_0 = g_0 + \dv h_0$ for some
    $g_0, h_0 \in C^\al$. Then
    \begin{align*}
        \rho_\delta * f_0
            = \rho_\delta * g_0 + \dv(\rho_\delta * h_0)
    \end{align*}
    so
    \begin{align*}
        &\smallnorm{(f^{(\delta)} - f)(t)}_{C^{\al - 1}}
            = \norm{(\rho_\delta * f_0 - f_0) \circ \eta_\eps^{-1}(t)}_{C^{\al - 1}} \\
            &\qquad
            \le \norm{(\rho_\delta * g_0 - g_0) \circ \eta_\eps^{-1}(t)}_{C^{\al - 1}}
                + \norm{(\dv(\rho_\delta * h_0) - \dv h_0)
                    \circ \eta_\eps^{-1}(t)}_{C^{\al - 1}} \\
            &\qquad
            \le \norm{(\rho_\delta * g_0 - g_0) \circ \eta_\eps^{-1}(t)}_{C^\al}
                + \norm{(\dv(\rho_\delta * h_0) - \dv h_0)
                    \circ \eta_\eps^{-1}(t)}_{C^{\al - 1}}
    \end{align*}
    } 
\end{proof}
} 

\begin{lemma}\label{L:ForY0}
    Let $Y_\eps$ be as in \cref{e:Yeps}, so that
    \begin{align*}
        Y_\eps(t, \eta_\eps(t, x))
            = Y_0(x) \cdot \grad \eta_\eps(t, x).
    \end{align*}    
    Define $Y_\eps^{(\delta)}$ by
    \begin{align*}
        Y_\eps^{(\delta)}(t, \eta_\eps(t, x))
            = (\rho_\delta * Y_0)(x) \cdot \grad \eta_\eps(t, x).
    \end{align*}
    Then
    \begin{align*}
        \smallnorm{Y_\eps^{(\delta)} - Y_\eps}_{L^\iny([0, T]; C^\al)}
            \to 0 \text{ as } \delta \to 0.
    \end{align*}
\end{lemma}
\OptionalDetails{
\vspace{-1.25em}
\begin{proof}
    We have
    \begin{align*}
        Y_\eps^{(\delta)}(t, x)
            = (\rho_\delta * Y_0)(\eta_\eps^{-1}(t, x))
                \cdot \grad \eta_\eps(t, \eta_\eps^{-1}(t, x))
    \end{align*}
    Hence,
    \begin{align*}
        \smallnorm{(Y_\eps^{(\delta)} - Y_\eps)(t)}_{\dot{C}^\al}
            &\le \norm{(\rho_\delta * Y_0 - Y_0)}_{\dot{C}^\al}
                \norm{\grad \eta_\eps^{-1}(t)}_{\dot{C}^\al}
                \norm{\grad \eta_\eps^{-1}(t)}_{L^\iny}^{2\al}
    \end{align*}
    and
    \begin{align*}
        \smallnorm{(Y_\eps^{(\delta)} - Y_\eps)(t)}_{L^\iny}
            \le \norm{(\rho_\delta * Y_0 - Y_0)}_{L^\iny}
                \norm{\grad \eta_\eps^{-1}(t)}_{L^\iny}
    \end{align*}
    so that
    \begin{align*}
        \smallnorm{(Y_\eps^{(\delta)} - Y_\eps)(t)}_{C^\al}
            &\le \norm{(\rho_\delta * Y_0 - Y_0)}_{C^\al}
                \norm{\grad \eta_\eps^{-1}(t)}_{\dot{C}^\al}
                \norm{\grad \eta_\eps^{-1}(t)}_{L^\iny}^{2 \al}.
    \end{align*}
    This vanishes as $\delta \to 0$ uniformly in time,
    since $\grad \eta_\eps^{-1}$ is
    smooth in time and space, though we do not have a concrete bound
    on its rate.
\end{proof}
} 

\section*{Acknowledgements}
\noindent H.B. is supported by the 2014 Research Fund(Project Number 1.140076.01) of UNIST(Ulsan National Institute of Science and Technology). H.B. gratefully acknowledges the support by the Department of Mathematics at UC Davis where part of this research was performed. H.B. would like to thank the Department of Mathematics at UC Riverside for its kind hospitality where part of this work was completed.

J.P.K. gratefully acknowledges NSF grants DMS-1212141 and DMS-1009545, and thanks Instituto Nacional de Matem\'{a}tica Pura e Aplicada in Rio de Janeiro, Brazil, at which a portion of this research was performed.

\ifbool{ForSubmission}{}
{
	\boolfalse{InlineBib}
}

\ifbool{InlineBib}
{
\def\cprime{$'$} \def\polhk#1{\setbox0=\hbox{#1}{\ooalign{\hidewidth
  \lower1.5ex\hbox{`}\hidewidth\crcr\unhbox0}}}

} 
{
\bibliography{Refs}
\bibliographystyle{plain}
} 

\end{document}